\documentclass[reqno]{amsart}
\pdfoutput=1

\usepackage[british]{babel}
\usepackage[utf8]{inputenc}
\usepackage[T1]{fontenc}

\usepackage[breaklinks=true]{hyperref}
\usepackage{amssymb}
\usepackage{amsfonts}
\usepackage{amsmath}
\usepackage{amsthm}
\usepackage{enumitem}
\usepackage{mathtools}
\usepackage{csquotes}
\usepackage{tikz-cd}
\usepackage{adjustbox}
\usepackage{accents}
\usepackage{MnSymbol}
\usepackage{soul}
\usepackage{stmaryrd}
\PassOptionsToPackage{hyphens}{url}\usepackage{hyperref}
\usepackage{xurl}
\usepackage{cleveref}
\usepackage{dsfont}
\usepackage{pdfpages}
\usepackage{bbm}
\usepackage{graphicx}
\usepackage{comment}
\usepackage{xcolor}
\usepackage{etoolbox}

\makeatletter
\patchcmd{\@setaddresses}{\indent}{\noindent}{}{}
\patchcmd{\@setaddresses}{\indent}{\noindent}{}{}
\patchcmd{\@setaddresses}{\indent}{\noindent}{}{}
\patchcmd{\@setaddresses}{\indent}{\noindent}{}{}
\makeatother

\numberwithin{equation}{section}
\newtheorem{thm}[equation]{Theorem}
\newtheorem{lem}[equation]{Lemma}
\newtheorem{prop}[equation]{Proposition}
\newtheorem{cor}[equation]{Corollary}

\theoremstyle{definition}
\newtheorem{defi}[equation]{Definition}
\newtheorem{ex}[equation]{Example}
\newtheorem{re}[equation]{Remark}
\newtheorem{rec}[equation]{Recollection}
\newtheorem{cons}[equation]{Construction}

\newtheorem{no}[equation]{Notation}

\newcommand{\nD}{\tn{D}}

\newcommand{\nH}{\tn{H}}

\newcommand{\nK}{\tn{K}}

\newcommand{\nN}{\tn{N}}

\newcommand{\bA}{\mathbb{A}}

\newcommand{\bF}{\mathbb{F}}

\newcommand{\bP}{\mathbb{P}}

\newcommand{\bR}{\mathbb{R}}
\newcommand{\bS}{\mathbb{S}}

\newcommand{\bZ}{\mathbb{Z}}

\newcommand{\cA}{\mathcal{A}}

\newcommand{\cC}{\mathcal{C}}
\newcommand{\cD}{\mathcal{D}}

\newcommand{\cF}{\mathcal{F}}
\newcommand{\cG}{\mathcal{G}}
\newcommand{\cH}{\mathcal{H}}
\newcommand{\cI}{\mathcal{I}}

\newcommand{\cK}{\mathcal{K}}

\newcommand{\cM}{\mathcal{M}}
\newcommand{\cN}{\mathcal{N}}

\newcommand{\cP}{\mathcal{P}}

\newcommand{\tn}[1]{\textnormal{#1}}
\newcommand{\bb}[1]{\mathbb{#1}}
\newcommand{\tsf}[1]{\textsf{#1}}
\newcommand{\tbf}[1]{\textbf{#1}}
\newcommand{\mcal}[1]{\mathcal{#1}}

\newcommand{\underl}[1]{\underline{#1}}
\newcommand{\overl}[1]{\overline{#1}}

\newcommand{\colimit}[1]{\underset{#1}{\tn{colim}} \;}

\newcommand*{\defeq}{\mathrel{\vcenter{\baselineskip0.4ex \lineskiplimit0pt
\hbox{\scriptsize.}\hbox{\scriptsize.}}}%
=}
\newcommand{\op}{\tn{op}}
\newcommand{\fun}{\tn{Fun}}
\newcommand{\Hom}{\tn{Hom}}
\newcommand{\funadd}{\tn{Fun}_{\tn{add}}}
\newcommand{\funrlin}{\tn{Fun}_{R\tn{-lin}}}
\newcommand{\map}{\tn{map}}
\newcommand{\Map}{\tn{Map}}
\newcommand{\psigma}{\cP_{\Sigma}}
\newcommand{\alg}{\tn{Alg}}
\newcommand{\calg}{\tn{CAlg}}
\newcommand{\bbone}{\mathbbm{1}}

\newcommand{\modr}{\tsf{\textup{Mod}}(R)}
\newcommand{\modgr}{\tsf{\textup{Mod}}(G;R)}
\newcommand{\Mod}{\tsf{\textup{Mod}}}
\newcommand{\heart}{\ensuremath\heartsuit}
\newcommand{\ch}{\tn{Ch}}
\newcommand{\finstar}{\tsf{\textup{Fin}}_{\ast}}
\newcommand{\infoperads}{\tsf{\textup{Op}}_{\infty}}

\newcommand{\Proj}{\tn{Proj}}
\newcommand{\proj}{\tn{proj}}
\newcommand{\spans}{\tsf{\textup{span}}}
\newcommand{\Spans}{\tsf{\textup{Span}}}
\newcommand{\spec}{\tsf{\textup{Sp}}}
\newcommand{\specg}{\tsf{\textup{Sp}}^G}
\newcommand{\specwgh}{\tsf{\textup{Sp}}^{\wgh}}
\newcommand{\spc}{\tsf{\textup{Spc}}}

\newcommand{\spcg}{\tsf{\textup{Spc}}^G}
\newcommand{\spcgp}{\tsf{\textup{Spc}}^G_{\ast}}
\newcommand{\catinf}{\tsf{\textup{Cat}}_{\infty}}
\newcommand{\catinfex}{\tsf{\textup{Cat}}_{\infty}^{\tn{ex}}}

\newcommand{\prl}{\tsf{\textup{Pr}}^{\tn{L}}}
\newcommand{\prlst}{\tsf{\textup{Pr}}^{\tn{L}}_{\tn{st}}}

\newcommand{\prrst}{\tsf{Pr}^{\tn{R}}_{\tn{st}}}

\newcommand{\mackg}{\tsf{\textup{Mack}}(G)}
\newcommand{\mackgr}{\tsf{\textup{Mack}}_R(G)}
\newcommand{\mack}{\tsf{\textup{Mack}}}
\newcommand{\cmackg}{\tsf{\textup{Mack}}^{\tsf{\textup{coh}}}(G)}
\newcommand{\cmackgr}{\tsf{\textup{Mack}}_R^{\tsf{\textup{coh}}}(G)}

\newcommand{\perm}{\tsf{\textup{perm}}}
\newcommand{\permgr}{\tsf{\textup{perm}}(G;R)}
\newcommand{\Perm}{\tsf{\textup{Perm}}}
\newcommand{\Permgr}{\tsf{\textup{Perm}}({G};{R})}
\newcommand{\gset}{G\tsf{\textup{-set}}}
\newcommand{\gSet}{G\tsf{\textup{-Set}}}
\newcommand{\freer}{\tn{free}_R}

\newcommand{\weyl}[2]{{#1}/\!\!/{#2}}
\newcommand{\wgh}{\weyl{G}{H}}
\newcommand{\pic}{\tn{Pic}}
\newcommand{\stab}{\tn{Stab}}

\newcommand{\sylp}{\tn{Syl}_p}
\newcommand{\subp}{\tn{Sub}_p}
\newcommand{\sub}{\tn{Sub}}
\newcommand{\glo}{\tsf{Glo}}
\newcommand{\gppairs}{\tsf{G}(p)}

\newsavebox{\pullback}
\sbox\pullback{%
\begin{tikzpicture}%
\draw (0,0) -- (1ex,0ex);%
\draw (1ex,0ex) -- (1ex,1ex);%
\end{tikzpicture}}

\tikzcdset{scale cd/.style={every label/.append style={scale=#1}, cells={nodes={scale=#1}}}}

\date{}
\author{Yorick Fuhrmann}
\address{Yorick Fuhrmann, Warwick Mathematics Institute, Coventry CV4 7AL, UK}
\email{yorick.fuhrmann@warwick.ac.uk}
\urladdr{https://warwick.ac.uk/fac/sci/maths/people/staff/fuhrmann}

\hypersetup{{pdfauthor={Yorick Fuhrmann}},pdftitle={Modular fixed points in equivariant homotopy theory}}

\usepackage[backend=biber, citestyle=authoryear, style=alphabetic, maxbibnames=99]{biblatex}
\AtBeginBibliography{\footnotesize}

\begin{document}

\title[Modular fixed points in equivariant homotopy theory]{Modular fixed points in \\ equivariant homotopy theory}

\begin{abstract}
We show that the derived $\infty$-category of permutation modules is equivalent to the category of modules over the Eilenberg-MacLane spectrum associated to a constant Mackey functor in the $\infty$-category of equivariant spectra. On such module categories we define a modular fixed point functor using geometric fixed points followed by an extension of scalars and identify it with the modular fixed point functor on derived permutation modules introduced by Balmer-Gallauer in \cite{bg25a}. As an application, we show that the Picard group of such a module category for a $p$-group is given by the group of class functions satisfying the Borel-Smith conditions. In the language of representation theory, this result was first obtained by Miller in \cite{mil25b}.
\end{abstract}


\maketitle
\tableofcontents

\section{Introduction}\label{sec:intro}

Geometric fixed points extend the fixed points of equivariant spaces to the level of equivariant spectra. That is, if $H$ is a subgroup of a finite group $G$ with Weyl group $\wgh = N_G(H)/H$, then by the universal property of the $\infty$-category of $G$-spectra the geometric $H$-fixed point functor 
$$\Phi^H: \specg \to \specwgh$$
is the essentially unique symmetric monoidal left adjoint $\specg \to \specwgh$ which makes the square
\[\begin{tikzcd}
	\spcgp & \spc^{\wgh}_{\ast} \\
	{\spec^{G}} & {\specwgh}
	\arrow["{(-)^H}", from=1-1, to=1-2]
	\arrow["{\Sigma^{\infty}}"', from=1-1, to=2-1]
	\arrow["{\Sigma^{\infty}}", from=1-2, to=2-2]
	\arrow["{\Phi^H}", from=2-1, to=2-2]
\end{tikzcd}\]
commute, where $\Sigma^{\infty}: \spc^{(-)}_* \to \spec^{(-)}$ is the infinite suspension functor from pointed equivariant spaces to equivariant spectra. We hence have the property
$$\Phi^H (\Sigma^{\infty} X) \simeq \Sigma^{\infty} (X^H)$$
for every pointed $G$-space $X$. From this there arises a natural question: If $R$ is a discrete commutative ring, is there an $R$-linear version of this property? In the realm of equivariant homotopy theory there are multiple sensible ways of promoting a discrete ring to an equivariant spectrum. For one of those, namely the Eilenberg-MacLane spectrum of the constant Green functor $\underl{R}$ associated to $R$, we answer this question in the affirmative by constructing a symmetric monoidal left adjoint $\nH R$-linear \emph{modular fixed point functor} 
$$\Psi^H: \Mod_{\nH \underl{R}}(\specg) \to \Mod_{\nH \underl{R}}(\specwgh)$$
which exists when we pass to the $\underl{R}$-linear version of the $\infty$-category of equivariant spectra for a finite group $G$ and the Weyl group $\wgh$ of a $p$-subgroup $H \leq G$, where we assume that $p=0$ in the ring $R$ (\Cref{defi:eqmodularfixedpoints}). This functor can be seen as an $R$-linearisation of geometric fixed points, since it makes the diagram
\[\begin{tikzcd}
	\specg & \specwgh \\
	{\Mod_{\nH \underl{R}}(\spec^{G})} & {\Mod_{\nH \underl{R}}(\specwgh)}
	\arrow["{\Phi^H}", from=1-1, to=1-2]
	\arrow["{F_{\nH \underl{R}}}"', from=1-1, to=2-1]
	\arrow["{F_{\nH \underl{R}}}", from=1-2, to=2-2]
	\arrow["{\Psi^H}", from=2-1, to=2-2]
\end{tikzcd}\]
with vertical free functors commute, see \Cref{lem:eqmodfixfree}. In particular, it satisfies the analogous property 
$$\Psi^H(\nH \underl{R} \otimes \Sigma^{\infty}X) \simeq \nH \underl{R} \otimes \Sigma^{\infty} (X^H)$$
for any pointed $G$-space $X$. Once defined for finite groups, the functor naturally extends to the case where $G$ is profinite and $H$ is a pro-$p$-subgroup. \\

\noindent Let us shine some light on the $\infty$-category $\Mod_{\nH \underl{R}}(\specg)$ by relating it to other ways of promoting the discrete ring $R$ to $\specg$. That is, it receives symmetric monoidal left adjoint functors from two other $G$-equivariant module $\infty$-categories:
\begin{enumerate}
    \item[(1)] Modules over the $R$-linear Burnside Mackey functor $\Mod_{\nH \bA_R}(\specg)$, which by \cite[Theorem 5.10]{psw22} are equivalent to the derived $\infty$-category of $R$-module valued Mackey functors $\cD(\mackgr)$. Below we observe that this category admits a spectral Mackey functor description.
    \item[(2)] Modules over the inflated coefficients $\nH R_G \defeq \tn{infl}^G_{1}(\nH R)$, $\Mod_{\nH R_G}(\specg)$, where $\tn{infl}^G_{1}: \spec \to \specg$ is the inflation functor from the trivial group to $G$. By \cite[Theorem 4.50]{psw22} we can view this category as the $\infty$-categorical version of Kaledin's category of derived Mackey functors $\nD\mackgr$ from \cite{kal11}, and with $\fun^{\times}(\Spans(G)^{\op},\Mod_{\nH R}(\spec))$ it also has a description in terms of spectral Mackey functors, see \cite[Proposition 4.41]{psw22}.
\end{enumerate}

As the name modular fixed points and the presence of constant Mackey functors might suggest, this construction is related to Brauer quotients \cite[Section 1]{bro85} and their extension to derived permutation modules: Motivated by the study of the tt-geometry of the derived category of permutation modules $\nD\Permgr$ defined in \cite{bg22a} (or equivalently, the derived category of cohomological Mackey functors or the derived category of Artin motives over a field,\footnote{In a subsequent paper we will extend this equivalence to more general base schemes.} see \cite{bg23}), Balmer and Gallauer in \cite{bg25a} under similar assumptions constructed a `modular fixed point' tt-functor
$$\Psi^H: \nD\Permgr \to \nD\Perm(\wgh;R)$$
which linearises $H$-fixed points of $G$-sets, i.e. it maps the permutation module $R(X)$ for a finite $G$-set $X$ to the permutation module $R(X^H)$ and hence recovers the classical Brauer quotient on permutation modules. These fixed point functors were crucial to stratify $\nD\Permgr$, see \cite[Section 9]{bg25a}. \\
We show that the $\infty$-categorical enhancement of $\nD\Permgr$ is equivalent to the module category $\Mod_{\nH \underl{R}}(\specg)$, a result which was already mentioned but not proven in \cite{bg23}. We do so in two different ways:
\begin{enumerate}
    \item[(1)] One goes through \emph{cohomological} Mackey functors $\cmackgr$, which form the category of modules over the constant Mackey functor $\underl{R}$ in abelian group or $R$-module valued Mackey functors. Such Mackey functors are characterised by the property that induction after restriction is equal to multiplication by the index. In \Cref{cor:comackasspgmodules} we show that $\Mod_{\nH \underl{R}}(\specg)$ is equivalent to the derived $\infty$-category of $\cmackgr$. \\
    The second required equivalence $\cD(\cmackgr) \simeq \cD\Permgr$ builds on an observation of Yoshida that cohomological Mackey functors are nothing but $R$-module valued presheaves on permutation modules, see \Cref{thm:eqdperminf}.
    \item[(2)] For the other arguably more direct way we construct a bounded weight structure on the compact part of $\Mod_{\nH \underl{R}}(\specg)$ which has weight heart $\permgr^{\natural}$, the idempotent completion of finitely generated permutation modules. This weight structure is of independent interest, since it allows us to translate statements about permutation modules into the world of genuine equivariance. Using the higher categorical machinery for weight structures developed by Sosnilo we obtain an equivalence of compact parts $\Mod_{\nH \underl{R}}^{\omega}(\specg) \simeq \cK_{\tn{b}}(\permgr^{\natural})$, see \Cref{thm:eqhrmoddpermweight}.
\end{enumerate}
Using weight structure arguments we then show in \Cref{thm:identifymodfix} that the equivariant modular fixed points identify with the modular fixed points on derived permutation modules by Balmer-Gallauer. \\   

\noindent The relationship between categories of modules in Mackey functors and $\infty$-categories of modules over equivariant Eilenberg-MacLane spectra turns out to be a robust one, which has already been observed model-categorically \cite{gs14}, \cite{zen18}. We use a description of the derived $\infty$-category of a Grothendieck abelian category with a system of compact projective generators as finite product preserving presheaves of spectra to show in \Cref{thm:modulesandderivation} that taking modules commutes with $\infty$-categorical derivation, a result which is probably well known but we couldn't find a reference for. This implies, relying on such a result for the Burnside Mackey functor \cite[Theorem 5.10]{psw22}, that for any Green functor $B \in \mackgr$ there is an equivalence
$$\cD(\Mod_B(\mackgr)) \simeq \Mod_{\nH B}(\specg),$$
see \Cref{thm:greenmodulesasspgmodules}. The presheaf description of $\cD(\Mod_B(\mackgr))$ we use gives a further equivalence
$$\Mod_{\nH B}(\specg) \simeq \fun^{\times}(B[\gset]^\op,\spec),$$
where $B[\gset] \subseteq \Mod_B(\mackgr)$ is the full subcategory on the $B$-modules $B(X \times -)$ for all finite $G$-sets $X$. Hence $\Mod_{\nH B}(\specg)$ is the initial presentable stable $\infty$-category receiving a finite direct sum preserving functor from $B[\gset]$, see \Cref{cor:modhbaspresheaves}. For specific Green functors this gives spectral Mackey functor descriptions of their module categories in equivariant spectra. The case we are mostly concerned with is the constant Mackey functor $\underl{R}$, where we get
$$\Mod_{\nH \underl{R}}(\specg) \simeq \fun^\times(\permgr^\op,\spec).$$ 
Here $\permgr$ again is the category of finitely generated $(G;R)$-permutation modules. This fact already appeared in \cite[Example 2.15]{bcn25}. For the $R$-linear Burnside Mackey functor $\bA_R$ we obtain
$$\Mod_{\nH \bA_R}(\specg) \simeq \fun^\times(\Omega_R(G)^\op,\spec),$$
where $\Omega_R(G)$ is the $R$-linear Burnside category. If we choose $R=\bZ$, the right-hand side simplifies to $\fun^\times(\spans(G),\spec)$, where $\spans(G)$ is the effective Burnside (1-)category. It is the homotopy category of $\Spans(G)$, Barwick's effective Burnside $\infty$-category of finite $G$-sets \cite[Section 3]{bar17}. Using the spectral Mackey functor description of $G$-spectra 
$$\specg \simeq \fun^\times(\Spans(G)^\op,\spec)$$
this identifies the restriction along the canonical map $\Spans(G) \to \spans(G)$ with the forgetful functor $\Mod_{\nH \bA}(\specg) \to \specg$. The free functor is given by left Kan extension. In particular, this means that the only `difference' between $\specg$ and $\Mod_{\nH \bA}(\specg)$ is the 2-morphisms in the (effective) Burnside category. For the trivial group $G=1$ we also recover the classical equivalence
$$\Mod_{\nH R}(\spec) \simeq \fun^\times(\tsf{Lat}_R^\op,\spec),$$
where $\tsf{Lat}_R$ is the category of $R$-modules that are free of finite rank, viewed as a full subcategory of all $R$-modules, cf. \cite[Example 5.1.6(2)]{cs24}. \\

\noindent We finally use the equivariant modular fixed point functors to construct a group isomorphism
$$\pic(\Mod_{\nH \underl{k}}(\spec^{G})) \cong \tn{CF}_{\tn{b}}(G)$$
for $G$ a $p$-group and $k$ a field of characteristic $p$, where $\tn{CF}_{\tn{b}}(G)$ is the group of integral class functions on conjugacy classes of subgroups of $G$ which satisfy the \emph{Borel-Smith conditions}, see \Cref{thm:picborelsmith}. This has been shown by Miller in a series of papers \cite{mil24}, \cite{mil25a}, \cite{mil25b} using the language of representation theory. Miller considers the invertible elements of $\nK_{\tn{b}}(\perm(G;k)^{\natural})$, which are called \emph{endotrivial complexes}. For Artin motives over a field, this Picard group has been considered by Bachmann \cite{bac16}, who also constructed modular fixed point functors in this geometric setting. \\
Our methods are more topological. On the one hand, they rely on work of Smith and Borel on fixed points of finite equivariant CW complexes which are $\bF_p$-cohomology spheres, as presented in \cite[Section III.4]{td87}. On the other hand, they use work of tom Dieck on dimension functions of orthogonal representations \cite[Section III.5]{td87}, which in turn builds on work of Atiyah-Tall and Serre on Adams operations on real orthogonal representation rings. From the above isomorphism for $p$-groups it can be deduced that for any finite group $G$ there is an isomorphism 
$\pic(\Mod_{\nH \underl{k}}(\specg)) \cong \tn{CF}_{\tn{b}}(G,p) \times \Hom(G,k^{\times}),$
where $\tn{CF}_{\tn{b}}(G,p)$ stands for the group of integral class functions on conjugacy classes of $p$-subgroups of $G$ which satisfy the Borel-Smith conditions, but it requires some nontrivial inputs as explained in \cite[Section 6]{mil25b}. In \Cref{cor:picborelsmithprop} we finally extend the isomorphism $\pic(\Mod_{\nH \underl{k}}(\spec^{G})) \cong \tn{CF}_{\tn{b}}(G)$ to the case where $G$ is a pro-$p$-group. \\

\noindent \tbf{Organisation of the paper.} We start by recalling facts about Mackey functors and permutation modules in \Cref{sec:mackey}. We then consider derived $\infty$-categories of modules in Grothendieck abelian categories, using a spectral presheaf description. In \Cref{sec:modules} we are concerned with modules in equivariant spectra and their relationship to derived $\infty$-categories of Mackey functors. We then introduce the derived $\infty$-category of permutation modules in \Cref{sec:derivedpermutationmodules} and compare it to modules over the constant Mackey functor in equivariant spectra. \Cref{sec:weightstructures} is devoted to weight structures and the construction of one on modules over the constant Mackey functor in equivariant spectra. \Cref{sec:fixedpoints} reviews geometric fixed points and constructs modular fixed points in the equivariant setting. We prove properties of the latter and compare it to the modular fixed points on derived permutation modules. We end by considering the Picard group of modules over the constant Mackey functor in \Cref{sec:picardgroups} by employing the constructed functors. \\

\noindent \tbf{Notations and conventions.} Throughout, $G$ is a profinite group, and we mention when we put further assumptions on $G$. If not mentioned otherwise the notation $H \leq G$ and the term \emph{subgroup} will always mean an open subgroup, or equivalently, a closed subgroup of finite index. If $H \leq G$ is a subgroup we denote by $\wgh \defeq N_G(H)/H$ its Weyl group. If $K,H \leq G$ are subgroups we write $K \leq_G H$ if $K$ is subconjugate to $H$ in $G$. When $G$ is finite and $p$ is a prime we write $\sub(G)$ for the set of subgroups, $\subp(G)$ for the set of $p$-subgroups and $\sylp(G)$ for the set of Sylow $p$-subgroups of $G$. By $R$ we denote a commutative unital ring, and $\underl{R}$ is the constant Mackey functor associated to $R$ which has identic restriction maps and whose induction maps multiply with the index. All $R$-linear categories are implicitly assumed additive. The term \emph{tensor category} will mean an additive category with a symmetric monoidal structure, additive in both variables, and an \emph{$R$-linear tensor category} is $R$-linear and a tensor category such that the tensor product is $R$-linear in both variables. A functor between ($R$-linear) tensor categories which is \emph{tensor} is an additive (resp. $R$-linear) and symmetric monoidal functor. If we write a pair of adjoint functors as $(F,G)$, then $F$ will be left and $G$ right adjoint. \\
We write $\spc$ for the $\infty$-category of spaces and $\spec$ for the $\infty$-category of spectra. The $G$-equivariant analogues are denoted $\spcg$ and $\specg$. Mapping spaces in an $\infty$-category $\cC$ are denoted $\Map_{\cC}$, and by $\Hom_{\cC}$ we mean $\pi_0 \Map_{\cC}$, the $\Hom$-set in the homotopy category $\tn{h}\cC$. If an $\infty$-category $\cC$ admits a symmetric monoidal enhancement we denote the corresponding operad by $\cC^{\otimes} \to \nN(\finstar)$, so that $\cC^{\otimes}_{\langle 1 \rangle} = \cC$ holds, and the symmetric monoidal structure on $\tn{h}\cC$ is denoted by $\otimes$. Abusing terminology we often say that $\cC$ is a symmetric monoidal $\infty$-category. The letters $t$ and $w$ are used for t- and weight structures respectively, e.g. hearts are denoted $\cC^{t \heart}$ and $\cC^{w \heart}$. For an additive ($\infty$-)category $\mcal{A}$ we denote its idempotent-completion by $\mcal{A}^{\natural}$. We tend to use upright letters for 1-categorical and curly letters for $\infty$-categorical constructions, e.g. for derived ($\infty$-)categories we write $\nD(-)$ and $\cD(-)$ respectively. \\

\noindent \tbf{Acknowledgements.} I want to thank my advisor Martin Gallauer for suggesting the topic, for his guidance during the work on it, and for his excellent comments on drafts of this paper. I furthermore would like to thank Benjamin Dünzinger, John Greenlees, Drew Heard, Achim Krause, Clover May and Isaac Moselle for excellent comments and discussions. \\
The author is supported by the Warwick Mathematics Institute Centre for Doctoral Training and gratefully acknowledges funding from the University of Warwick.

\section{Mackey Functors}\label{sec:mackey}

Let us start by recalling the basic objects of our study: (cohomological) Mackey functors and permutation modules. After that we prove a fact about the commutativity of $\infty$-categorical derivation and taking modules in the context of Grothendieck abelian categories, which we will apply to the category of Mackey functors. 

\begin{rec}
Let $\gSet$ be the category with objects sets (viewed as discrete topological spaces) on which $G$ acts continuously and morphisms $G$-equivariant maps. The category $\gset$ is the full subcategory on finite $G$-sets. Both categories become symmetric monoidal when equipped with the cartesian monoidal structure.
\end{rec}

\begin{rec}
We denote by 
$$\mackgr \defeq \funadd(\Omega(G)^{\op},\modr) = \funrlin(\Omega_R(G)^{\op},\modr)$$
the Grothendieck abelian category of \emph{Mackey functors} with values in $R$-modules. Here $\Omega(G)$ is the \emph{Burnside category} obtained by group completing the hom-monoids in $\spans(G)$, which is the category of finite $G$-sets and isomorphism classes of spans of $G$-equivariant maps between them, and $\Omega_R(G)$ is obtained from $\Omega(G)$ by tensoring each hom-group with $R$. The ($R$-linear) \emph{Burnside Mackey functor} is given by the functor $\bA_R \defeq \Hom_{\Omega_R(G)}(-,G/G)$. Since $(\Omega_R(G))^{\op} \simeq \Omega_R(G)$ the exact evaluation functor $\tn{ev}_{G/H}: \mackgr \to \Mod(R), \; M \mapsto M(G/H)$ is both represented and corepresented by $\bA_R(G/H \times -)$, and hence $\{ \bA_R(G/H \times -) \mid H \leq G \}$ is a set of compact\footnote{By a \emph{compact} object we always mean that mapping out of it commutes with filtered colimits. In the context of Grothendieck abelian categories this is often called \emph{finitely presented}.} projective generators for $\mackgr$. When $R = \bZ$, we write $\mackg$ for $\mack_{\bZ}(G)$ and $\bA$ for $\bA_{\bZ}$.
\end{rec}

\begin{rec}
We let $\modgr$ be the category of \emph{discrete $(G;R)$-modules}, which are $R$-modules with the discrete topology on which $G$ acts continuously. There is a natural functor 
$$\freer: \gSet \to \modgr, \; X \mapsto R(X),$$ and the objects in the essential image of this functor are called \emph{permutation modules}, the full subcategory they span is denoted $\Permgr$. The objects in the essential image of the finite $G$-sets are the \emph{finitely generated permutation modules}, the full subcategory they span is denoted $\permgr$. Both categories are symmetric monoidal under tensor product over $R$, and it makes the above free functor symmetric monoidal. When $R = \bZ$ we will just write $\Perm(G)$ and $\perm(G)$.
\end{rec}

\begin{defi}\label{defi:cohomologicalmackeyfunctors}
A Mackey functor $M: \Omega(G)^{\op} \to \modr$ is called \emph{cohomological} if it takes a span of canonical projections
\[\begin{tikzcd}
    & {G/K} \\
    {G/H} && {G/H}
	\arrow["{\pi^H_K}"', from=1-2, to=2-1]
	\arrow["{\pi^H_K}", from=1-2, to=2-3]
\end{tikzcd}\]
to multiplication by $[H:K]$ whenever $K \leq H \leq G$. Then $\cmackgr \subseteq \mackgr$ is defined as the full subcategory on cohomological Mackey functors. In the case $R=\bb{Z}$ we will again just write $\cmackg$.
\end{defi}

The following is shown by noting that $\perm(G;R)$ is equivalent to the locally group completed and $R$-linearised span category $\Omega_R(G)$ modulo a ``cohomological'' ideal $\cI_R(G)$, which is the two-sided ideal of homomorphisms on the generators 
$$(G/H \stackrel{\pi^H_K}{\longleftarrow} G/K \stackrel{\pi^H_K}{\longrightarrow} G/H) - ([H:K]\cdot \tn{id}_{G/H})$$
for each pair of subgroups $K \leq  H \leq G$, see \cite[Proposition 4.17]{bg23}.

\begin{lem}[Yoshida]\label{lem:comackasperm}
Precomposition with $\freer: \Omega_R(G) \to \permgr$ induces an equivalence of Grothendieck abelian categories 
$$\cmackgr \simeq \funrlin(\permgr^{\op},\modr) \simeq \funadd(\perm(G)^{\op},\modr).$$
\end{lem}
\begin{proof}
We refer to \cite[Corollary 4.22]{bg23}.
\end{proof}

\begin{rec}\label{re:dayconvolutionmack}
The category $\mackgr$ becomes an $R$-linear tensor category under Day convolution, which makes the $R$-linear Yoneda embedding 
$$\Omega_R(G) \hookrightarrow \funadd(\Omega_R(G)^{\op},\modr) = \mackgr$$
symmetric monoidal. That is, for $M,N \in \mackgr$, we define their monoidal product $M \boxtimes N$, also called \emph{box product}, to be the $R$-linear left Kan extension of $M(-) \otimes N(-)$ along the functor $- \times -$ given by the cartesian products of $G$-sets:
\[\begin{tikzcd}
	{\Omega_R(G)^{\op} \times \Omega_R(G)^{\op}} & {\modr \times \modr} & {\modr} \\
	& {\Omega_R(G)^{\op}}
	\arrow["{M \times N}", from=1-1, to=1-2]
	\arrow["{-\otimes-}", from=1-2, to=1-3]
	\arrow["{-\times-}"', from=1-1, to=2-2]
	\arrow["{M \boxtimes N}"', dashed, from=2-2, to=1-3]
\end{tikzcd}\]
The unit is the Burnside Mackey functor $\bA_R$. The symmetric monoidal structure is closed with internal hom given by $$\underl{\tn{Hom}}(M,N)(T) = \Hom_{\mackgr}(M,N(T\times -)) \tn{ for } T \in \gset.$$
\end{rec}

\begin{rec}
A \emph{Green functor} is a commutative monoid object in $\mackgr$.\footnote{The terminology Green functor is also used for non-commutative monoid objects, but we will always assume commutativity.} If $B$ is a Green functor, the category of modules $\Mod_B(\mackgr)$ is again Grothendieck abelian and the relative tensor product over $B$ defined as a coequaliser endows it with the structure of an $R$-linear tensor category. A set of compact projective generators for $\Mod_B(\mackgr)$ is given by the collection $\{ B(G/H \times -) \mid H \leq G \}$.
\end{rec}

\begin{ex}\label{ex:coefficientsaremodules}
We can express $\mackgr$ as $\Mod_{\bA_R}(\mackg)$. Then the functor $\mackgr \to \mackg$ forgetting the $R$-linear structure is just the restriction of scalars along the $R$-linearisation map $\bA \to \bA_R$.
\end{ex}

\begin{ex}\label{ex:cohomological}
We have $\cmackgr = \Mod_{\underl{R}}(\mackgr)$ by \cite[Proposition 16.3]{tw95}, where $\underl{R}$ is the constant Mackey functor associated to $R$. Under this identification the inclusion $\cmackgr \hookrightarrow \mackgr$ is just the restriction of scalars along $\bA_R \to \underl{R}$. Via \Cref{lem:comackasperm} we can endow $\cmackgr$ with a Day convolution symmetric monoidal structure. This symmetric monoidal structure agrees with the canonical one on $\Mod_{\underl{R}}(\mackgr)$.
\end{ex}

\begin{rec}\label{rec:modules}
If $\cC^{\otimes}$ is a presentably symmetric monoidal\footnote{That is, it is presentable and the associated tensor bifunctor commutes with colimits in both variables separately.} stable $\infty$-category and $A \in \calg(\cC^{\otimes})$, then the $\infty$-category $\Mod_A(\cC)^{\otimes}$ of $A$-modules in $\cC$ is presentably symmetric monoidal and stable as well \cite[Theorem 3.4.4.2]{lur17}. It comes with a free-forgetful adjunction
$$(F_A,U_A): \cC \rightleftarrows \Mod_A(\cC),$$
the free functor $F_A$ is symmetric monoidal, and the composite $U_AF_A$ is the functor $(-) \otimes A: \cC \to \cC$ \cite[Corollary 4.2.3.7, 4.2.4.8, Theorem 4.5.2.1, Section 4.5.3]{lur17}. If $\cC$ is compactly generated by a set of objects $\cG$, then $\Mod_A(\cC)$ is compactly generated by the set $F_A(\cG)$. If the objects of $\cG$ are also dualisable, then the objects of $F_A(\cG)$ are. If the unit of $\cC$ is compact then $A$ is compact in $\Mod_A(\cC)$, so in this case an object of $\Mod_A(\cC)$ is compact if and only if it is dualisable, and $\Mod_A(\cC)$ is rigidly-compactly generated by the set $F_A(\cG)$. \\
A symmetric monoidal functor $\alpha: \cC^{\otimes} \to \cD^{\otimes}$ restricts to commutative algebra objects $\calg(\cC^{\otimes}) \to \calg(\cD^{\otimes})$, and for any $A \in \calg(\cC^{\otimes})$ there is an induced symmetric monoidal functor
$$\bar{\alpha}: \Mod_A(\cC)^{\otimes} \to \Mod_{\alpha(A)}(\cD)^{\otimes}$$
on module categories. If $\alpha$ preserves colimits then so does the induced $\bar{\alpha}$ \cite[Corollary 4.2.3.5]{lur17}, and as explained in \cite[Remark 1.1.11]{erg22} both squares in the following diagram commute.
\[\begin{tikzcd}
	\cC & \cD \\
	{\Mod_A(\cC)} & {\Mod_{\alpha(A)}(\cD)}
	\arrow["\alpha", from=1-1, to=1-2]
	\arrow[""{name=0, anchor=center, inner sep=0}, "{F_A}"', shift right=2, from=1-1, to=2-1]
	\arrow[""{name=1, anchor=center, inner sep=0}, "{F_{\alpha(A)}}"', shift right=2, from=1-2, to=2-2]
	\arrow[""{name=2, anchor=center, inner sep=0}, "{U_A}"', shift right=2, from=2-1, to=1-1]
	\arrow["\bar{\alpha}", from=2-1, to=2-2]
	\arrow[""{name=3, anchor=center, inner sep=0}, "{U_{\alpha(A)}}"', shift right=2, from=2-2, to=1-2]
	\arrow["\dashv"{anchor=center}, draw=none, from=0, to=2]
	\arrow["\dashv"{anchor=center}, draw=none, from=1, to=3]
\end{tikzcd}\]
If $\alpha$ preserves limits, then so does $\bar{\alpha}$ \cite[Corollary 4.2.3.3]{lur17}. In particular, a right adjoint $\beta: \cD \to \cC$ to $\alpha$ induces a right adjoint $\bar{\beta}: \Mod_{\alpha(A)}(\cD) \to \Mod_A(\cC)$ to the functor $\bar{\alpha}: \Mod_A(\cC) \to \Mod_{\alpha(A)}(\cD)$. As a right adjoint to a symmetric monoidal functor the functor $\beta: \cD \to \cC$ is lax symmetric monoidal, and when $B \in \calg(\cD^{\otimes})$, then we also have an induced lax symmetric monoidal functor
$$\tilde{\beta}: \Mod_B(\cD)^{\otimes} \to \Mod_{\beta(B)}(\cC)^{\otimes}.$$
It factors as $\Mod_B(\cD) \stackrel{\varepsilon_B^*}{\longrightarrow} \Mod_{\alpha \beta(B)}(\cD) \stackrel{\bar{\beta}}{\longrightarrow} \Mod_{\beta(B)}(\cC)$, where $\varepsilon_B^*$ is the restriction of scalars along the counit $\varepsilon_B: \alpha \beta(B) \to B$, and hence $\tilde{\beta}$ has a symmetric monoidal left adjoint $\tilde{\alpha}: \Mod_{\beta(B)}(\cC) \to \Mod_B(\cD)$ given by the composition $\Mod_{\beta(B)}(\cC) \stackrel{\bar{\alpha}}{\longrightarrow} \Mod_{\alpha \beta(B)}(\cD) \stackrel{\varepsilon_{B,!}}{\longrightarrow} \Mod_B(\cD)$, where $\varepsilon_{B,!}$ is the extension of scalars along $\varepsilon_B: \alpha \beta(B) \to B$.
\end{rec}

\begin{rec}\label{rec:psigma}
Let $\cC$ be an $\infty$-category which admits finite coproducts. Let $\psigma(\cC) \subseteq \cP(\cC) = \fun(\cC^{\op},\spc)$ be the full subcategory on presheaves that preserve finite products. Then $\psigma(\cC)$ is compactly generated \cite[Proposition 5.5.8.10(6)]{lur09} and an accessible localisation of $\cP(\cC)$ \cite[Proposition 5.5.8.10(1)]{lur09}, we write 
$$L: \cP(\cC) \to \psigma(\cC)$$
for the left adjoint localisation functor. The Yoneda embedding $y: \cC \hookrightarrow \cP(\cC)$ factors through $\psigma(\cC)$ \cite[Proposition 5.5.8.10(2)]{lur09}. \\
If $\cC$ is additively symmetric monoidal,\footnote{That is, $\cC$ is additive and admits a symmetric monoidal structure for which the tensor product is additive in both variables.} then the presheaf $\infty$-categories $\cP(\cC)$ and $\cP(\cC;\spec) = \fun(\cC^{\op},\spec)$ become symmetric monoidal under Day convolution, which makes the stabilisation map 
$$\Sigma^{\infty}_{+}: \cP(\cC) \to \spec(\cP(\cC)) \simeq \cP(\cC;\spec)$$
symmetric monoidal \cite[Section 2.2.6, 4.8.1]{lur17}. As the tensor product preserves $L$-equivalences in both variables, $\psigma(\cC)$ by \cite[Proposition 2.2.1.9]{lur17} inherits a symmetric monoidal structure from $\cP(\cC)$ which informally is given by performing a Day convolution in $\cP(\cC)$ and then applying the localisation functor $L$. This makes $L: \cP(\cC) \to \psigma(\cC)$ and $y: \cC \hookrightarrow \psigma(\cC)$ symmetric monoidal.\\
Analogously, the full subcategory $\psigma(\cC;\spec) \subseteq \cP(\cC;\spec)$ on spectral presheaves which preserve finite products is a symmetric monoidal localisation of $\cP(\cC;\spec)$, cf. \cite[Remark 2.10]{aok20}. By \cite[Remark C.1.5.9]{lur18} or \cite[Lemma 2.6]{aok20} the $\infty$-category $\psigma(\cC;\spec)$ is the stabilisation of $\psigma(\cC)$, and there is a symmetric monoidal stabilisation map
$$\Sigma^{\infty}_{+}: \psigma(\cC) \to \spec(\psigma(\cC)) \simeq \psigma(\cC;\spec)$$
which we can view as the unit of the adjunction induced by the inclusion $\prlst \hookrightarrow \prl$ \cite[Proposition 4.8.2.18]{lur17}. Another way to think about this map is as follows: by \cite[Corollary 2.10(iii), Example 5.3(ii)]{ggn15} or \cite[Proposition C.1.5.7]{lur18} there is a symmetric monoidal equivalence $\psigma(\cC) \simeq \psigma(\cC;\spec_{t \geq 0})$ which is induced by the functor $\Omega^{\infty}: \spec_{t \geq 0} \to \spc$, and hence the inclusion of connective spectra into spectra induces a fully faithful symmetric monoidal functor $\psigma(\cC) \hookrightarrow \psigma(\cC;\spec)$. This gives a `stable Yoneda embedding' 
$$\bar{y}: \cC \stackrel{y}{\longrightarrow} \psigma(\cC) \to \psigma(\cC;\spec),$$
as in \cite[Corollary 2.9]{aok20}. It follows as in the proof of \cite[Theorem 3.7]{pst23} that the stable $\infty$-category $\psigma(\cC;\spec)$ is compactly generated by the presheaves $\bar{y}(c)$ for all $c \in \cC$. Stabilising \cite[Proposition 5.3.6.2]{lur09} it follows that $\psigma(\cC;\spec)$ is the initial presentable stable $\infty$-category receiving a functor from $\cC$ that preserves finite direct sums. If $\cD$ is any stable presentably symmetric monoidal $\infty$-category, then we have equivalences
\begin{align*}
\fun_{\otimes}^L(\psigma(\cC;\spec),\cD) &\simeq \fun_{\otimes}^L(\psigma(\cC),\cD) \\
&\simeq \fun_{\otimes}^{L,L\tn{-eq}}(\cP(\cC),\cD),
\end{align*}
where $L\tn{-eq}$ stands for functors inverting $L$-equivalences and $\otimes$ stands for symmetric monoidal functors. Functors in the latter category correspond to symmetric monoidal functors $\cC \to \cD$ whose colimit preserving extension $\cP(\cC) \to \cD$ factors through $\psigma(\cC)$, cf. \cite[Remark 2.31]{pst23}.
\end{rec}

\begin{defi}[cf. {\cite[Definition 2.51]{pst23}}]
Let $\cA$ be an abelian category. We call a full subcategory $\bP \subseteq \cA$ a \emph{system of compact projective generators} if it consists of compact projective objects which generate $\cA$ under colimits, contains the zero object and is closed under taking retracts and finite direct sums.
\end{defi}

\begin{re}
That $\bP$ generates $\cA$ under colimits is equivalent to the statement that for all $X \in \cA$ there exists an epimorphism $\bigoplus_{i \in I} P_i \twoheadrightarrow X$ from a direct sum of objects of $\bP$. \\
In \cite[Definition 2.51]{pst23} Pstrągowski without the projectivity assumption on the objects demands $\bP$ to be closed under pullbacks along epimorphisms (i.e. the cospan defining the pullback lies in $\bP$ and one of its legs is an epimorphism). Since we assume the objects of $\bP$ to be projective this condition is equivalent to closure under taking retracts and finite direct sums.
\end{re}

The following statement is a variant of \cite[Lemma 2.61]{pst23} which expresses the derived $\infty$-category of a Grothendieck abelian category as an $\infty$-category of additive presheaves of spectra on a system of compact projective generators. For this we fix a Grothendieck abelian category $\cA$ which has a system of compact projective generators $\bP$. Assume further that $\cA$ is closed symmetric monoidal, that $\bP$ is closed under tensor products and that $\bP$ contains the unit $\bbone_{\cA}$. We fix a commutative monoid object $B$ of $\cA$.

\begin{prop}\label{prop:psigmamodules}
Let $\bP_B$ be the full subcategory of $\Mod_B(\cA)$ consisting of objects $P \otimes B$ for $P \in \bP$ and direct summands of such. Then the inclusion $\bP_B \hookrightarrow \Mod_B(\cA)$ induces an equivalence
$$\cP_{\Sigma}(\bP_B; \spec) \simeq \cD(\Mod_B(\cA))$$
under which the complex $M[0]$ for $M \in \Mod_B(\cA)$ corresponds to the representable presheaf $\bar{y}(M)|_{\bP_B}$, where $\bar{y}: \Mod_B(\cA) \hookrightarrow \psigma(\Mod_B(\cA);\spec)$ is the stable Yoneda embedding from \Cref{rec:psigma}.
\end{prop}
\begin{proof}
We can view the forgetful functor $U: \Mod_B(\cA) \to \cA$ as a restriction of scalars along $B \to \bbone_{\cA}$, so it has adjoints on both sides given by extension and coextension of scalars. As it is faithful it is conservative, so it follows from \cite[Proposition 10.6.3.1]{lur18} that $\Mod_B(\cA)$ is Grothendieck abelian as well. We also check that $\bP_B$ inherits a symmetric monoidal structure from $\Mod_B(\cA)$, which we will use later. Since $\bP$ contains the unit $\bbone_{\cA}$, $\bP_B$ contains the unit $B \cong \bbone_{\cA} \otimes B$. We have $(P \otimes B) \otimes_B (P' \otimes B) \cong (P \otimes P') \otimes B$, and the tensor product of direct summands of $P \otimes B$ and $P' \otimes B$ is a direct summand of $(P \otimes P') \otimes B$. So since $\bP$ is closed under tensor products, $\bP_B$ is as well. \\
Now by \cite[Lemma 2.61]{pst23} it suffices to show that $\bP_B$ is a system of compact projective generators of $\Mod_B(\cA)$. If $\tn{colim}_{i \in I} M_i$ is a filtered colimit of $B$-modules and $P \in \bP$ we have
\begin{align*}
\Hom_{\Mod_B(\cA)}(P \otimes B, \colimit{i \in I} M_i) &\cong \Hom_{\cA}(P, \, U(\colimit{i \in I} M_i)) \\
&\cong \colimit{i \in I} \Hom_{\cA}(P, \, U(M_i)) \\
&\cong \colimit{i \in I} \Hom_{\Mod_B(\cA)}(P \otimes B, M_i),
\end{align*}
so the $P \otimes B$ and hence their direct summands are compact. For $P \in \bP$ we have $\Hom_{\Mod_B(\cA)}(P \otimes B,-) \cong \Hom_{\cA}(P,U(-))$, which preserves epimorphisms since $U$ preserves epimorphisms as a left adjoint and $P$ is projective. So $P \otimes B$ and any direct summand thereof is projective. If $M \in \Mod_B(\cA)$ there exists an epimorphism $\bigoplus_{i \in I} P_i \twoheadrightarrow U(M)$ with $P_i \in \bP$. Since $U$ is faithful, the counit of $(F \dashv U)$ is a pointwise epimorphism, so under the adjunction the above map corresponds to an epimorphism $\bigoplus_{i \in I} P_i \otimes B \twoheadrightarrow M$. Hence $\bP_B$ generates. Since $0 \in \bP$, $0 \cong 0 \otimes B \in \bP_B$. Finally, $\bP_B$ is closed under taking direct summands by definition, and since $\bP$ is closed under taking finite direct sums the same holds for $\bP_B$.
\end{proof}

\begin{re}
As $\bP_B$ inherits a symmetric monoidal structure from $\Mod_B(\cA)$, the $\infty$-category $\cP_{\Sigma}(\bP_B; \spec)$ always admits a symmetric monoidal enhancement given by localised Day convolution, see \Cref{rec:psigma}. If the symmetric monoidal structure of $\Mod_B(\cA)$ descends to $\cD(\Mod_B(\cA))$ so that the latter is presentably symmetric monoidal, then the equivalence of \Cref{prop:psigmamodules} enhances to a symmetric monoidal one, since it is induced by a symmetric monoidal inclusion and $\cD(\Mod_B(\cA))$ is stable and presentably symmetric monoidal. This uses the description of symmetric monoidal functors given in \Cref{rec:psigma}. So if we mention a symmetric monoidal structure on $\cD(\Mod_B(\cA))$ we will always mean the one inherited from $\cP_{\Sigma}(\bP_B; \spec)$, and there is no ambiguity.
\end{re}

\begin{re}\label{re:tstructurespsigmamodules}
By \cite[Proposition 2.16]{pst23} there is a standard t-structure on $\cP_{\Sigma}(\bP_B; \spec)$ for which an additive presheaf $F$ is connective if $\pi_n(F)(M)$ vanishes for all $M \in \bP_B$ and $n<0$, and it is coconnective if $\Omega^{\infty}F$ is a discrete presheaf of spaces, where $\Omega^{\infty}: \cP_{\Sigma}(\bP_B; \spec) \to \cP_{\Sigma}(\bP_B)$ is the right adjoint to stabilisation. With respect to the standard t-structure on the derived $\infty$-category the equivalence of \Cref{prop:psigmamodules} is t-exact, since it is induced by an equivalence of connective parts, see \cite[Remark 2.62]{pst23}.
\end{re}

We now furthermore assume that the class $\bP \subseteq \cA$ consists of dualisable objects and that $\bP$ is closed under taking duals, that is, all objects of $\bP$ are dualisable inside the symmetric monoidal subcategory $\bP$. Since $\cA$ is closed symmetric monoidal $\bP$ is automatically closed under taking duals if it is closed under taking internal homs.

\begin{thm}\label{thm:modulesandderivation}
There is a t-exact equivalence of symmetric monoidal $\infty$-categories 
$$\cD(\Mod_B(\cA)) \simeq \Mod_{B}(\cD(\cA)).$$
\end{thm}
\begin{proof}
We use the descriptions $\cD(\cA) \simeq \psigma(\bP;\spec)$, $\cD(\Mod_B(\cA)) \simeq \psigma(\bP_B;\spec)$ of \Cref{prop:psigmamodules}. Extension of scalars along $\bbone_{\cA} \to B$ restricts to a symmetric monoidal additive functor
$$\alpha: \bP \to \bP_B, \; P \mapsto P \otimes B,$$
which induces a restriction functor $\alpha^*: \cP(\bP_B;\spec) \to \cP(\bP;\spec)$ given by precomposition with $\alpha$. It has a left adjoint $\alpha_! \dashv \alpha^*$ given by left Kan extension along $\alpha$ and restricts to a functor $\bar{\alpha}^*: \psigma(\bP_B;\spec) \to \psigma(\bP;\spec)$. By \cite[Proposition 3.3.2]{chll24} the left Kan extension $\alpha_!$ also restricts to a functor $\bar{\alpha}_!:  \psigma(\bP;\spec) \to \psigma(\bP_B;\spec)$ which is a left adjoint to $\bar{\alpha}^*$. As $\bar{y} \circ \alpha: \bP \to \bP_B \to \psigma(\bP_B;\spec)$ is symmetric monoidal with left Kan extension along $y: \bP \hookrightarrow \cP(\bP)$ factoring through $\psigma(\bP)$, $\bar{\alpha}_!$ is canonically symmetric monoidal, see \Cref{rec:psigma}. \\
We want to apply the symmetric monoidal Barr-Beck-Lurie theorem \cite[Proposition 5.29]{mnn17} to the adjunction 
$$(\bar{\alpha}_!,\bar{\alpha}^*): \psigma(\bP;\spec) \rightleftarrows \psigma(\bP_B;\spec),$$
so we need to show that $\bar{\alpha}^*$ is conservative, preserves colimits, and that the projection formula holds. Filtered colimits in additive presheaves are computed pointwise, hence $\bar{\alpha}^*$ commutes with them. As it is an exact functor of stable $\infty$-categories it also commutes with finite colimits, hence with all colimits. If $G \in \psigma(\bP_B;\spec)$ such that $\bar{\alpha}^*(G) \simeq 0$, then $G(P \otimes B) \simeq 0$ for all $P \in \bP$. Since $\bP_B$ is the idempotent completion of the subcategory spanned by the objects $P \otimes B$ for $P \in \bP$, it follows $G \simeq 0$ and $\bar{\alpha}^*$ is conservative. For the projection formula we have to show that the natural map
$$\bar{\alpha}^*(G) \otimes F \to \bar{\alpha}^*(G \otimes \bar{\alpha}_!(F))$$
which is adjoint to 
\[\begin{tikzcd}
	{\bar{\alpha}_!(\bar{\alpha}^*(G) \otimes F) \simeq \bar{\alpha}_!(\bar{\alpha}^*(G)) \otimes \bar{\alpha}_!(F)} && {G \otimes \bar{\alpha}_!(F)}
	\arrow["{\varepsilon_G \otimes \tn{id}_{\bar{\alpha}_!(F)}}", from=1-1, to=1-3]
\end{tikzcd}\]
is an equivalence for all $F \in \psigma(\bP;\spec), \; G \in \psigma(\bP_B;\spec)$. Since the presheaf categories $\psigma(\bP;\spec)$ and $\psigma(\bP_B;\spec)$ are presentably symmetric monoidal and the functors $\bar{\alpha}_!$ and $\bar{\alpha}^*$ preserve colimits this follows from \cite[Lemme 2.8]{ayo14} if we show that both categories are generated under colimits by dualisable objects. \\
For $\psigma(\bP;\spec)$ this follows from the fact that $\bar{y}: \bP \hookrightarrow \psigma(\bP;\spec)$ is symmetric monoidal, that $\psigma(\bP;\spec)$ is generated under colimits by shifts of objects in the image of $\bar{y}$ and the assumption that the objects of $\bP$ are dualisable. For $\psigma(\bP_B;\spec)$ it holds for similar reasons, noting that the objects $P \otimes B$ for $P \in \bP$ are dualisable as they are images of dualisables under a symmetric monoidal functor and that dualisables in $\Mod_B(\cA)$ are closed under retracts since $\Mod_B(\cA)$ is idempotent complete. We obtain a symmetric monoidal equivalence 
$$\psigma(\bP_B;\spec) \simeq \Mod_{\bar{\alpha}^*(\bar{y}(B))}(\psigma(\bP;\spec))$$
which under \Cref{prop:psigmamodules} identifies with an equivalence 
$$\cD(\Mod_B(\cA)) \simeq \Mod_{B}(\cD(\cA)).$$
The asserted t-exactness follows from \Cref{re:tstructurespsigmamodules} and the fact that the restriction functor $\bar{\alpha}^*$ satisfies $\pi_n\bar{\alpha}^* \simeq \bar{\alpha}^*\pi_n$ and $\Omega^{\infty}\bar{\alpha}^* \simeq \bar{\alpha}^*\Omega^{\infty}$.
\end{proof}

\begin{cor}\label{cor:greenmodulesasmackmodules}
Let $B \in \mackgr$ be a Green functor. Then there is a t-exact equivalence of symmetric monoidal $\infty$-categories 
$$\cD(\Mod_B(\mackgr)) \simeq \Mod_{B}(\cD(\mackgr)).$$
\end{cor}
\begin{proof}
This follows from \Cref{thm:modulesandderivation} if we show that $\mackgr$ admits a system of compact projective dualisable generators $\bP$ which is closed under taking tensor products, finite direct sums, direct summands and contains the unit $\bA_R$. \\
Choose $\bP$ to be the set of Mackey functors which are direct summands of $\bA_R(T \times -)$ for a finite $G$-set $T$.\footnote{By \cite[Corollary 1.5]{Gre92} this is the set of compact projective Mackey functors.} Then $\bP$ clearly contains the zero object $0$, the unit $\bA_R$, and it consists of compact projective objects which generate $\mackgr$ under colimits. The objects $\bA_R(T \times -)$ are self-dual, so as $\mackgr$ is idempotent complete their direct summands are dualisable. Since $\bA_R(T \times -) \boxtimes \bA_R(T' \times -) \cong \bA_R(T \times T' \times -)$ and $\bA_R(T \times -) \oplus \bA_R(T' \times -) \cong \bA_R((T \sqcup T') \times -)$ it is closed under tensor products and direct sums. By definition $\bP$ is closed under taking direct summands.
\end{proof}

\section{Modules in Equivariant Spectra}\label{sec:modules}

There are multiple models for the $\infty$-category of $G$-spectra $\specg$. To name a few: It can be realised as the underlying $\infty$-category of the model category of orthogonal $G$-spectra, in the profinite case this model category has been constructed in \cite{fau08}. Furthermore there is a parametrised model \cite{bh21}, a model of spectral Mackey functors \cite{bar17}, and if $G$ is a compact Lie group there is a model whose construction starts from $G$-spaces and then inverts representation spheres \cite{gm20}. \\
By \cite[Theorem 1.14]{gm24}, \cite[Appendix A]{nar16}, \cite[Appendix A]{cmnn24}, \cite[Section 4]{psw22}, \cite[Appendix C]{gm20} and \cite[Section 9.2]{bh21} these different approaches are all equivalent. We also refer to \cite[Appendix A]{bbb24} for an extensive comparison in the profinite case.

\begin{no}
We fix some notation for the profinite group $G$. Let $(N_i)_{i \in I}$ be the cofiltered system of open normal subgroups of $G$ under inclusion and let $G_i \defeq G/N_i$, so that $G = \lim_i G_i$. Write $\alpha_i: G \twoheadrightarrow G_i$ for the natural projection. If $H \leq G$ is a closed subgroup, then we have $H \cong \lim_i H_i$ and $H$ is itself profinite, since $H_i \defeq HN_i/N_i \cong H/(H \cap N_i)$ is a subgroup of the finite group $G_i$.
\end{no}

\begin{rec}
The abelian category $\mackg$ of Mackey functors is equivalent to the t-heart of $\specg$, taken with respect to the standard t-structure. The image of a Mackey functor $M$ under the lax symmetric monoidal composition $\nN(\mack(G)) \simeq (\specg)^{t \heart} \hookrightarrow \specg$ is denoted $\tn{H}M \in \specg$, it is the \emph{Eilenberg-MacLane $G$-spectrum associated to the Mackey functor $M$}. By passing to commutative algebra objects we get a functor 
The \emph{$\infty$-category of $G$-spectra} $\specg$ is stable and presentably symmetric monoidal. We denote the symmetric monoidal product by $\otimes$ and the internal mapping $G$-spectrum by $\map_{\specg}(-,-)$. There is a symmetric monoidal left adjoint suspension functor from pointed $G$-spaces $\Sigma^{\infty}: \spcgp \to \specg$, and we denote the sphere spectrum $\Sigma^{\infty} G/G_{\plus}$ by $\bS = \bS^0$. It is the unit of the symmetric monoidal structure. The $\infty$-category $\specg$ is rigidly compactly generated by the suspension spectra of orbits $\Sigma^{\infty}G/H_+$, $H \leq G$. It has a natural t-structure for which the t-structure homotopy functors agree with the homotopy Mackey functors $\pi_n^{H}(-) = \Hom_{\specg}(\Sigma^{\infty + n} G/H_+,-)$. This t-structure is compatible with the symmetric monoidal structure on $\specg$, in the sense that $\specg_{t \geq 0}$ is a symmetric monoidal subcategory and the inclusion $\iota_{t \geq 0}: \specg_{t \geq 0} \hookrightarrow \specg$ is symmetric monoidal. By \cite[Example 2.2.1.10]{lur17} this implies that the t-heart $(\specg)^{t \heart}$ inherits a symmetric monoidal structure which makes the truncation $\tau_{t \leq 0}: \specg_{t \geq 0} \to (\specg)^{t \heart}$ into a symmetric monoidal functor.
\end{rec}

When $G$ is a compact Lie group, the $\infty$-category $\specg$ enjoys a universal property.

\begin{thm}[{\cite[Proposition 2.9, Corollary 2.20]{rob15}, \cite[Corollary C.7]{gm20}}]\label{thm:specguniversal}
For any symmetric monoidal left adjoint $F: \spcgp \to \cD$ into a presentably symmetric monoidal $\infty$-category $\cD$ so that $F(S^V)$ is invertible for every irreducible $G$-representation $V$, there is an essentially unique symmetric monoidal left adjoint $\overl{F}: \specg \to \cD$ with an equivalence $\overl{F} \circ \Sigma^{\infty} \simeq F$.
\end{thm}

\begin{rec}
For any continuous group homomorphism $\alpha: G \to G'$ there is a symmetric monoidal left adjoint functor $\alpha^*: \spec^{G'} \to \specg$, and we denote its lax symmetric monoidal right adjoint by $\alpha_*: \spec^{G} \to \spec^{G'}$. If $\alpha$ is the inclusion of a closed subgroup $H \hookrightarrow G$, the left adjoint $\alpha^*: \spec^{G} \to \spec^{H}$ is the \emph{restriction functor} $\tn{res}^G_H$, and its right adjoint is the \emph{coinduction} $\tn{coind}^G_H$. If $H$ is furthermore of finite index, $\tn{res}^G_H$ has a left adjoint $\tn{ind}^G_H$ called \emph{induction}. When $G$ is finite, the Wirthmüller isomorphism provides an equivalence $\tn{ind}^G_H \simeq \tn{coind}^G_H$.\\
If $\alpha: G \twoheadrightarrow G/N$ is a quotient homomorphism for $N \trianglelefteq G$ a closed normal subgroup then $\alpha^*: \spec^{G/N} \to \spec^{G}$ is the \emph{inflation functor} $\tn{infl}^G_{G/N}$, and its right adjoint $\alpha_*$ is the \emph{categorical fixed point functor} $(-)^N: \spec^{G} \to \spec^{G/N}$. 
\end{rec}

\begin{rec}
The abelian category $\mackg$ of Mackey functors is equivalent to the t-heart of $\specg$, taken with respect to the standard t-structure. The image of a Mackey functor $M$ under the lax symmetric monoidal composition $\nN(\mack(G)) \simeq (\specg)^{t \heart} \hookrightarrow \specg$ is denoted $\tn{H}M \in \specg$, it is the \emph{Eilenberg-MacLane $G$-spectrum associated to the Mackey functor $M$}. By passing to commutative algebra objects we get a functor 
$$\nN(\calg(\mackg)) \simeq \calg(\nN(\mackg)) \to \calg(\specg).$$
The functor $\mackgr \to \mackg$ which forgets $R$-module structures is lax symmetric monoidal as well, it is right adjoint to the symmetric monoidal $R$-linearisation functor $\mackg \to \mackgr$. So it induces a functor
$$\nN(\calg(\mackgr)) \to \nN(\calg(\mackg)) \to \calg(\specg).$$
\end{rec}

\begin{rec}\label{rec:preservation}
When $\alpha: G \to G'$ is a homomorphism of finite groups, an explicit calculation shows that the right adjoint $\alpha_*: \specg \to \spec^{G'}$ of the functor $\alpha^*: \spec^{G'} \to \specg$ is t-exact, and its restriction to t-hearts $\alpha_*^{t \heart}: \mack(G) \to \mack(G')$ is given by precomposition\footnote{Although it is given by precomposition we will keep the lower star notation.} with the canonical functor $\Omega(G') \to \Omega(G)$ which pulls back the group action along $\alpha: G \to G'$. \\
When $\alpha$ is the inclusion of a subgroup $H \hookrightarrow G$, $\alpha^* = \tn{res}^G_H: \spec^{G} \to \spec^{H}$ is t-exact as well\footnote{An abstract reason for this is that $\alpha^*$ preserves limits and thus has a left adjoint if and only if $\alpha$ is injective. In this case $\alpha^*$ is both left and right adjoint to a t-exact functor, thus itself t-exact.} and on t-hearts identifies with the left adjoint of $\alpha_*^{t \heart}$. This left adjoint is symmetric monoidal and preserves the Burnside Mackey functor as well as constant Mackey functors, see e.g. \cite[Section 6.1]{hmq23}. \\
But when $\alpha$ is a quotient map $G \twoheadrightarrow G/N$ for $N \trianglelefteq G$ a normal subgroup, then $\alpha^* = \tn{infl}^G_{G/N}: \spec^{G/N} \to \spec^{G}$ does not preserve Eilenberg-MacLane spectra of constant Mackey functors, as for example $\pi_0(\tn{infl}^G_1(\nH \bZ))$ by \cite[Section 5.2]{gs14} is given by the Burnside Mackey functor $\bA$ and not the constant Mackey functor $\underl{\bZ}$. This also shows that inflation does not preserve Eilenberg-MacLane spectra of Burnside Mackey functors either: if $\nH \bZ_G = \tn{infl}^G_1(\nH \bZ)$ were equivalent to $\nH \bA$, the $\infty$-categories $\Mod_{\nH \bZ_G}(\specg)$ and $\Mod_{\nH \bA}(\specg)$ would coincide. By the equivalences proven in \cite{psw22} this would give an equivalence between Kaledin's $\nD\mackg$ and the `honest' derived category $\nD(\mackg)$. But Kaledin shows in \cite{kal11} that these categories don't coincide. \\
The right adjoint $\alpha_* = (-)^N: \spec^{G} \to \spec^{G/N}$ on the other hand does preserve Eilenberg-MacLane spectra of constant Mackey functors but not the ones of Burnside Mackey functors, as the categorical fixed points $(\nH\bA)^G$ are the Eilenberg-MacLane spectrum of the Burnside ring $\bA(G/G) = A(G)$ of $G$, which differs from the integers when $G$ is nontrivial.
\end{rec}

We will often be in situations where maps of equivariant spectra by t-structure arguments are already determined on $\pi_0$, or even on specific sections of it. For one of these cases we record the following elementary lemma.

\begin{lem}\label{lem:pi0ofinflation}
Let $N \unlhd G$ be a normal subgroup of a finite group $G$ and $X \in \spec^{G/N}$. Then we have that $\pi_n^1(\tn{infl}_{G/N}^G(X)) \cong \pi_n^1(X)$ for all $n \in \bZ$. 
\end{lem}
\begin{proof} We calculate
\begin{align*}
\pi_n^1(\tn{infl}_{G/N}^G(X)) &= \Hom_{\specg}(\Sigma^{\infty+n} G/1_+, \, \tn{infl}_{G/N}^G(X)) \\
&\cong \Hom_{\specg}(\tn{ind}^G_1(\bS^n), \, \tn{infl}_{G/N}^G(X)) \\
&\cong \Hom_{\spec}(\bS^n, \, \tn{res}^G_1\tn{infl}_{G/N}^G(X)) \\
&\cong \Hom_{\spec}(\bS^n, \, \tn{res}^{G/N}_1(X)) \\
&\cong \Hom_{\spec^{G/N}}(\Sigma^{\infty+n} (G/N)/(N/N)_+, \, X) = \pi_n^1(X). \qedhere
\end{align*}
\end{proof}

\begin{cons}\label{cons:restrictionmodules}
From the example for inflation given in \Cref{rec:preservation} it follows that the functor $\alpha^*: \spec^{G'} \to \spec^{G}$ induced by some homomorphism of finite groups $\alpha: G \to G'$ in general does not induce a functor
$$\Mod_{\nH B}(\spec^{G'}) \to \Mod_{\nH B}(\specg)$$
in the naive way, simply because typically $\alpha^*(\nH B) \not\simeq \nH B$, where $B$ is either $\bA$ or the constant Mackey functor $\underl{R}$ associated to a ring $R$. We will remedy this problem by postcomposing the induced functor on module categories with a base change, so we need to construct ring maps
$$\alpha^*(\nH B) \to \nH B$$
along which we can extend scalars. Since every homomorphism of finite groups can be factored as the quotient by a normal subgroup $G \twoheadrightarrow G/N$ followed by an inclusion $G/N \hookrightarrow G'$, and the restriction $\tn{res}^{G'}_{G/N}: \spec^{G'} \to \spec^{G/N}$ preserves the Eilenberg-MacLane spectra of the Mackey functors $\bA$ and $\underl{R}$ in question, it suffices to construct the base change for quotient homomorphisms $G \twoheadrightarrow G/N$. \\
For the Burnside Mackey functor $\bA$, giving a ring map $\tn{infl}^G_{G/N}(\nH \bA) \to \nH \bA$ is equivalent to giving a ring map $\nH \bA \to (\nH \bA)^N \simeq \nH(\bA^N)$ in the t-heart. When $M$ is a $G$-Mackey functor, we write $M^N$ for the $G/N$-Mackey functor that is given by $M$ on the orbits inflated from $G/N$. We choose the unit map in $\mack(G/N)$, evaluated on a $G/N$-orbit $(G/N)/(H/N)$ it is the natural map 
$$\Hom_{\Omega(G/N)}((G/N)/(H/N),(G/N)/(G/N)) \to \Hom_{\Omega(G)}(G/H,G/G)$$
which is induced by the fully faithful inflation functor $\tn{infl}^G_{G/N}: G/N\tsf{-set} \hookrightarrow \gset$. \\
For the constant Mackey functor $\underl{R}$ we have $(\nH \underl{R})^N \simeq \nH (\underl{R}^N) \simeq \nH \underl{R}$, and hence we can use the ring map $\tn{infl}^G_{G/N}(\nH \underl{R}) \to \nH \underl{R}$ which under adjunction corresponds to the identity $\nH \underl{R} \to \nH \underl{R}$. Note that this is just the counit map $\tn{infl}^G_{G/N}((\nH \underl{R})^N) \to \nH \underl{R}$ of the adjunction $\tn{infl}^G_{G/N} \dashv (-)^N$. Note that since the domain is connective and the target is in the t-heart the map $\tn{infl}^G_{G/N}(\nH \underl{R}) \to \nH \underl{R}$ is determined on $\pi_0$, and as it maps into a constant Mackey functor and by \Cref{lem:pi0ofinflation} it is already determined by being identic on $\pi_0^1$. \\
Again writing $B$ for either $\bA$ or $\underl{R}$ and letting $\alpha: G \twoheadrightarrow G/N \hookrightarrow G'$ be a general homomorphism of finite groups, the above choices define a composition functor
$$\Mod_{\nH B}(\spec^{G'}) \stackrel{\alpha^*}{\longrightarrow} \Mod_{\alpha^*(\nH B)}(\spec^{G}) \to \Mod_{\nH B}(\spec^{G})$$
of the `honest' restriction functor on module categories induced by $\alpha$, followed by an extension of scalars along $\tn{infl}^G_{G/N}(\nH B) \to \nH B$. By the \emph{restriction functor on module categories} 
$$\alpha^*:  \Mod_{\nH B}(\spec^{G'}) \to \Mod_{\nH B}(\specg)$$
we will always mean the composition constructed in this way. We write $\nu(\alpha)$ for the ring map $\alpha^*(\nH \underl{R}) \to \nH \underl{R}$ to emphasise the dependency on the group homomorphism. We will use the notation $\nu^G_H$ when $\alpha$ is the inclusion of a subgroup $H \hookrightarrow G$ and $\nu^G_{G/N}$ when it is a quotient map $G \twoheadrightarrow G/N$.
\end{cons}

The following theorem expresses the derived $\infty$-category of Mackey functors as a category of modules in $\specg$. When $G$ is finite, this was shown by Patchkoria-Sanders-Wimmer. Relying on their result we provide a version for profinite $G$.

\begin{thm}[{\cite[Theorem 5.10]{psw22}} for $G$ finite]\label{thm:mackeyasspgmodules}
There is an equivalence of symmetric monoidal $\infty$-categories 
$$\cD(\mackg) \simeq \Mod_{\nH \bA}(\specg).$$
\end{thm}
\begin{proof}
When $G$ is finite, \cite[Theorem 5.10]{psw22} uses the universal property of $\specg$ to produce a symmetric monoidal left adjoint $\alpha: \specg \to \cD(\mackg)$. Then the Barr-Beck-Lurie theorem provides a symmetric monoidal equivalence functor $\bar{\beta}: \cD(\mackg) \stackrel{\sim}{\longrightarrow} \Mod_{\nH \bA}(\specg)$ which makes the diagram
\[\begin{tikzcd}
	& {\Mod_{\nH \bA}(\specg)} \\
	\cD(\mackg) & \specg.
	\arrow["{U_{\nH \bA}}", from=1-2, to=2-2]
	\arrow["\beta"', from=2-1, to=2-2]
	\arrow[from=2-1, to=1-2, "\bar{\beta}"]
\end{tikzcd}\]
commute, where $\beta$ is right adjoint to $\alpha$. \\
When $G$ is profinite, there is an equivalence $\specg \simeq \tn{colim}_{i} \, \spec^{G_i}$, where the colimit is taken in $\prlst$ along the (symmetric monoidal) left adjoints $f_{ij}^*$ induced by the group homomorphisms $f_{ij}: G_i \to G_j$ associated to the inclusions $N_i \hookrightarrow N_j$, see e.g. \cite[Section 6.A]{bbb24}. The natural maps $\spec^{G_i} \to \tn{colim}_i \, \spec^{G_i} \simeq \specg$ are the inflation functors $\alpha_{i}^* = \tn{infl}^G_{G_i}$ induced by the group homomorphisms $\alpha_i: G \twoheadrightarrow G/N_i = G_i$. We obtain an equivalence 
$$\Mod_{\nH \bA}(\specg) \simeq \tn{colim}_{i} \, \Mod_{\nH \bA}(\spec^{G_i}),$$
where the colimit is taken along the maps $\Mod_{\nH \bA}(\spec^{G_j}) \to \Mod_{\nH \bA}(\spec^{G_i})$ induced by the $f_{ij}$, as explained in \Cref{cons:restrictionmodules}. One can see this by writing $\specg$ as a $\prrst$-limit of the associated diagram of right adjoints, then passing to modules as in \Cref{rec:modules}, using that the right adjoints preserve the relevant Eilenberg-MacLane spectra and the proof of \cite[Theorem 6.42]{mnn17}, and finally passing back to left adjoints. \\
On the other hand, the Burnside category $\Omega(G)$ can be written as $\tn{colim}_i \, \Omega(G_i)$, the filtered colimit along the functors $\Omega(G_j) \to \Omega(G_i)$ given by pulling back along the $f_{ij}: G_i \to G_j$. This descends to a decomposition 
$$\mackg \simeq \tn{colim}_{i} \, \mack(G_i),$$
where the colimit is taken along the precomposition functors $\mack(G_i) \to \mack(G_j)$. Since these functors have both a right and a left adjoint given by the respective Kan extensions, they are exact and preserve projective objects, and the equivalence restricts to compact projectives: $\proj(\mackg) \simeq \tn{colim}_i \, \proj(\mack(G_i))$. Using \cite[Theorem 7.4.9]{bckw24} and the fact that the nerve functor commutes with filtered colimits we obtain an equivalence
$$\cK_b(\proj(\mackg)) \simeq \tn{colim}_{i} \, \cK_b(\proj(\mack(G_i)))$$
in $\catinf^{\tn{perf}}$, and taking Ind-objects gives the desired equivalence 
$$\cD(\mackg) \simeq \tn{colim}_{i} \, \cD(\mack(G_i))$$
in $\prlst$. We can therefore reduce to the case for finite groups. \\
We finally observe that the obtained equivalence is symmetric monoidal since the forgetful functor $\calg(\prlst) \to \prl$ preserves filtered colimits \cite[Remark 4.8.1.24, Corollary 3.2.3.2]{lur17} and the two diagrams of which we take colimits entirely lie in $\calg(\prlst)$.
\end{proof}

\begin{re}
When $G$ is finite, the proof of \cite[Theorem 5.10]{psw22} shows that the right adjoint $\beta$ satisfies $\pi_*^H\beta(M_{\bullet}) \cong H_*(M_{\bullet}(G/H)) \cong H_*(M_{\bullet})(G/H)$ for a complex of Mackey functors $M_{\bullet}$ and a subgroup $H \leq G$. Hence the above equivalence is t-exact for the standard t-structure on $\cD(\mackg)$ and the t-structure on $\Mod_{\nH \bA}(\specg)$ detected in $\specg$.\footnote{That is, a module is connective (resp. coconnective) if and only if it is connective (resp. coconnective) after applying the forgetful functor.} \\
From the equivalence $\Mod_{\nH \bA}(\specg) \simeq \tn{colim}_i \, \Mod_{\nH \bA}(\spec^{G_i})$ for profinite $G$ we obtain an equivalence
$$\Mod_{\nH \bA}(\specg) \simeq \tn{lim}_{i} \, \Mod_{\nH \bA}(\spec^{G_i})$$
by passing to the $\prrst$-diagram of right adjoint functors $f_{ij,*}$ on module categories (the right adjoints are given by restrictions of scalars, followed by the right adjoints to the restriction maps $f_{ij}^*: \Mod_{\nH \bA}(\spec^{G_j}) \to \Mod_{f_{ij}^*(\nH \bA)}(\spec^{G_i})$ as in \Cref{cons:restrictionmodules}). Since all t-structures on module categories are defined on underlying categories and all functors involved are t-exact the above equivalence $\Mod_{\nH \bA}(\specg) \simeq \lim_i \Mod_{\nH \bA}(\spec^{G_i})$ is t-exact as well, and the same holds for the equivalence $\cD(\mackg) \simeq \lim_i \cD(\mack(G_i))$ in $\prrst$ obtained from the one in $\prlst$ by passing to right adjoints. So the equivalence of \Cref{thm:mackeyasspgmodules} is t-exact.
\end{re}

\begin{re}\label{re:cellular}
It also follows from the proof of \cite[Theorem 5.10]{psw22} that when $G$ is finite the inverse $\bar{\alpha}: \Mod_{\nH \bA}(\specg) \to \cD(\mackgr)$ of $\bar{\beta}$ sends a free $\nH \bA$-module $\nH \bA \otimes \Sigma^{\infty}X$ for $X$ a pointed finite $G$-CW complex to the cellular chain complex of Mackey functors with $\bA$-coefficients $C^{\tn{cell}}_{\bullet}(X;\bA)$. In degree $n$ it is given by the Mackey functor $\pi_n^{(-)}(\Sigma^{\infty}X^n/X^{n-1})$, and the differentials are determined by the connecting homomorphism of the triple $(X^n,X^{n-1},X^{n-2})$.
\end{re}

For Mackey functors different from the Burnside Mackey functor, the following has appeared in the literature using the language of model categories \cite[Section 5.2]{gs14}, \cite[Corollary 5.2, Remark 5.3]{zen18} and $\infty$-categorically for the group $G=C_2$ \cite[Proposition 2.11]{hp25}.

\begin{thm}\label{thm:greenmodulesasspgmodules}
Let $B \in \mackgr$ be a Green functor. Then there is a t-exact equivalence of symmetric monoidal $\infty$-categories 
$$\cD(\Mod_B(\mackgr)) \simeq \Mod_{\nH B}(\specg).$$
\end{thm}
\begin{proof}
By \Cref{ex:coefficientsaremodules} we have $\mackgr \simeq \Mod_{\bA_R}(\mackg)$, so we can view $\Mod_B(\mackgr)$ as $\Mod_B(\mackg)$. We have t-exact equivalences of symmetric monoidal $\infty$-categories
\begin{align*}
\cD(\Mod_B(\mackg)) &\stackrel{\ref{cor:greenmodulesasmackmodules}}{\simeq} \Mod_{B}(\cD(\mackg)) \\
&\;\stackrel{\ref{thm:mackeyasspgmodules}}{\simeq} \; \Mod_{\nH B}(\Mod_{\nH \bA}(\specg)) \\
&\;\;\simeq\;\; \Mod_{\nH B}(\specg), 
\end{align*}
where the last equivalence holds by \cite[Corollary 3.4.1.9]{lur17}.
\end{proof}

\begin{re}
Since the equivalence of \Cref{thm:greenmodulesasspgmodules} is obtained from the one of \Cref{thm:mackeyasspgmodules} through a base change, it follows from \Cref{re:cellular} that when $G$ is finite, a free $\nH B$-module $\nH B \otimes \Sigma^{\infty}X$ (for $X$ a pointed finite $G$-CW complex) under $\Mod_{\nH B}(\specg) \simeq \cD(\Mod_B(\mackgr))$ corresponds to the cellular chain complex of Mackey functors with $B$-coefficients $C^{\tn{cell}}_{\bullet}(X;B) \simeq C^{\tn{cell}}_{\bullet}(X;\bA) \boxtimes B$.
\end{re}

\begin{re}
\Cref{thm:greenmodulesasspgmodules} can be seen as an equivariant analogue of the Schwede-Shipley theorem \cite[Theorem 5.1.6]{ss03}, for the symmetric monoidal version see \cite{shi07}, \cite{str20}: There is a Quillen equivalence between the model category of chain complexes over an ordinary ring $A$ and the model category of modules of spectra over the Eilenberg-MacLane ring spectrum $\nH A$. In the $\infty$-categorical setting, Lurie shows in \cite[Theorem 7.1.2.13]{lur17} that there is an equivalence of symmetric monoidal $\infty$-categories 
$$\cD(\Mod(A)) \simeq \Mod_{\nH A}(\spec).$$
\end{re}

For $B = \underl{R}$, \Cref{thm:greenmodulesasspgmodules} in conjunction with \Cref{ex:coefficientsaremodules} gives the following.

\begin{cor}\label{cor:comackasspgmodules}
There is a t-exact symmetric monoidal equivalence
$$\cD(\cmackgr) \simeq \Mod_{\nH \underl{R}}(\specg).$$
\end{cor}

\Cref{thm:greenmodulesasspgmodules} can be reformulated using the presheaf description of the derived $\infty$-category we used in \Cref{sec:mackey}. We thank Drew Heard and Achim Krause for pointing us to \cite[Example 2.15]{bcn25}, where this description is spelled out for the constant Mackey functor.

\begin{cor}\label{cor:modhbaspresheaves}
Let $B \in \mackgr$ be a Green functor. Then there is a t-exact equivalence of symmetric monoidal $\infty$-categories 
$$\Mod_{\nH B}(\specg) \simeq \psigma(B[\gset];\spec),$$
where $B[\gset] \subseteq \Mod_B(\mackgr)$ is the full subcategory on the functors $B(X \times -)$ for all $X \in \gset$. Hence $\Mod_{\nH B}(\specg)$ is the initial presentable stable $\infty$-category receiving a functor from $B[\gset]$ that preserves finite direct sums.
\end{cor}
\begin{proof}
Using \Cref{thm:greenmodulesasspgmodules} and the description of $\cD(\Mod_B(\mackgr))$ given in \Cref{prop:psigmamodules} and \Cref{cor:greenmodulesasmackmodules} we can identify $\Mod_{\nH B}(\specg)$ with $\psigma(\bP_B;\spec)$, where $\bP_B$ is the full (additive) subcategory of $\Mod_B(\mackgr)$ on direct summands of the objects $\bA_R(X \times -) \boxtimes B \cong B(X \times -)$ for finite $G$-sets $X$. Note that $\Mod_B(\mackgr)$ is abelian, so $\bP_B$ is just the idempotent completion of $B[\gset]$. Since $\spec$ is idempotent complete we have
$\psigma(\bP_B;\spec) \simeq \psigma(B[\gset];\spec)$.
\end{proof}

In the two cases $B = \bA_R$ and $B=\underl{R}$ we can simplify this further.

\begin{cor}
There is a t-exact symmetric monoidal equivalence
$$\Mod_{\nH \bA_R}(\specg) \simeq \psigma(\Omega_R(G);\spec).$$
If $R=\bZ$ the right hand side simplifies to $\psigma(\spans(G);\spec)$.
\end{cor}
\begin{proof}
This follows from \Cref{cor:modhbaspresheaves} by noting that $\bA_R[\gset]$ is equivalent to $\Omega_R(G)$ via the Yoneda embedding $\Omega_R(G) \hookrightarrow \mackgr$. When $R = \bZ$ then the local group completion of $\spans(G)$ which gives $\Omega(G)$ is not necessary, since $\spec$ is additive and we are considering functors which preserve finite direct sums.
\end{proof}

\begin{cor}[cf. {\cite[Example 2.15]{bcn25}}]\label{cor:modhunderlraspresheaves}
There is a t-exact symmetric monoidal equivalence
$$\Mod_{\nH \underl{R}}(\specg) \simeq \psigma(\permgr;\spec).$$
\end{cor}
\begin{proof}
Via the classical Yoshida equivalence of \Cref{lem:comackasperm}, the category $\underl{R}[\gset]$ is equivalent to $\permgr$. So this again follows from \Cref{cor:modhbaspresheaves}.
\end{proof}

\begin{re}\label{re:hrlinearstructure}
Since $\cD(\cmackgr)$ by definition is the dg-nerve of an $R$-linear dg-category, the $\infty$-category $\Mod_{\nH \underl{R}}(\specg)$ inherits the structure of an $\nH R$-linear stable $\infty$-category. That is, it is a module over $\Mod_{\nH R}(\spec)$ in $\prlst$. It receives symmetric monoidal $\nH R$-linear left adjoint functors from two other module categories of $G$-spectra.
Since the Burnside Mackey functor $\bA_R$ is the unit of $\mackgr$, there is a unique map of Green functors $\bA_R \to \underl{R}$. The induced ring map $\nH \bA_R \to \nH \underl{R}$ gives an extension of scalars functor
$$\Mod_{\nH \bA_R}(\specg) \to \Mod_{\nH \underl{R}}(\specg)$$
which identifies with the derived functor of the $1$-categorical extension of scalars $\mackgr \to \Mod_{\underl{R}}(\mackgr)$. We can also precompose with the extension of scalars $\Mod_{\nH \bA}(\specg) \to \Mod_{\nH \bA_R}(\specg)$ induced by $\bA \to \bA_R$. \\
The inflation $\tn{infl}^G_{1}: \spec \to \specg$ induces a symmetric monoidal functor on module categories $\Mod_{\nH R}(\spec) \to \Mod_{\nH R_G}(\specg)$, where $\nH R_G \defeq \tn{infl}^G_{1}(\nH R)$. Giving a symmetric monoidal left adjoint $\Mod_{\nH R_G}(\specg) \to \Mod_{\nH \underl{R}}(\specg)$ under $\specg$ by \cite[Theorem 4.8.5.11]{lur17} is equivalent to giving a ring map $\zeta_G: \nH R_G \to \nH \underl{R}$, which by adjunction corresponds to a ring map $\nH R \to (\nH \underl{R})^G \simeq \nH R$, hence to an endomorphism of the ring $R$. Choosing the identity (which is the only $R$-linear ring endomorphism of $R$), we obtain another extension of scalars functor
$$\zeta_{G,!}: \Mod_{\nH R_G}(\specg) \to \Mod_{\nH \underl{R}}(\specg).$$
When composed with $\Mod_{\nH R}(\spec) \to \Mod_{\nH R_G}(\specg)$ the latter functor also recovers the $\nH R$-linear structure on $\Mod_{\nH \underl{R}}(\specg)$ without going through dg-categories. \\
Since categorical fixed points are t-exact, the inflation $\tn{infl}^G_1$ is right t-exact and $\nH R_G$ is connective. Hence the ring map $\zeta_G: \nH R_G \to \nH \underl{R}$ factors through $\tau_{t \leq 0} \nH R_G$, which by \cite[Section 5.2]{gs14} is $\nH \bA_R$. The constructed ring maps give a commuting triangle of scalar extension functors
\[\begin{tikzcd}
	{\Mod_{\nH R_G}(\spec^G)} && {\Mod_{\nH \underl{R}}(\specg),} \\
	& {\Mod_{\nH \bA_R}(\specg)}
	\arrow[from=1-1, to=1-3]
	\arrow[from=1-1, to=2-2]
	\arrow[from=2-2, to=1-3]
\end{tikzcd}\]
relating an $\infty$-categorical analogue of Kaledin's derived Mackey functors \cite[Theorem 4.50]{psw22} and the derived $\infty$-categories of $\mackgr$ and $\cmackgr$.
\end{re}

\begin{re}\label{re:restrictionunderhr}
If $\alpha: G \to G'$ is a group homomorphism, then the restriction functor
$$\alpha^*: \Mod_{\nH \underl{R}}(\spec^{G'}) \to \Mod_{\nH \underl{R}}(\specg)$$
from \Cref{cons:restrictionmodules} becomes a functor under $\Mod_{\nH R}(\spec)$ for the left adjoints $\Mod_{\nH R}(\spec) \to \Mod_{\nH \underl{R}}(\spec^{(-)})$ provided by \Cref{re:hrlinearstructure}. When $\alpha$ is the inclusion of a subgroup this is clear, and when $\alpha$ is a quotient homomorphism $G \twoheadrightarrow G/N$ this boils down to the fact that $\zeta_G: \nH R_G \to \nH \underl{R}$ factors as $\nu^G_{G/N} \circ \tn{infl}^G_{G/N}(\zeta_{G/N})$, which can be checked on $\pi_0$.
\end{re}

\begin{re}
As in the previous two remarks, we will often reduce commutativity of diagrams of symmetric monoidal module $\infty$-categories to commutativity of diagrams of algebra objects and ultimately of objects in a t-heart, that is, to $1$-categorical data. This is possible in our cases since the adjunctions we use to do so, e.g. the ones coming from a t-structure, restrict to adjunctions on commutative algebra objects.
\end{re}

\section{Derived Permutation Modules}\label{sec:derivedpermutationmodules}

Permutation modules already came into play when we discussed cohomological Mackey functors, and we will now consider them in greater detail. Since they form an additive but non-abelian category, determining a suitable notion of `derivation' is a non-trivial task. Balmer-Gallauer proposed a candidate by taking into account fixed points for all subgroups and defined a derived tt-category of permutation modules in \cite{bg22a}. \\
In this section we will define the derived category of permutation modules as an $\infty$-category and then show that it is equivalent to the derived $\infty$-category of cohomological Mackey functors, this relies on the corresponding 1-categorical statements of \cite{bg23}. By \Cref{cor:comackasspgmodules} this will also express it as modules over the Eilenberg-MacLane spectrum associated to a constant Mackey functor.

\begin{cons}
The category $\Permgr$ is additive and equipped with a symmetric monoidal structure which is $R$-linear and preserves coproducts in both variables, inherited from the tensor product over $R$ on $\modgr$. Using \cite[Remark 1.3.2.2]{lur17} and by combining \cite[Proposition 1.3.4.5]{lur17} and \cite[Proposition 3.2.2]{hin16} we obtain a symmetric monoidal stable $\infty$-category $\cK(\Permgr)$ of unbounded complexes of permutation modules up to chain homotopy equivalence. \\
Denote by $\cK_{G-\tn{ac}}(\Permgr)$ the full subcategory of $\cK(\Permgr)$ on $G$-acyclic objects, i.e. on complexes $X \in \cK(\Permgr)$ for which the $H$-fixed points $X^H$ are acyclic for each subgroup $H \leq G$. As it is closed under taking shifts and cofibres, it defines a stable subcategory.
\end{cons}

\begin{defi}
The \emph{derived $\infty$-category of permutation modules of $G$ with coefficients in $R$} is defined as the Verdier localisation 
$$\cD\Permgr \defeq \frac{\cK(\Permgr)}{\cK_{G-\tn{ac}}(\Permgr)}.$$
That is, $\cD\Permgr$ is the cofibre of the inclusion of the stable subcategory 
$$\cK_{G-\tn{ac}}(\Permgr) \hookrightarrow \cK(\Permgr),$$
taken in $\catinfex$, the $\infty$-category of small stable $\infty$-categories and exact functors.
\end{defi}

\begin{lem}\label{lem:dpermlocalising}
Let $\tn{Loc}(R(G/H) \mid H \leq G) \subseteq \cK(\Permgr)$ be the localising subcategory generated by the objects $R(G/H)$, $H \leq G$. Then there is an equivalence
$$\tn{Loc}(R(G/H) \mid H \leq G) \simeq \cD\Permgr.$$
\end{lem}
\begin{proof}
The composition of exact functors between stable $\infty$-categories
$$\tn{Loc}(R(G/H) \mid H \leq G) \hookrightarrow \cK(\Permgr) \twoheadrightarrow \cD\Permgr.$$
is an equivalence on homotopy categories, as \cite[Proposition 3.9]{bg23} shows by employing the Neeman-Thomason localisation theorem. It hence is an equivalence of $\infty$-categories.
\end{proof}

\begin{re}
It follows that the quotient functor $\cK(\Permgr) \twoheadrightarrow \cD\Permgr$ exhibits $\cD\Permgr$ as a colocalisation of $\cK(\Permgr)$, i.e. it has a fully faithful left adjoint. In particular, we can view $\cD\Permgr$ not only as a quotient, but also as a full subcategory of $\cK(\Permgr)$.
\end{re}

\begin{lem}
The stable $\infty$-category $\cD\Permgr$ is compactly generated with subcategory of compact objects given by the thick subcategory of $\cK(\Permgr)$ generated by $\permgr$:
$$\cD\Permgr^{\omega} \simeq \tn{Thick}(\permgr) \simeq \cK_{\tn{b}}(\permgr^{\natural}).$$
\end{lem}
\begin{proof}
As $\cK(\Permgr)$ has small coproducts, \cite[Remark 1.4.4.1]{lur17} implies that compactness in and compact generation of $\cK(\Permgr)$ can be studied in terms of the homotopy category $\nK(\Permgr)$. Given this, the result again follows from the Neeman-Thomason localisation theorem, see \cite[Corollary 3.10]{bg23}.
\end{proof}

\begin{re}
It follows that $\cD\Permgr$, as a subcategory of the symmetric monoidal $\infty$-category $\cK(\Permgr)$, is closed under tensor products, since the compact part $\cK_{\tn{b}}(\Permgr^{\natural})$ is. So $\cD\Permgr$ inherits a symmetric monoidal structure which makes it into a presentably symmetric monoidal stable $\infty$-category.
\end{re}

\begin{thm}\label{thm:eqdperminf}
There is an equivalence of symmetric monoidal $\infty$-categories $$\cD\Permgr \simeq \cD(\cmackgr)$$
under which the $(G;R)$-permutation module $R(X)$ on a finite $G$-set $X$ corresponds to the cohomological Mackey functor $\underl{R}(X \times -)$.
\end{thm}
\begin{proof}
We will rebuild the equivalence of tt-categories
$$\nD\Permgr \hookrightarrow \nK(\Permgr) \stackrel{\tn{FP}}{\longrightarrow} \nK(\cmackgr) \twoheadrightarrow \nD(\cmackgr)$$
from \cite[Corollary 5.7]{bg23} $\infty$-categorically. Here $\tn{FP}$ is the fixed point functor
\begin{align*}
\tn{FP}: \; &\modgr \to \funrlin(\permgr^{\op},\modr) \stackrel{\ref{lem:comackasperm}}{\simeq} \cmackgr, \\
& M \mapsto \Hom_{\modgr}(-,M)|_{\, \permgr}.
\end{align*}
The functor $\tn{FP}$ is lax symmetric monoidal, $R$-linear \cite[Lemma 5.5 (a)]{bg23} and by \cite[Theorem 16.5]{tw95},\cite[Proposition 5.6]{bg23} it restricts to equivalences of tensor categories $\Permgr^{\natural} \simeq \Proj(\cmackgr)$ and $\permgr^{\natural} \simeq \proj(\cmackgr)$, where $\Proj(-)$ stands for the subcategory of projective objects and $\proj(-)$ for the one on compact projective objects. The restriction of $\tn{FP}$ to $\Permgr$ induces a lax symmetric monoidal functor 
$$\cK(\Permgr) \to \cK(\Proj(\cmackgr)),$$
and there is a canonical symmetric monoidal functor 
$$\cK(\Proj(\cmackgr)) \twoheadrightarrow \cD(\cmackgr)$$ which inverts quasi-isomorphisms. The inclusion $\cD\Permgr \hookrightarrow \cK(\Permgr)$ by definition of the symmetric monoidal structure on $\cD\Permgr$ is symmetric monoidal. We obtain a lax symmetric monoidal composition
$$\cD\Permgr \hookrightarrow \cK(\Permgr) \to \cK(\cmackgr) \twoheadrightarrow \cD(\cmackgr)$$
which is an equivalence on the level of homotopy categories. So as the outer two $\infty$-categories are stable, the composition is an equivalence on the level of $\infty$-categories as well. On compact parts, it restricts to a symmetric monoidal equivalence
$$\cK_{\tn{b}}(\permgr^{\natural}) \simeq \cK_{\tn{b}}(\proj(\cmackgr)),$$
since it is induced by the tensor equivalence $\permgr^{\natural} \simeq \proj(\cmackgr)$ to which $\tn{FP}$ restricts. So as the symmetric monoidal structures commute with colimits in both variables, the above composition is symmetric monoidal as well.
\end{proof}

\begin{cor}\label{cor:eqhrmoddperm}
There is an equivalence of symmetric monoidal $\infty$-categories
$$\Mod_{\nH \underl{R}}(\specg) \simeq \cD\Permgr.$$
If $X$ is a pointed finite $G$-CW complex, then the free module $\nH \underl{R} \otimes \Sigma^{\infty} X$ corresponds to the cellular chain complex of $(G;R)$-permutation modules $C^{\tn{cell}}_{\bullet}(X;R)$.
\end{cor}
\begin{proof}
Combine \Cref{cor:comackasspgmodules} and \Cref{thm:eqdperminf}.
\end{proof}

\section{Weight Structures}\label{sec:weightstructures}

We first recall the definition of a weight structures on a stable $\infty$-category and consider weight exact and weight conservative functors between $\infty$-categories with weight structures. We then define a weight structure on the compact part of the $\infty$-category of $\nH \underl{R}$-modules in equivariant spectra and identify its weight heart.

\begin{rec}
Let $\cC$ be a stable $\infty$-category. A \emph{weight structure} on $\cC$ is a pair of full subcategories $(\cC_{w \leq 0}, \cC_{w \geq 0})$ such that the following properties hold:
\begin{enumerate}
    \item[(1)] $\cC_{w \leq 0}$ and $\cC_{w \geq 0}$ are closed under retracts in $\cC$,
    \item[(2)] $\Hom_{\cC}(x,\Sigma^n y) \cong 0$ for all $x \in \cC_{w \leq 0}$, $y \in \cC_{w \geq 0}$ and $n \geq 1$,
    \item[(3)] for all $x \in \cC$ there exists a \emph{weight decomposition} cofibre sequence
    $$x_{w \leq 0} \to x \to x_{w \geq 1}$$
    with $x_{w \leq 0} \in \cC_{w \leq 0}$ and $\Sigma^{-1}x_{w \geq 1} \in \cC_{w \geq 0}$.
\end{enumerate}
The \emph{heart} of the weight structure $(\cC_{w \leq 0}, \cC_{w \geq 0})$ is the additive subcategory of $\cC$ given by the intersection
$$\cC^{w\heart} \defeq \cC_{w \leq 0} \cap \cC_{w \geq 0}.$$
For $n \in \bZ$ we write $\cC_{w \leq n} \defeq \Sigma^n \cC_{w \leq 0}$, $\cC_{w \geq n} \defeq \Sigma^n \cC_{w \geq} 0$ and $\cC_{w = n} \defeq \cC_{w \leq n} \cap \cC_{w \geq n}$. The weight structure is \emph{bounded above} if $\cC = \bigcup_{n \in \bZ} \cC_{w \leq n}$, it is \emph{bounded below} if $\cC = \bigcup_{n \in \bZ} \cC_{w \geq n}$, and it is \emph{bounded} if it is both bounded above and bounded below.
\end{rec}

\begin{rec}
An exact functor $F: \cC \to \cD$ between stable $\infty$-categories with weight structures is called \emph{left weight exact} if $F(\cC_{w \leq 0}) \subseteq \cD_{w \leq 0}$, \emph{right weight exact} if $F(\cC_{w \geq 0}) \subseteq \cD_{w \geq 0}$, and \emph{weight exact} if it is both left and right weight exact. It is called \emph{left weight conservative} if for $x \in \cC$, $F(x) \in \cD_{w \leq 0}$ implies $x \in \cC_{w \leq 0}$, \emph{right weight conservative} if for $x \in \cC$, $F(x) \in \cD_{w \geq 0}$ implies $x \in \, \cC_{w \geq 0}$, and \emph{weight conservative} if it is both left and right weight conservative.
\end{rec}

We will now show that weight conservativity of a weight exact functor can be phrased in terms of a condition on the weight hearts, if the weight structure on the source is bounded. To a large extent this follows ideas of Bachmann \cite[Section 5.2.1]{bac16}. We later apply it to the family of yet to be defined modular fixed point functors followed by restriction to the trivial subgroup, where the family is indexed on the $p$-subgroups of a finite group $G$.

\begin{lem}
Let $\cC$ be a stable $\infty$-category with a weight structure $(\cC_{w \leq 0}, \cC_{w \geq 0})$.
\begin{enumerate}
    \item[(1)] $\cC_{w \leq 0} = \{ x \in \cC \mid \Hom_{\cC}(x,y) = 0 \tn{ for all } y \in \cC_{w \geq 1} \}$.
    \item[(2)] $\cC_{w \geq 0} = \{ y \in \cC \mid \Hom_{\cC}(x,y) = 0 \tn{ for all } x \in \cC_{w \leq -1} \}$.
\end{enumerate}
\end{lem}
\begin{proof}
We prove (1), assertion (2) is analogous. The inclusion of $\cC_{w \leq 0}$ into the right-hand side holds by property (2) defining a weight structure. Conversely, let $x \in \cC$ such that $\Hom_{\cC}(x,y)=0$ for all $y \in \cC_{w \geq 1}$. Pick a cofibre sequence $x_{w \leq 0} \to x \to x_{w \geq 1}$ with $x_{w \leq 0} \in \cC_{w \leq 0}$ and $\Sigma^{-1}x_{w \geq 1} \in \cC_{w \geq 0}$. Then the map $x \to x_{w \geq 1}$ is trivial, and hence $x_{w \leq 0} \simeq x \oplus \Sigma^{-1}x_{w \geq 1}$. So $x$ is a retract of an object in $\cC_{w \leq 0}$ and thus itself in $\cC_{w \leq 0}$.
\end{proof}

\begin{lem}\label{lem:weightcofibre}
Let $\cC$ be a stable $\infty$-category with a weight structure $(\cC_{w \leq 0}, \cC_{w \geq 0})$. 
\begin{enumerate}
    \item[(1)] If $f: x \to y$ is a morphism with $x, y \in \cC_{w \leq n}$ then $\tn{cofib}(f) \in \cC_{w \leq n+1}$.
    \item[(2)] If $x \to y \to z$ is a cofibre sequence with $x,z \in \cC_{w \leq n}$ then $y \in \cC_{w \leq n}$ as well.
    \item[(3)] A cofibre sequence $x \to y \to z$ with $x \in \cC_{w \geq n}$ and $z \in \cC_{w \leq n}$ splits.
\end{enumerate}
\end{lem}
\begin{proof}
(1) We need to show that $\Hom_{\cC}(\tn{cofib}(f),t) = 0$ for all $t \in \cC_{w \geq n+2}$. We have an exact sequence 
$$\pi_1\Map_{\cC}(x,t) \to \pi_0\Map_{\cC}(\tn{cofib}(f),t) \to \pi_0\Map_{\cC}(y,t).$$
The right term vanishes by definition and $\pi_1\Map_{\cC}(x,t) \simeq \pi_0\Map_{\cC}(\Sigma x,t) = 0$. So $\Hom_{\cC}(\tn{cofib}(f),t) = \pi_0\Map_{\cC}(\tn{cofib}(f),t) = 0$ as well.\\
(2) For $t \in \cC_{w \geq n+1}$ we have an exact sequence with vanishing outer terms
$$\Hom_{\cC}(z,t) \to \Hom_{\cC}(y,t) \to \Hom_{\cC}(x,t).$$
(3) The cofibre sequence $y \to z \to \Sigma x$ gives an exact sequence
$$\Hom_{\cC}(z,y) \to \Hom_{\cC}(z,z) \to \Hom_{\cC}(z,\Sigma x).$$
As $\Hom_{\cC}(z,\Sigma x) = 0$, $\Hom_{\cC}(z,y) \to \Hom_{\cC}(z,z)$ is onto and $y \to z$ has a section.
\end{proof}

\begin{lem}\label{lem:retractionandweighti}
Let $\cC$ be a stable $\infty$-category with a weight structure $(\cC_{w \leq 0}, \cC_{w \geq 0})$. Then a morphism $f: x \to y$ with $x \in \cC_{w = n}$ and $y \in \cC_{w \leq n}$ admits a retraction if and only if $\tn{cofib}(f) \in \cC_{w \leq n}$.
\end{lem}
\begin{proof}
By \Cref{lem:weightcofibre}(1) we have a cofibre sequence 
$$x \to y \to \tn{cofib}(f)$$
for which $x \in \cC_{w = n}$, $y \in \cC_{w \leq n}$ and $\tn{cofib}(f) \in \cC_{w \leq n+1}$. If $f$ admits a retraction the sequence splits and $\tn{cofib}(f)$ is a retract of $y$. Hence $\tn{cofib}(f) \in \cC_{w \leq n}$. Conversely, if $\tn{cofib}(f) \in \cC_{w \leq n}$ then \Cref{lem:weightcofibre}(3) implies that the sequence splits, and so $f$ admits a retraction.
\end{proof}

\begin{lem}\label{lem:retractionandweightii}
Let $\cC$ be a stable $\infty$-category with a weight structure $(\cC_{w \leq 0}, \cC_{w \geq 0})$. For $x \in \cC_{w \leq 0}$ choose a weight decomposition $a \to x \to b$ with $a \in \cC_{w \leq -1}, b \in \cC_{w \geq 0}$ and another weight decomposition $c \to a \to d$ with $c \in \cC_{w \leq -2}$, $d \in \cC_{w \geq -1}$. Then $b, \Sigma d \in \cC^{w\heart}$ and the composition $b \to \Sigma a \to \Sigma d$ admits a retraction if and only if $x \in \cC_{w \leq -1}$.
\end{lem}
\begin{proof}
Rotating the first cofibre sequence gives a cofibre sequence $x \to b \to \Sigma a$ with $x,\Sigma a \in \cC_{w \leq 0}$, so by \Cref{lem:weightcofibre}(2) we have $b \in \cC^{w \heart}$. By the same argument applied to the second cofibre sequence we have $\Sigma d \in \cC^{w \heart}$. For $t \in \cC^{w \heart}$ the first cofibre sequence gives an exact sequence
$$\Hom_{\cC}(\Sigma a,t) \to \Hom_{\cC}(b,t) \to \Hom_{\cC}(x,t) \to \Hom_{\cC}(a,t)$$
with $\Hom_{\cC}(a,t)=0$. The second cofibre sequence gives an exact sequence
$$\Hom_{\cC}(\Sigma d,t) \to \Hom_{\cC}(\Sigma a,t) \to \Hom_{\cC}(\Sigma c,t),$$
and as $\Hom_{\cC}(\Sigma c,t)=0$, $\Hom_{\cC}(\Sigma d,t) \to \Hom_{\cC}(\Sigma a,t)$ is surjective. It follows that 
$$\Hom_{\cC}(\Sigma d,t) \to \Hom_{\cC}(b,t) \to \Hom_{\cC}(x,t) \to 0$$
is exact, where the first map is induced by the composition $b \to \Sigma a \to \Sigma d$. \\
Now if $x \in \cC_{w \leq -1}$ then $\Hom_{\cC}(x,t) = 0$ and $\Hom_{\cC}(\Sigma d,t) \to \Hom_{\cC}(b,t)$ is surjective, so choosing $t = b$ shows that $b \to \Sigma a \to \Sigma d$ admits a retraction. \\
Conversely, if $b \to \Sigma a \to \Sigma d$ admits a retraction then $\Hom_{\cC}(\Sigma d,t) \to \Hom_{\cC}(b,t)$ is surjective and so $\Hom_{\cC}(x,t) = 0$. Hence the map $x \to b$ is zero and $a \simeq x \oplus \Sigma^{-1} b$, so $x$ is a retract of an object in $\cC_{w \leq -1}$ and hence in $\cC_{w \leq -1}$ as well.
\end{proof}

\begin{prop}[{\cite[Corollary 5.15]{bac16}}]\label{prop:weightconservativecharac}
Let $F: \cC \to \cD$ be a weight exact functor between stable $\infty$-categories equipped with weight structures such that the weight structure on $\cC$ is bounded.\footnote{Bounded above would be enough for the first and bounded below for the second assertion.} Then $F$ is left weight conservative if and only if $F^{w\heart}: \cC^{w\heart} \to \cD^{w\heart}$ detects retractions, i.e. if $f: x \to y$ is a morphism in $\cC^{w\heart}$ such that $F(f): F(x) \to F(y)$ admits a retraction, then $f$ admits a retraction as well. Dually, $F$ is right weight conservative if and only if $F^{w\heart}$ detects sections.
\end{prop}
\begin{proof}
We prove the statement about left weight conservativity, the statement about right weight conservativity is dual. Assume that $F$ is left weight conservative and let $f: x \to y$ be a morphism in $\cC^{w\heart}$. We get a cofibre sequence 
$$x \to y \to \tn{cofib}(f)$$
with $x,y \in \cC_{w = 0}$ and $\tn{cofib}(f) \in \cC_{w \leq 1}$. Applying $F$ gives a cofibre sequence 
$$F(x) \to F(y) \to \tn{cofib}(F(f))$$
with $F(x),F(y) \in \cD_{w = 0}$ and $\tn{cofib}(F(f)) \in \cD_{w \leq 1}$. If $F(f)$ admits a retraction, then $F(\tn{cofib}(f)) \simeq \tn{cofib}(F(f)) \in \cD_{w \leq 0}$ by \Cref{lem:retractionandweighti}, so $\tn{cofib}(f) \in \cC_{w \leq 0}$ and hence again by \Cref{lem:retractionandweighti} the map $f$ admits a retraction. \\
Conversely, assume that $F^{w\heart}$ detects retractions and let $x \in \cC$ with $F(x) \in \cD_{w \leq 0}$. Since the weight structure on $\cC$ is bounded, $x \in \cC_{w \leq n}$ for some $n \geq 1$. Hence it suffices to show: if $x \in \cC_{w \leq 0}$ and $F(x) \in \cD_{w \leq -1}$, then $x \in \cC_{w \leq -1}$. This now follows from \Cref{lem:retractionandweightii}.
\end{proof}

\begin{rec}
Let $\cA$ be an additive idempotent complete $1$-category. Then the $\infty$-category of bounded chain complexes up to chain homotopy equivalence defined as $\cK_{\tn{b}}(\cA) = \nN_{\tn{dg}}(\ch_{\tn{b}}(\cA))$ admits a canonical bounded weight structure for which $\cK_{\tn{b}}(\cA)_{w \geq 0}$ resp. $\cK_{\tn{b}}(\cA)_{w \leq 0}$ are the essential images of $\nN(\ch_{\tn{b},\geq 0}(\cA))$ resp. $\nN(\ch_{\tn{b},\leq 0}(\cA))$ under the functor $\nN(\ch_{\tn{b}}(\cA)) \twoheadrightarrow \cK_{\tn{b}}(\cA)$ which inverts chain homotopy equivalences. The canonical fully faithful functor $\cA \hookrightarrow \cK_{\tn{b}}(\cA)$ factors through the weight heart and exhibits an equivalence $\cK_{\tn{b}}(\cA)^{w\heart} \simeq \cA$ for which we refer to \cite[Section 1.3]{sos19} and \cite[Example 3.5]{aok20}.
\end{rec}

\begin{re}
The equivalence $\Mod_{\nH \underl{R}}(\specg) \simeq \cD\Permgr$ of \Cref{cor:eqhrmoddperm} restricts to an equivalence $\Mod_{\nH \underl{R}}^{\omega}(\specg) \simeq \cD\Permgr^{\omega} \simeq \cK_{\tn{b}}(\permgr^{\natural})$ of compact parts. Since $\cK_{\tn{b}}(\permgr^{\natural})$ has a canonical bounded weight structure it endows $\Mod_{\nH \underl{R}}^{\omega}(\specg)$ with a bounded weight structure. We will put ourselves in the situation where we don't know about this equivalence and independently construct a bounded weight structure on $\Mod_{\nH \underl{R}}^{\omega}(\specg)$. We will see that from this the equivalence $\Mod_{\nH \underl{R}}^{\omega}(\specg) \simeq \cK_{\tn{b}}(\permgr^{\natural})$ already follows. In particular, it is not necessary to go through cohomological Mackey functors. \\
In addition, the presence of this weight structure will provide the correct framework to translate certain properties which lie in the realm of permutation modules to the world of genuine equivariance (see \Cref{prop:conservativity}), a circumstance which justifies independent interest in constructing such a weight structure.
\end{re}

\begin{prop}\label{prop:weightonhrmod}
Let $\cH \subseteq \Mod_{\nH \underl{R}}^{\omega}(\specg)$ be the full subcategory on finite direct sums of the $\underl{R}$-linear orbits $\nH \underl{R} \otimes \Sigma^{\infty}G/H_+$ for $H \leq G$. Then there exists a bounded weight structure on $\Mod_{\nH \underl{R}}^{\omega}(\specg)$ with weight heart $\cH^{\natural}$. 
\end{prop}
\begin{proof}
First note that the embedding $\cH \hookrightarrow \Mod_{\nH \underl{R}}^{\omega}(\specg)$ extends to an embedding $\cH^{\natural} \hookrightarrow \Mod_{\nH \underl{R}}^{\omega}(\specg)$. Indeed, since $\Mod_{\nH \underl{R}}(\specg)$ is presentable it has all colimits and we have an embedding $\cH^{\natural} \hookrightarrow \Mod_{\nH \underl{R}}(\specg)$. As the objects of $\cH$ are compact their direct summands are, and the embedding of $\cH^{\natural}$ lands in the compact part. \\ 
The statement follows from \cite[Theorem 4.3.2 II]{bon10} once we have shown that $\cH^{\natural}$ generates $\tn{h}\Mod_{\nH \underl{R}}^{\omega}(\specg)$ as a thick subcategory and that it is negative, i.e. the mapping spectrum $\map_{\cC}(a,b)$ is connective for all $a,b \in \cH^{\natural}$. \\
Since the orbits $\nH \underl{R} \otimes \Sigma^{\infty}G/H_+$ form a set of compact generators for $\Mod_{\nH \underl{R}}(\specg)$ and compact generation of stable $\infty$-categories can be studied in terms of the homotopy category, the same is true for $\tn{h}\Mod_{\nH \underl{R}}(\specg)$. So $\tn{h}\cH$ generates $\tn{h}\Mod_{\nH \underl{R}}^{\omega}(\specg)$ as a thick subcategory. We show that $\cH$ is negative. For this, note that orbits are self-dual. This is well known for finite groups, and in the profinite case an orbit corresponding to an open subgroup of $G$ is inflated from some $\spec^{G_i}$, where it is self-dual. So for $n \geq 1$ we have
\begin{align*}
& \, \Hom_{\Mod_{\nH \underl{R}}(\specg)}(\nH \underl{R} \otimes \Sigma^{\infty}G/H_+, \bS^n \otimes \nH \underl{R} \otimes \Sigma^{\infty}G/K_+) \\
&\cong \Hom_{\specg}(\bS^{-n} \otimes \Sigma^{\infty}G/H_+, \nH \underl{R} \otimes \Sigma^{\infty}G/K_+) \\
&\cong \Hom_{\specg}(\bS^{-n} \otimes \Sigma^{\infty}(G/H \times G/K)_+, \nH \underl{R}) = 0
\end{align*}
by definition of $\nH \underl{R}$. This also implies that $\cH^{\natural}$ is negative, since the idempotent completion adds directs summands of linearised orbits $\nH \underl{R} \otimes \Sigma^{\infty}G/H_+$ to $\cH$.
\end{proof}

\begin{lem}\label{lem:heartperm}
There is a symmetric monoidal equivalence $\cH^{\natural} \simeq \permgr^{\natural}$. In particular, the subcategory $\cH^{\natural} \subseteq \Mod_{\nH \underl{R}}^{\omega}(\specg)$ is (the nerve of) a $1$-category.
\end{lem}
\begin{proof}
Since compact and dualisable objects in $\Mod_{\nH \underl{R}}(\specg)$ agree, its compact part inherits a symmetric monoidal structure. We first show that $\cH^{\natural}$ inherits a symmetric monoidal structure from $\Mod_{\nH \underl{R}}^{\omega}(\specg)$. It contains the unit $\nH \underl{R}$, and
\begin{align*}
(\nH \underl{R} \otimes \Sigma^{\infty}G/H_+) \otimes_{\nH \underl{R}} (\nH \underl{R} \otimes \Sigma^{\infty}G/K_+) &\simeq \nH \underl{R} \otimes \Sigma^{\infty}(G/H \times G/K)_+ \\
&\simeq \bigoplus_{[g] \in K \backslash G/H} \nH \underl{R} \otimes \Sigma^{\infty}(G/K^g \cap H)_+,
\end{align*}
hence $\cH$ is closed under tensor products, and the same follows for $\cH^{\natural}$. We show that $\permgr \simeq \cH$, then the statement follows by passing to idempotent completions. There is a symmetric monoidal and finite coproduct preserving composition
\[\begin{tikzcd}
	\Spans(G) & \fun^{\times}(\Spans(G)^{\op},\spc) & \specg & {\Mod_{\nH \underl{R}}(\specg),}
	\arrow[hook, from=1-1, to=1-2]
    \arrow[from=1-2, to=1-3]
	\arrow["{F_{\nH \underl{R}}}", from=1-3, to=1-4]
\end{tikzcd}\]
where the second functor is the map 
\begin{align*}
\fun^{\times}(\Spans(G)^{\op},\spc) &\simeq \fun^{\times}(\Spans(G)^{\op},\tn{Mon}_{\tn{Comm}}(\spc)) \\
&\to \fun^{\times}(\Spans(G)^{\op},\spec) \simeq \specg,
\end{align*}
using \cite[Corollary 2.5(iii)]{ggn15}, the finite product preserving group completion functor $\tn{Mon}_{\tn{Comm}}(\spc) \to \tn{Mon}_{\tn{Comm}}^{\tn{gp}}(\spc)$ constructed in \cite[Corollary 4.4]{ggn15} and the inclusion $\tn{Mon}_{\tn{Comm}}^{\tn{gp}}(\spc) \simeq \spec_{t \geq 0} \hookrightarrow \spec$ that uses \cite[Remark 5.2.6.26]{lur17}, see also \cite[Section A.2]{nar17}. Here $\Spans(G)$ is the effective Burnside $\infty$-category of finite $G$-sets \cite[Section 3]{bar17}. Using the fact that every finite $G$-set splits into a finite disjoint sum of orbits the essential image is precisely $\cH$. Let us consider mapping spaces, for which it suffices to consider single orbits. We have that
\begin{align*}
& \, \pi_n\Map_{\Mod_{\nH \underl{R}}(\specg)}(\nH \underl{R} \otimes \Sigma^{\infty}G/K_+, \nH \underl{R} \otimes \Sigma^{\infty}G/H_+) \\
&\cong \Hom_{\specg}(\bS^{n} \otimes \Sigma^{\infty}(G/K \times G/H)_+, \nH \underl{R}) \\
&\cong \begin{cases}
			\underl{R}(G/K \times G/H), & n = 0, \\
            0, & n \neq 0,
		 \end{cases}
\;\; \cong \begin{cases}
			R(K \backslash G/H), & n = 0, \\
            0, & n \neq 0.
		 \end{cases}
\end{align*}
Hence $\cH$ is a $1$-category, and since $\Mod_{\nH \underl{R}}^{\omega}(\specg)$ is $\nH R$-linear, $\cH$ is $R$-linear. Thus the above composition factors through the $R$-linearisation of the homotopy category of $\Spans(G)$, and we obtain a symmetric monoidal functor $\Omega_R(G) \to \cH$. Tracing through the isomorphisms, the functor $\Omega_R(G) \to \cH$ for $K,H \leq G$, $L \leq K^g \cap H$ and $[g] \in K\backslash G /H$ maps a span 
$$G/K \stackrel{{}^g\pi^K_{L}}{\longleftarrow} G/L \stackrel{\pi^H_{L}}{\longrightarrow} G/H$$
to the element $[K^g \cap H : L] \cdot [g] \in R(K\backslash G /H)$.
In particular, if $K \leq H \leq G$, then the span $(G/H \leftarrow G/K \rightarrow G/H)$ consisting of two canonical projections induces multiplication by the index $[H:K]$. Hence the functor $\Omega_R(G) \to \cH$ kills the cohomological ideal $\cI_R(G)$ and descends to a symmetric monoidal functor $\permgr \to \cH$. It follows from the description of the hom-sets in $\permgr$ \cite[Corollary 2.13, Proposition 4.17]{bg23} that $\permgr \to \cH$ is fully faithful.
\end{proof}

Having established the weight structure and identified its weight heart, we can reprove the equivalence $\Mod_{\nH \underl{R}}^{\omega}(\specg) \simeq \cK_{\tn{b}}(\permgr^{\natural})$. This is independent of the considerations around \Cref{cor:eqhrmoddperm}.

\begin{thm}\label{thm:eqhrmoddpermweight}
There is a symmetric monoidal weight exact equivalence 
$$\Mod_{\nH \underl{R}}^{\omega}(\specg) \simeq \cK_{\tn{b}}(\permgr^{\natural}).$$
\end{thm}
\begin{proof}
The weight structure on $\Mod_{\nH \underl{R}}^{\omega}(\specg)$ from \Cref{prop:weightonhrmod} is bounded and by \Cref{lem:heartperm} its weight heart is equivalent to $\permgr^{\natural}$. Hence \cite[Corollary 3.4]{sos19} gives a weight exact equivalence 
$$\Mod_{\nH \underl{R}}^{\omega}(\specg) \simeq \stab(\permgr^{\natural}),$$
where $\stab(\permgr^{\natural})$ is the stable envelope of $\permgr^{\natural}$, the full subcategory of $\psigma(\permgr^{\natural};\spec)$ generated by the essential image of the stable Yoneda embedding $\permgr^{\natural} \hookrightarrow \psigma(\permgr^{\natural};\spec)$ of \Cref{rec:psigma} under finite limits and colimits. Using the same argument one obtains a further weight exact equivalence $\cK_{\tn{b}}(\permgr^{\natural}) \simeq \stab(\permgr^{\natural})$. We get a composite equivalence 
$$\Mod_{\nH \underl{R}}^{\omega}(\specg) \simeq \cK_{\tn{b}}(\permgr^{\natural})$$
which restricts to the identity on weight hearts and hence identifies with the weight complex functor in the sense of \cite[Section 3]{bon10}, see \cite[Proposition 3.3, Corollary 3.5]{sos19}. Since the weight hearts on both sides are closed under tensor products and the identification of the weight heart of $\Mod_{\nH \underl{R}}^{\omega}(\specg)$ is symmetric monoidal, it follows from \cite[Theorem 4.3]{aok20} that $\Mod_{\nH \underl{R}}^{\omega}(\specg) \simeq \cK_{\tn{b}}(\permgr^{\natural})$ refines to a symmetric monoidal equivalence.
\end{proof}

\section{Geometric and Modular Fixed Points}\label{sec:fixedpoints}

In \cite{bg22a,bg25b}, Balmer and Gallauer construct, for a closed pro-$p$-subgroup $H$ of a profinite group $G$ and a commutative ring $R$ with $p=0$ in $R$, a weight exact $R$-linear `modular fixed point' tt-functor
$$\Psi^H: \nD\Perm(G;R) \to \nD\Perm(\wgh;R)$$
which linearises $H$-fixed points of $G$-sets; that is, it maps the permutation module $R(X)$ for a finite $G$-set $X$ to the permutation module $R(X^H)$. The functor already exists on additive categories of permutation modules, where it recovers the classical Brauer quotient \cite[Section 1]{bro85}, and its construction naturally extends to the $\infty$-categorical setting of \Cref{sec:derivedpermutationmodules}. \Cref{cor:eqhrmoddperm} and \Cref{thm:eqhrmoddpermweight} suggest that such a functor should equally exist in the equivariant context of $\nH \underl{R}$-modules of equivariant spectra, and the correspondence of the permutation module $R(X)$ to the free $\nH \underl{R}$-module $\nH \underl{R} \otimes \Sigma^{\infty}_+X$ furthermore suggests that it should be induced by geometric fixed points. Its construction will be the main objective of this section. We will first assume all groups to be finite and then extend the construction to the profinite case. Let us first review the construction of geometric fixed points.

\begin{rec}\label{rec:geometricfixedpoints}
There are multiple ways to define geometric fixed points, one of them uses the universal property of $G$-spectra (\Cref{thm:specguniversal}), as explained in \Cref{sec:intro}. We now recall the classical, arguably more hands on construction. For a subgroup $H \leq G$ we denote by $\cF_H$ the family of subgroups of $N_G(H)$ which do not contain $H$. By 
$$A_H \defeq \Pi_{K \in \cF_H} \, \map_{\spec^{N_G(H)}}(\Sigma^{\infty} N_G(H)/K_+,\bS) \in \calg(\spec^{N_G(H)})$$
we denote the commutative algebra object corresponding to this family. Here each of the mapping spectra $\map_{\spec^{N_G(H)}}(\Sigma^{\infty} N_G(H)/K_+,\bS)$ is a commutative algebra object of $\spec^{N_G(H)}$, since $N_G(H)/K$ is a commutative coalgebra in $N_G(H)$-spaces. Write $\spec^{N_G(H)}[A_H^{-1}]$ for the full subcategory of $\spec^{N_G(H)}$ on the $A_H^{-1}$-local objects \cite[Definition 3.10]{mnn17}. The localisation 
$$\spec^{N_G(H)} \stackrel{L}{\longrightarrow} \spec^{N_G(H)}[A_H^{-1}] \hookrightarrow \spec^{N_G(H)}$$
associated to the set of orbits $\Sigma^{\infty} N_G(H)/K_+$ for $K \in \cF_H$ is smashing, and by \cite[Proposition 6.5]{mnn17} it is given by tensoring with $\Sigma^{\infty}\tilde{E}\cF_H$, where $\tilde{E}\cF_H$ is the pointed classifying space of the family $\cF_H$. It fits into the cofibre sequence 
$$E\cF_{H ,+} \to S^0 \to \tilde{E}\cF_H.$$
The \emph{geometric $H$-fixed point functor} $\Phi^H = \Phi^{H;G}: \specg \to \specwgh$ is defined as the composition
\[\begin{tikzcd}[column sep=1.4cm]
	\specg & {\spec^{N_G(H)}} & {\spec^{N_G(H)}[A_H^{-1}] \subseteq \spec^{N_G(H)}} & {\specwgh.}
	\arrow["{\tn{res}^G_{N_G(H)}}", from=1-1, to=1-2]
	\arrow["{(-)\otimes \Sigma^{\infty}\tilde{E}\cF_H}", from=1-2, to=1-3]
	\arrow["(-)^H", from=1-3, to=1-4]
\end{tikzcd}\]
By the proof of \cite[Proposition 3.22]{psw22} the composition 
$$\spec^{N_G(H)}[A_H^{-1}] \hookrightarrow \spec^{N_G(H)} \stackrel{(-)^H}{\longrightarrow} \specwgh$$
is an equivalence of symmetric monoidal $\infty$-categories, and the inverse is given by $L \circ \tn{infl}_{\wgh}^{N_G(H)}$. It follows that $\Phi^{H;G}$ is a symmetric monoidal left adjoint, and we denote its right adjoint by $\Phi_H = \Phi_{H;G}: \specwgh \to \specg$. By \cite[Corollary 9.9]{lms86} we furthermore have that 
$$\Phi^{H} \Sigma^{\infty}X \simeq \Sigma^{\infty}(X^H)$$
for any $X \in \spcgp$. Let $\alpha: G \to G'$ be a group homomorphism, let $H \leq G$ be a subgroup and set $H' \defeq \alpha(H)$, giving an induced homomorphism $\overl{\alpha}: \wgh \to \weyl{G'}{H'}$. It follows from the essential uniqueness in \Cref{thm:specguniversal} that there is a commutative square of symmetric monoidal left adjoints
\[\begin{tikzcd}
	{\spec^{G'}} & {\spec^{G}} \\
	{\spec^{\weyl{G'}{H'}}} & \specwgh
	\arrow["{\alpha^*}", from=1-1, to=1-2]
	\arrow["{\Phi^{H';G'}}"', from=1-1, to=2-1]
	\arrow["{\Phi^{H;G}}", from=1-2, to=2-2]
	\arrow["{\overl{\alpha}^*}", from=2-1, to=2-2]
\end{tikzcd}\]
which in the case of a normal subgroup $N \trianglelefteq G$ implies that $\Phi^N : \specg \to \spec^{G/N}$ is split by $\tn{infl}^G_{G/N}$. Similarly, if $H \leq G$ is a subgroup and if $\bar{K} = K/H \leq \wgh$ is another subgroup, then ``nesting'' geometric fixed points gives a commutative diagram
\[\begin{tikzcd}
	\specg && \specwgh \\
	{\spec^{\weyl{G}{K}}} && {\spec^{\weyl{(\wgh)}{(\bar{K})}}.}
	\arrow["{\Phi^{H;G}}", from=1-1, to=1-3]
	\arrow["{\Phi^{K;G}}"', from=1-1, to=2-1]
	\arrow["{\Phi^{\bar{K};\wgh}}", from=1-3, to=2-3]
	\arrow["{\tn{res}_{\weyl{(\wgh)}{(\bar{K})}}^{\weyl{G}{K}}}", from=2-1, to=2-3]
\end{tikzcd}\]
If we choose $H=G$ we get the total geometric fixed point functor $\Phi^G: \spec^G \to \spec$, and it sends an object $X \in \specg$ to the spectrum $(X \otimes \Sigma^{\infty}\tilde{E}\cF_G)^G$, where now $\cF_G$ is the family of all proper subgroups of $G$. For these (and further) properties of geometric fixed points we also refer to \cite[Section 2]{bs17}.
\end{rec}

\begin{re}
For $G = C_{2^n}$ by \cite[Proposition 3.18]{hhr16} we have that 
$$\pi_*(\Phi^G(\nH\underl{\bZ})) \cong \bZ/2[b], \tn{ where } b \in \pi_2(\Phi^G(\nH\underl{\bZ})).$$ 
Using the arguments presented in \cite[Section 3]{hhr16} we can generalise this to other constant coefficients. If $M$ is any Mackey functor, the group $\pi_k(\Phi^G(\nH M)) \cong \pi_k^G(\nH M \otimes \Sigma^{\infty} \tilde{E}\cF_G)$ is given by the Bredon homology group $H_k^G(\tilde{E}\cF_G;M)$, see \cite[Section 5]{gm95}. If $\sigma$ is the sign representation of $C_{2^n}$, then the representation sphere $S^{\infty \sigma}$ of countably many copies of $\sigma$ provides a model for $\tilde{E}\cF_G$, and we can write it as the colimit $S^{\infty \sigma} \cong \tn{colim}_{d \to \infty} S^{d \sigma}$. By \cite[Example 3.8]{hhr16} the group $H_k^G(S^{d \sigma};M)$ is computed as the homology of the length $d+1$ chain complex 
\[\begin{tikzcd}
    {M(C_{2^n}/C_{2^{n-1}})} & \dots & {M(C_{2^n}/C_{2^{n-1}})} & {M(C_{2^n}/C_{2^{n-1}})}
	\arrow[from=1-1, to=1-2]
	\arrow["2", from=1-2, to=1-3]
	\arrow["0", from=1-3, to=1-4]
\end{tikzcd}\]
with alternating differentials. If $M=\underl{R}$, the constant Mackey functor associated to a ring $R$ in which $2=0$, we get the complex
\[\begin{tikzcd}
	R & \dots & R & R
	\arrow["0", from=1-1, to=1-2]
	\arrow["0", from=1-2, to=1-3]
	\arrow["0", from=1-3, to=1-4]
\end{tikzcd}\]
and it follows that $H_k^G(S^{d \sigma};M)=R$ for $0 \leq k \leq d$. So $\pi_k(\Phi^G(\nH \underl{R}))=R$ for all $k \geq 0$ and hence geometric fixed points in general do not preserve Eilenberg-MacLane spectra, they are generally not left t-exact.\footnote{Since they are defined by tensoring with a connective object followed by a t-exact functor they are of course right t-exact, and hence their right adjoint is left t-exact.}
\end{re}

\begin{re}
The previous remark shows that a modular fixed point functor
$$\Mod_{\nH \underl{R}}(\specg) \to \Mod_{\nH \underl{R}}(\specwgh)$$
in the equivariant setting can not be induced by geometric fixed points in the naive way. That is, it is not the functor 
$$\Mod_{\nH \underl{R}}(\specg) \to \Mod_{\Phi^H(\nH \underl{R})}(\specwgh)$$
induced by geometric fixed points on module categories. This goes in line with the arguments presented in \cite[Remark 4.11]{bg25a}: Writing $\tn{SH}(G)$ and $\tn{SH}$ for the triangulated homotopy categories $h\spec^G$ and $h\spec$ respectively, the total geometric fixed point functor on compact parts
$$\Phi^G: \tn{SH}^c(G) \to \tn{SH}^c$$
can be obtained as the Verdier localisation at ``everything induced from proper subgroups'', i.e. at the thick subcategory generated by the subcategories $\tn{ind}^G_H(\tn{SH}^c(H))$ for all $H \lneqq G$. This localisation is equivalent to the triangulated category of finite spectra $\tn{SH}^c$. If a total modular fixed point functor was induced by total geometric fixed points on module categories in the naive way, it would correspond to a localisation on module categories (see \cite[Proposition 3.18]{psw22}) with target $\cD_{\tn{perf}}(R)$. Translated to permutation modules, this localisation is the Verdier quotient of $\nD\Permgr^c$ by the thick subcategory generated by the $\tn{ind}^G_H(\nD\Perm(H;R)^c)$ for all $H \lneqq G$, and this quotient is richer than $\nD\Perm(1;R)^c \simeq \nD_{\tn{perf}}(R)$. But by how we will construct the total equivariant modular fixed point functor it will be clear that it factors through this quotient.
\end{re}

\begin{re}\label{re:geometricfixedpointsremedy}
Although the modular fixed point functor in the equivariant context is not induced by geometric fixed points in the naive way, there is an obvious fix which we will use as a definition. For $K/H \leq \wgh$ there are isomorphisms
\begin{align*}
\pi^{K/H}_n(\Phi^{H;G}(\nH M)) &= \Hom_{\specwgh}(\bS^n \otimes \Sigma^{\infty} (\wgh)/(K/H)_+,\Phi^{H;G}(\nH M)) \\
&\cong \Hom_{\spec^{K/H}}(\bS^n, \tn{res}^{\wgh}_{K/H} \Phi^{H;G}(\nH M)) \\
&\cong \Hom_{\spec^{K/H}}(\bS^n, \Phi^{H;K}\tn{res}^G_K(\nH M)) \\
&\cong \Hom_{\spec^{K}}(\bS^n, \Sigma^{\infty}\tilde{E}\cF_H \otimes \tn{res}^G_K(\nH M)) \\
&\cong H_n^{K}(\tilde{E}\cF_H;M),
\end{align*}
the $n$-th Bredon homology of the finite $K$-CW complex $\tilde{E}\cF_H$ with coefficients in the Mackey functor $M$. From this observation or the definition as a smash with a connective spectrum followed by categorical fixed points one sees that the spectrum $\Phi^H(\nH M)$ will always be connective, and the $0$-truncation
$$(\Phi^H(\nH M))^{t \heart} \defeq \iota_{t \leq 0}\tau_{t \leq 0}(\Phi^H(\nH M))$$
lives in the t-heart of $\specwgh$. If $M$ is a Green functor, then $\nH M$ is a commutative algebra object of $\specg$, and if $(\Phi^H(\nH M))^{t \heart} \simeq \nH M$ as algebra objects such that the unit map 
$$\Phi^H(\nH M) \to (\Phi^H(\nH M))^{t \heart} \simeq \nH M$$
is a map of algebra objects, then we can consider the composition
\[\begin{tikzcd}
	{\Mod_{\nH M}(\spec^{G})} & {\Mod_{\Phi^H(\nH M)}(\specwgh)} &&& {\Mod_{\nH M}(\specwgh)}
	\arrow["{\Phi^H}", from=1-1, to=1-2]
	\arrow["{(-) \otimes_{\Phi^H(\nH M)} \nH M}", from=1-2, to=1-5]
\end{tikzcd}\]
of geometric fixed points on module categories, followed by a base change back to $\nH M$-modules. We will revisit this construction in \Cref{defi:eqmodularfixedpoints}.
\end{re}

\begin{lem}\label{lem:weylspectramodules}
For a subgroup $H \leq G$ there is a symmetric monoidal equivalence
$$\overl{\Phi}_H: \specwgh \stackrel{\sim}{\longrightarrow} \Mod_{\Phi_H(\bS)}(\specg).$$ 
The equivalence functor makes the following diagram commute:
\[\begin{tikzcd}
	& {\Mod_{\Phi_H(\bS)}(\specg)} \\
	{\specwgh} & {\specg.}
	\arrow["{U_{\Phi_H(\bS)}}", from=1-2, to=2-2]
	\arrow["{\overl{\Phi}_H}", from=2-1, to=1-2]
	\arrow["{\Phi_H}"', from=2-1, to=2-2]
\end{tikzcd}\]
In particular, geometric fixed points $\Phi^H: \specg \to \specwgh$ under the above equivalence identify with the extension of scalars $\specg \to \Mod_{\Phi_H(\bS)}(\specg)$ along $\bS \to \Phi_H(\bS)$.
\end{lem}
\begin{proof}
We apply the Barr-Beck-Lurie theorem in the symmetric monoidal version \cite[Proposition 5.29]{mnn17} to the adjunction $(\Phi^H,\Phi_H): \specg \rightleftarrows \specwgh$. Both categories are rigidly compactly generated and presentably symmetric monoidal, so the projection formula is automatically satisfied, cf. \cite[Lemme 2.8]{ayo14}. As $\Phi^H$ preserves dualisable objects it preserves compact objects, so $\Phi_H$ commutes with filtered colimits. As $\Phi_H$ preserves limits and is a functor of stable $\infty$-categories it also preserves finite colimits, hence all colimits. \\
It remains to show that $\Phi_H$ is conservative. In the notation of \Cref{rec:geometricfixedpoints}, $\Phi_H = \tn{coind}^G_{N_G(H)} \circ \iota \circ L \circ \tn{infl}_{\wgh}^{N_G(H)}$. The composition $L \circ \tn{infl}_{\wgh}^{N_G(H)}$ is conservative as it is an equivalence, the functor $\iota: \spec^{N_G(H)}[A_H^{-1}] \hookrightarrow \spec^{N_G(H)}$ is conservative as it is fully faithful, and $\tn{coind}^G_{N_G(H)}$ is conservative by the proof of \cite[Theorem 5.32]{mnn17}, where we use that $G$ is finite. \\
Thus we obtain an induced adjunction $(\overl{\Phi}^H,\overl{\Phi}_H): \Mod_{\Phi_H(\bS)}(\specg) \rightleftarrows \specwgh$ which is an inverse equivalence of symmetric monoidal $\infty$-categories.
\end{proof}

\begin{cor}\label{cor:weylspectramodulesanyring}
For any $A \in \calg(\specwgh)$, there is an equivalence of symmetric monoidal $\infty$-categories 
$$\overl{\Phi_HU_A}: \Mod_A(\specwgh) \stackrel{\sim}{\longrightarrow} \Mod_{\Phi_H(A)}(\specg).$$
The equivalence functor makes the following diagram commute:
\[\begin{tikzcd}
	&& {\Mod_{\Phi_H(A)}(\specg)} \\
	{\Mod_A(\specwgh)} & {\specwgh} & {\specg.}
	\arrow["{U_{\Phi_H(A)}}", from=1-3, to=2-3]
	\arrow["{\overl{\Phi_HU_A}}", from=2-1, to=1-3]
	\arrow["{U_A}"', from=2-1, to=2-2]
    \arrow["{\Phi_H}"', from=2-2, to=2-3]
\end{tikzcd}\]
In particular, $F_A\Phi^H: \specg \to \Mod_A(\specwgh)$ under the above equivalence identifies with the extension of scalars $\specg \to \Mod_{\Phi_H(A)}(\specg)$ along $\bS \to \Phi_H(A)$.
\end{cor}
\begin{proof}
By \Cref{lem:weylspectramodules} there is a symmetric monoidal equivalence 
$$\Mod_A(\specwgh) \simeq \Mod_{\overl{\Phi}_H(A)}\Mod_{\Phi_H(\bS)}(\specg),$$ and by \cite[Corollary 3.4.1.9]{lur17} there is a further change of algebra equivalence 
$$\Mod_{\overl{\Phi}_H(A)}\Mod_{\Phi_H(\bS)}(\specg) \simeq \Mod_{\Phi_H(A)}(\specg).$$
Alternatively, we can again apply the symmetric monoidal Barr-Beck-Lurie theorem \cite[Proposition 5.29]{mnn17}, but now to the adjunction $$(F_A\Phi^H,\Phi_HU_A): \specg \rightleftarrows \Mod_A(\specwgh).$$
The composition $F_A\Phi^H$ is a symmetric monoidal left adjoint. Both categories are rigidly compactly generated and presentably symmetric monoidal, so as in the previous proof the projection formula holds. Again by the proof of the previous lemma $\Phi_H$ preserves colimits and is conservative, and the same holds for $U_A$.
\end{proof}

\begin{re}\label{re:weylspectramodulesanyringfunctor}
The inverse to the equivalence functor $\overl{\Phi_HU_A}$ by \cite[Section 5.3]{mnn17} is the composition
\begin{equation}\label{tik:compone}
\begin{tikzcd}[column sep=0.6cm]
	{\Mod_{\Phi_H(A)}(\specg)} & {\Mod_{F_A\Phi^H\Phi_H(A)}\Mod_A(\specwgh)} & {\Mod_{A}(\specwgh),}
	\arrow["{F_A\Phi^H}", from=1-1, to=1-2]
	\arrow["{\varepsilon'_{A,!}}"{pos=0.4}, from=1-2, to=1-3]
\end{tikzcd}
\end{equation}
where $\varepsilon'_A: F_A\Phi^H\Phi_H(A) \to A$ is the counit of $F_A\Phi^H \dashv \Phi_HU_A$, and $F_A$ is the free functor $\specwgh \to \Mod_A(\specwgh)$. By \cite[Corollary 3.4.1.9]{lur17} the nested module category in the middle is equivalent to $\Mod_{\Phi^H\Phi_H(A) \otimes A}(\specwgh)$, and this identifies the composition \ref{tik:compone} with 
\begin{equation}\label{tik:comptwo}
\begin{tikzcd}
	{\Mod_{\Phi_H(A)}(\specg)} & {\Mod_{\Phi^H\Phi_H(A)}(\specwgh)} & {\Mod_{A}(\specwgh),}
	\arrow["{\Phi^H}", from=1-1, to=1-2]
	\arrow["{\varepsilon_{A,!}}"{pos=0.4}, from=1-2, to=1-3]
\end{tikzcd}
\end{equation}
where $\varepsilon_A$ is the counit map $\Phi^H\Phi_H(A) \to A$ of $\Phi^H \dashv \Phi_H$. As \ref{tik:compone} is the inverse of $\overl{\Phi_HU_A}: \Mod_A(\specwgh) \to \Mod_{\Phi_H(A)}(\specg)$, we can also express $\overl{\Phi_HU_A}$ as the inverse of \ref{tik:comptwo}, namely as
\[\begin{tikzcd}
	{\Mod_A(\specwgh)} & {\Mod_{\Phi^H\Phi_H(A)}(\specwgh)} & {\Mod_{\Phi_H(A)}(\specg),}
	\arrow["{\varepsilon_A^*}", from=1-1, to=1-2]
	\arrow["{\Phi_H}", from=1-2, to=1-3]
\end{tikzcd}\]
where $\varepsilon_A^*$ is the restriction of scalars along $\varepsilon_A: \Phi^H\Phi_H(A) \to A$.
\end{re}

\begin{re}\label{re:anotherapproach}
If $\cC^\otimes$ is a presentably symmetric monoidal $\infty$-category, then by \cite[Theorem 4.8.5.11]{lur17} there is a fully faithful functor 
\begin{align*}
\calg(\cC) &\hookrightarrow \Mod_{\cC}(\prl)_{\cC/} \\
A &\mapsto \Mod_A(\cC) \\
(f: A \to B) &\mapsto (f_!: \Mod_A(\cC) \to \Mod_B(\cC)),
\end{align*}
where $f_!: \Mod_A(\cC) \to \Mod_B(\cC)$ is the extension of scalars functor $M \mapsto M \otimes_A B$. In particular, any symmetric monoidal left adjoint $\Mod_A(\cC) \to \Mod_B(\cC)$ which commutes with the free functors is the extension of scalars along a map of algebra objects $A \to B$. By \Cref{cor:weylspectramodulesanyring} any functor $\Mod_{\nH \underl{R}}(\specg) \to \Mod_{\nH \underl{R}}(\specwgh)$ is equivalent to a functor of module categories of $G$-spectra
$$\Mod_{\nH \underl{R}}(\specg) \to \Mod_{\Phi_H(\nH \underl{R})}(\specg),$$
so if such a functor is a symmetric monoidal left adjoint which commutes with the free functors, it is induced by a ring map $\nH \underl{R} \to \Phi_H(\nH \underl{R})$. By adjunction, giving such a ring map is equivalent to giving a ring map $\Phi^H(\nH \underl{R}) \to \nH \underl{R} \simeq \iota_{t\leq 0} (\nH \underl{R})$, which in turn is equivalent to giving a ring map 
$$\tau_{t \leq 0} (\Phi^H(\nH \underl{R})) \to \nH \underl{R}.$$
As $\Phi^H(\nH \underl{R})$ is connective by \Cref{re:geometricfixedpointsremedy}, this corresponds to a ring map 
$$\pi_0 (\Phi^H(\nH \underl{R})) \to \underl{R}$$
in the t-heart $(\specg)^{t \heart} \simeq \mackg$. Hence, once we computed the Mackey functor $\pi_0 (\Phi^H(\nH \underl{R}))$, we have - in addition to the construction of \Cref{re:geometricfixedpointsremedy} - another candidate for the equivariant modular fixed point functor: an extension of scalars along a ring map $\nH \underl{R} \to \Phi_H(\nH \underl{R})$ corresponding to a ring map $\pi_0 (\Phi^H(\nH \underl{R})) \to \underl{R}$ under the process described above. We will show that under our assumptions both of these approaches agree.
\end{re}

\begin{lem}\label{lem:geometricfixedpointsofem}
Let $p$ be a prime, let $R$ be a discrete ring in which $p=0$, and let $H$ be a $p$-subgroup of $G$. Then $\pi_0 (\Phi^H(\nH \underl{R})) \cong \underl{R}$ as Green functors.
\end{lem}
\begin{proof}
If $H=1$ then $\Phi^H$ is identic, so let us assume that $H$ is nontrivial. As we saw in \Cref{re:geometricfixedpointsremedy} the group $\pi_0^{K/H} (\Phi^H(\nH \underl{R}))$ for $K/H \leq \wgh$ is isomorphic to the Bredon homology group 
$$H_0^{K}(\tilde{E}\cF_H;\underl{R}) \cong \tn{coker}(d_1: C^{\tn{cell}}_1(\tilde{E}\cF_H;\underl{R}) \to C^{\tn{cell}}_0(\tilde{E}\cF_H;\underl{R}))$$
of the $K$-space $\tilde{E}\cF_H$. By \cite[Theorem 1.9]{lüc05} there is a model for $E\cF_H$ which for each $L \leq K$ with $H \nleq L$ has an $n$-cell $D^n \times K/L$. The unreduced suspension $\tilde{E}\cF_H \simeq S^0 \ast E\cF_H$ then has one pointed $0$-cell $K/K_+$, and pointed $1$-cells $D^n \times K/L_+$ for each $H \nleq L \leq G$. It follows that
$$H_0^{K}(\tilde{E}\cF_H;\underl{R}) \cong \tn{coker} \left(\oplus \tn{ind}_L^{K}: \bigoplus_{H \nleq L \leq K} \underl{R}(K/L) \to  \underl{R}(K/K) \right).$$
An alternative way to arrive at this description is provided by geometric fixed points on Mackey functors. We refer to \cite[Section 5]{hmq23}, \cite[Section 5.2]{bghl19} and \cite[Section 2]{tw95}. \\
The maps $\tn{ind}_L^{K}$ are multiplication by $[K : L]$. We show that all of these indices are divisible by $p$ and hence vanish in $R$. Let $L$ be a subgroup of $K$ which does not contain $H$. Write $|H|=p^n$, $n \geq 1$, $|K| = p^k \cdot a$, $p \nmid a$, $|L| = p^l \cdot b$, $p \nmid b$. As $H \leq K$ we have $k \geq 1$. Then $[K : L] = p^{k-l}\cdot \frac{a}{b}$, and we need to show that $k > l$. Assume $l=k$. Let $S \leq L$ be a Sylow $p$-subgroup of $L$, then $|S| = p^l$ and $l=k$ implies $S \in \tn{Syl}_p(K)$. As $H$ is a $p$-subgroup of $K$ it is contained in a Sylow $p$-subgroup of $K$, so in $S^x$ for some $x \in K$. But then $H = H^{x^{-1}} \leq S$ as $x^{-1} \in K \leq N_G(H)$ and hence $H \leq S \leq L$, which is a contradiction. So $k > l$ and all of the indices $[K : L]$ are zero in $R$. It follows that $H_0^{K}(\tilde{E}\cF_H;\underl{R}) \cong R$.
\end{proof}

\begin{re}[cf. {\cite[Warning 5.1]{bg25a}}]
If $R=k$ a field of characteristic $p$, then $\pi_0^{K/H} (\Phi^H(\nH \underl{k}))$ will either be $k$ or $0$, depending on the invertibility of one of the indices $[K : L]$ in $k$, and being invertible is then equivalent to not being divisible by $p$. In the case that $H$ is not a $p$-subgroup, let $K = N_G(H)$ and choose a Sylow $p$-subgroup $P \leq N_G(H)$ (which can't contain $H$). Then $[N_G(H):P]$ is invertible in $k$, and hence $\pi_0^{\wgh} (\Phi^H(\nH \underl{R}))$ vanishes. This illustrates that the assumption that $H$ is a $p$-group is crucial to ensure that none of the sections of $\pi_0 (\Phi^H(\nH \underl{k}))$ vanishes, which we need to define the equivariant modular fixed point functor.
\end{re}

For the rest of this section we fix a prime $p$ and a discrete ring $R$ in which $p=0$. We are now ready to define the equivariant modular fixed point functor, following the approach sketched in \Cref{re:geometricfixedpointsremedy}.

\begin{defi}\label{defi:eqmodularfixedpoints}
For a $p$-subgroup $H \leq G$, we define the \emph{equivariant modular fixed point functor}
$$\Psi^H = \Psi^{H;G}: \Mod_{\nH \underl{R}}(\specg) \to \Mod_{\nH \underl{R}}(\specwgh)$$
as the composition
\[\begin{tikzcd}
	{\Mod_{\nH \underl{R}}(\spec^{G})} & {\Mod_{\Phi^H(\nH \underl{R})}(\specwgh)} & {\Mod_{\nH \underl{R}}(\specwgh),}
	\arrow["{\Phi^{H}}", from=1-1, to=1-2]
	\arrow["{\psi_{H,!}}", from=1-2, to=1-3]
\end{tikzcd}\]
where the first functor is the functor induced by geometric fixed points on module categories and the second functor is an extension of scalars along the ring map $\psi_{H;G} = \psi_H: \Phi^H(\nH \underl{R}) \to \nH \underl{R}$ which under the adjunction $(\tau_{t \leq 0} \dashv \iota_{t \leq 0})$ corresponds to the identic map of Mackey functors $\pi_0(\Phi^H(\nH \underl{R})) \cong \underl{R} \to \underl{R}$. \\
As a composition of such $\Psi^H$ is a symmetric monoidal left adjoint, and it has a right adjoint $\Psi_H = \Psi_{H;G} = \Phi_{H} \circ \psi_{H}^*: \Mod_{\nH \underl{R}}(\specwgh) \to \Mod_{\nH \underl{R}}(\spec^{G})$. 
\end{defi}

\begin{re}
We are choosing a map of Green functors $\underl{R} \to \underl{R}$, which is the same as a ring map $R \to R$, since the restriction maps of the constant Mackey functor are identic. Many properties in the remainder of this section will hold for any choice of ring map $\Phi^H(\nH \underl{R}) \to \nH \underl{R}$. In \Cref{thm:identifymodfix} we will identify the just defined equivariant modular fixed point functor with the one on derived permutation modules, and for this the choice of ring map will be crucial. In \Cref{prop:eqmodfixhrlinear} we will also see that once we fix the $\nH R$-linear structure on the $\infty$-categories of modules over $\nH \underl{R}$ to come from the identity on $R$ (see \Cref{re:hrlinearstructure}), we're left with no choice for the ring map $\psi_H: \Phi^H(\nH \underl{R}) \to \nH \underl{R}$ if we want the functor $\Psi^H$ to be $\nH R$-linear.
\end{re}

\begin{re}
The construction of the equivariant modular fixed point functor on $\nH \underl{R}$-modules comes in the same fashion as the construction of the $\nH R$-linear structure and the restriction along a group homomorphism on these - we employ the functors induced on module categories by the usual symmetric monoidal left adjoints available on equivariant spectra, but we also postcompose with a base change. This ensures that the `$R$-linear equivariant stable homotopy theory' we're doing is happening in $\calg(\prlst)_{\Mod_{\nH R}(\spec)/}$, that is, under $\Mod_{\nH R}(\spec)$, see \Cref{re:hrlinearstructure}, \Cref{re:restrictionunderhr} and also \Cref{prop:eqmodfixhrlinear} below.
\end{re}

\begin{re}
Another way to think about the ring map $\psi_H: \Phi^H(\nH \underl{R}) \to \nH \underl{R}$ is the following one: the map $S^0 \to \tilde{E}\cF_H$ induces a map 
$$\nH \underl{R} \simeq (\nH \underl{R})^H \to (\nH \underl{R} \otimes \Sigma^{\infty}\tilde{E}\cF_H)^H = \Phi^H(\nH \underl{R}),$$
and \Cref{lem:geometricfixedpointsofem} shows that it is an isomorphism on $\pi_0$, in fact under the choices of isomorphisms we make it is identic. The inverse of this isomorphism then using t-structure arguments gives the map $\psi_H$.
\end{re}

As desired, the functor $\Psi^H$ inherits a key property from geometric fixed points: fixed points of a free module on a $G$-space are the free module on the fixed points.

\begin{lem}\label{lem:eqmodfixfree}
For $H \leq G$ a $p$-subgroup, there is a commutative square
\[\begin{tikzcd}
	\specg & \specwgh \\
	{\Mod_{\nH \underl{R}}(\spec^{G})} & {\Mod_{\nH \underl{R}}(\specwgh).}
	\arrow["{\Phi^H}", from=1-1, to=1-2]
	\arrow["{F_{\nH \underl{R}}}"', from=1-1, to=2-1]
	\arrow["{F_{\nH \underl{R}}}", from=1-2, to=2-2]
	\arrow["{\Psi^H}", from=2-1, to=2-2]
\end{tikzcd}\]
In particular, $\Psi^H(\nH \underl{R} \otimes \Sigma^{\infty}X) \simeq \nH \underl{R} \otimes \Sigma^{\infty} X^H$ for any $X \in \spcgp$.
\end{lem}
\begin{proof}
This follows immediately from \Cref{rec:modules} and the fact that precomposing an extension of scalars with a free functor gives a free functor.
\end{proof}

\begin{cor}\label{cor:weightexact}
Let $H \leq G$ be a $p$-subgroup. Then the restriction to compact parts
$$\Psi^H: \Mod_{\nH \underl{R}}^{\omega}(\specg) \to \Mod_{\nH \underl{R}}^{\omega}(\specwgh)$$
is weight exact, where we consider the weight structures of \Cref{prop:weightonhrmod}.
\end{cor}
\begin{proof}
Since both weight structures are bounded it suffices by \cite[Proposition 3.3]{sos19} to show that the functor restricts to an additive functor between the weight hearts. This follows from \Cref{lem:eqmodfixfree}.
\end{proof}

The following proposition will identify the just defined functor with the composition proposed in \Cref{re:anotherapproach}, namely base changing from $\nH \underl{R}$ to $\Phi_H(\nH \underl{R})$ and then using the equivalence of \Cref{cor:weylspectramodulesanyring}.

\begin{prop}\label{cor:eqmodfixequivalent}
The identic map of Mackey functors $\pi_0(\Phi^H(\nH \underl{R})) \cong \underl{R} \to \underl{R}$ under the adjunction $(\tau_{t \leq 0} \dashv \iota_{t \leq 0})$ gives adjoint ring maps $\psi_H: \Phi^H(\nH \underl{R}) \to \nH \underl{R}$ and $\varphi_H: \nH \underl{R} \to \Phi_H(\nH \underl{R})$ in $\calg(\specwgh)$ and $\calg(\specg)$ respectively. For these, geometric fixed points and extension of scalars provide a commutative square 
\begin{equation}\label{diag:equimodfixedpointsquare}
\begin{tikzcd}
	{\Mod_{\nH \underl{R}}(\specg)} & {\Mod_{\Phi^H(\nH \underl{R})}(\specwgh)} \\
	{\Mod_{\Phi_H(\nH \underl{R})}(\specg)} & {\Mod_{\nH \underl{R}}(\specwgh).}
	\arrow["{\Phi^H}", from=1-1, to=1-2]
	\arrow["{\varphi_{H,!}}"', from=1-1, to=2-1]
	\arrow["{\psi_{H,!}}", from=1-2, to=2-2]
    \arrow["\stackrel{\ref{cor:weylspectramodulesanyring}}{\sim}", from=2-1, to=2-2]
\end{tikzcd}
\end{equation}
In particular, the modular fixed point functor $\Psi^H: \Mod_{\nH \underl{R}}(\specg) \to \Mod_{\nH \underl{R}}(\specwgh)$ is equivalent to the composition
\[\begin{tikzcd}
	{\Mod_{\nH \underl{R}}(\specg)} & {\Mod_{\Phi_H(\nH \underl{R})}(\specg) \stackrel{\ref{cor:weylspectramodulesanyring}}{\simeq} \Mod_{\nH \underl{R}}(\specwgh).}
	\arrow["{\varphi_{H,!}}", from=1-1, to=1-2]
\end{tikzcd}\]
\end{prop}
\begin{proof}
By \Cref{re:weylspectramodulesanyringfunctor} the bottom arrow is equivalent to the functor $\varepsilon_{\nH \underl{R},!}\Phi^H$, where $\varepsilon_{\nH \underl{R}}: \Phi^H\Phi_H(\nH \underl{R}) \to \nH \underl{R}$ is the counit of the adjunction $\Phi^H \dashv \Phi_H$. Precomposing both compositions in \ref{diag:equimodfixedpointsquare} with the free functor $\specg \to \Mod_{\nH \underl{R}}(\specg)$ gives equivalences $\varepsilon_{\nH \underl{R},!}\Phi^H\varphi_{H,!} F_{\nH \underl{R}} \simeq \varepsilon_{\nH \underl{R},!}\Phi^H F_{\Phi_H(\nH \underl{R})} \simeq \varepsilon_{\nH \underl{R},!} F_{\Phi^H\Phi_H(\nH \underl{R})} \Phi^H \simeq F_{\nH \underl{R}} \Phi^H \simeq \psi_{H,!} \Phi^H F_{\nH \underl{R}}$, where we use \Cref{lem:eqmodfixfree}, \Cref{rec:modules} and \Cref{re:anotherapproach}. This shows that both compositions give $\Mod_{\nH \underl{R}}(\specwgh)$ the same $\specg$-algebra structure. As the bottom arrow is an equivalence, this implies that both $\specg$-algebra structures on $\Mod_{\Phi_H(\nH \underl{R})}(\specg)$ available in \ref{diag:equimodfixedpointsquare} agree as well. \\
So by \cite[Theorem 4.8.5.11]{lur17} the two functors $\Mod_{\nH \underl{R}}(\specg) \to \Mod_{\Phi_H(\nH \underl{R})}(\specg)$ agree if and only if they are equivalent to extensions of scalars along equivalent ring maps $\nH \underl{R} \to \Phi_H(\nH \underl{R})$. The functor $\varphi_{H,!}$ is an extension of scalars along $\varphi_H$, so it remains to show that the composition $\Phi_H\varepsilon_{\nH \underl{R},*}\psi_{H,!}\Phi^H$ is equivalent to an extension of scalars along $\varphi_H$ as well. The ring map it corresponds to is the unit map $\nH \underl{R} \to \Phi_H(\nH \underl{R})$ of the adjunction $(\Phi_H\varepsilon_{\nH \underl{R},*}\psi_{H,!}\Phi^H \dashv \Phi_H\psi_H^*\varepsilon_{\nH \underl{R},!}\Phi^H)$. The unit of the adjunction $(\psi_{H,!}\Phi^H \dashv \Phi_H\psi_H^*)$ is given by the composition
\[\begin{tikzcd}
	{\nH \underl{R}} & {\Phi_H\Phi^H(\nH \underl{R})} &&& {\Phi_H\psi_H^*\psi_{H,!}\Phi^H(\nH \underl{R}) \simeq \Phi_H(\nH \underl{R}),}
	\arrow["{\eta_{\nH \underl{R}}}", from=1-1, to=1-2]
	\arrow["{\Phi_H(\bar{\eta}_{\Phi^H(\nH \underl{R})})}", from=1-2, to=1-5]
\end{tikzcd}\]
where $\eta$ is the unit of $(\Phi^H \dashv \Phi_H)$ and $\bar{\eta}$ is the unit of $(\psi_{H,!} \dashv \psi_H^*)$. The evaluation of $\bar{\eta}$ on $\Phi^H(\nH \underl{R})$ is just $\psi_H$. So as $\psi_H$ is adjoint to $\varphi_H$, the above composition is already equivalent to $\varphi_H$. Since the adjunction $(\Phi_H\varepsilon_{\nH \underl{R},*} \dashv \varepsilon_{\nH \underl{R},!}\Phi^H)$ is an equivalence, its unit is an equivalence. So postcomposing the above ring map with an equivalence will give a ring map equivalent to $\varphi_H$. This finishes the proof.
\end{proof}

We will continue by proving that the functor $\Psi^H$ is $\nH R$-linear by showing the stronger statement that it is a morphism in $\calg(\prlst)_{\Mod_{\nH R}(\spec)/}$. We then continue by establishing basic properties about nested modular fixed points and naturality with respect to restriction along a group homomorphism.

\begin{lem}\label{lem:hrlinearstructureweylgroup}
Let $H \leq G$ be a $p$-subgroup and let $\xi_{\wgh}: \nH R_G \to \Phi_H(\nH \underl{R})$ be the ring map which corresponds to the identity on $R$ under the chain of adjunctions $(\nH R_G \to \Phi_H(\nH \underl{R})) \Leftrightarrow (\nH R_{\wgh} \simeq \Phi^H(\nH R_G) \to \nH \underl{R}) \Leftrightarrow (\nH R \to (\nH \underl{R})^{\wgh} \simeq \nH R) \Leftrightarrow (R \to R)$. Then the $\nH R$-linear structure on $\Mod_{\nH \underl{R}}(\specwgh)$ obtained from the symmetric monoidal composition
$$\Mod_{\nH R}(\spec) \stackrel{\tn{infl}^G_{1}}{\longrightarrow} \Mod_{\nH R_G}(\specg) \stackrel{\xi_{\wgh,!}}{\longrightarrow} \Mod_{\Phi_H(\nH \underl{R})}(\specg) \stackrel{\ref{cor:weylspectramodulesanyring}}{\simeq} \Mod_{\nH \underl{R}}(\specwgh)$$
agrees with the one from \Cref{re:hrlinearstructure}.
\end{lem}
\begin{proof}
Recall that the $\nH R$-linear structure on $\Mod_{\nH \underl{R}}(\specwgh)$ from \Cref{re:hrlinearstructure} is obtained from the symmetric monoidal composition
\[\begin{tikzcd}
	{\Mod_{\nH R}(\spec)} & {\Mod_{\nH R_{\wgh}}(\spec^{\wgh})} & {\Mod_{\nH \underl{R}}(\spec^{\wgh}),}
	\arrow["{\tn{infl}^{\wgh}_{1}}", from=1-1, to=1-2]
	\arrow["{\zeta_{\wgh,!}}", from=1-2, to=1-3]
\end{tikzcd}\]
where $\zeta_{\wgh}: \nH R_{\wgh} \to \nH \underl{R}$ is the ring map $\xi_{\wgh}$ corresponds to. The assertion thus follows from the description of the equivalence functor from \Cref{cor:weylspectramodulesanyring} given in \Cref{re:weylspectramodulesanyringfunctor}, the functoriality in \cite[Theorem 4.8.5.11]{lur17} and the fact that $\Phi^H \tn{infl}^G_1 \simeq \tn{infl}^{\wgh}_1$, see \Cref{rec:geometricfixedpoints}.
\end{proof}

\begin{prop}\label{prop:eqmodfixhrlinear}
Let $H \leq G$ be a $p$-subgroup. Using the maps from \Cref{re:hrlinearstructure}, the equivariant modular fixed point functor $\Psi^H: \Mod_{\nH \underl{R}}(\specg) \to \Mod_{\nH \underl{R}}(\specwgh)$ is a morphism under $\Mod_{\nH R}(\spec)$. In particular, it is $\nH R$-linear.
\end{prop}
\begin{proof}
We show that the diagram
\[\begin{tikzcd}
	& {\Mod_{\nH R}(\spec)} \\
	{\Mod_{\nH R_G}(\specg)} && {\Mod_{\nH R_{\wgh}}(\specwgh)} \\
	{\Mod_{\nH \underl{R}}(\specg)} && {\Mod_{\nH \underl{R}}(\specwgh)}
	\arrow["{\tn{infl}^G_1}"', from=1-2, to=2-1]
	\arrow["{\tn{infl}^{\wgh}_1}", from=1-2, to=2-3]
	\arrow["{\zeta_{G,!}}"', from=2-1, to=3-1]
	\arrow["{\zeta_{\wgh,!}}", from=2-3, to=3-3]
	\arrow["{\Psi^H}", from=3-1, to=3-3]
\end{tikzcd}\]
commutes. By \cite[Theorem 4.8.5.11]{lur17} there is a fully faithful functor 
$$\calg(\spec^G)_{\nH R_G/} \hookrightarrow \calg(\prl)_{\Mod_{\nH R_G}(\specg)/}$$
which sends a commutative algebra object $A$ under $\nH R_G$ to the module category $\Mod_A(\spec^G)$ under $\Mod_{\nH R_G}(\specg)$. So using \Cref{lem:hrlinearstructureweylgroup} and \Cref{cor:eqmodfixequivalent} it suffices to check that the triangle of ring maps
\[\begin{tikzcd}
	& {\nH R_G} \\
	{\nH \underl{R}} && {\Phi_H(\nH \underl{R})}
	\arrow["{\zeta_G}"', from=1-2, to=2-1]
	\arrow["\xi_{\wgh}", from=1-2, to=2-3]
	\arrow["\varphi_H", from=2-1, to=2-3]
\end{tikzcd}\]
commutes, where we use the notation from \Cref{lem:hrlinearstructureweylgroup} and \Cref{cor:eqmodfixequivalent}. Using the adjunction $(\Phi^H \dashv \Phi_H)$, this is equivalent to the commutativity of the following diagram of ring maps:
\[\begin{tikzcd}
	{\Phi^H(\nH R_G)\simeq \nH R_{\wgh}} \\
	{\Phi^H(\nH \underl{R})} & {\Phi^H\Phi_H(\nH \underl{R})} & {\nH \underl{R}.}
	\arrow["{\Phi^H(\zeta_G)}"', from=1-1, to=2-1]
	\arrow["{\zeta_{\wgh}}", shift left=3, from=1-1, to=2-3]
	\arrow["{\Phi^H(\varphi_H)}", from=2-1, to=2-2]
	\arrow["\varepsilon_{\nH \underl{R}}", from=2-2, to=2-3]
\end{tikzcd}\]
Here $\varepsilon_{\nH \underl{R}}: \Phi^H\Phi_H(\nH \underl{R}) \to \nH \underl{R}$ is the counit of $(\Phi^H \dashv \Phi_H)$ and hence the lower horizontal composite using the notation of \Cref{defi:eqmodularfixedpoints} and \Cref{cor:eqmodfixequivalent} is given by $\psi_H$. Since categorical fixed points $(-)^{\wgh}: \specwgh \to \spec$ are t-exact, its left adjoint $\tn{infl}_1^{\wgh}: \spec \to \specwgh$ is right t-exact and hence preserves connective parts. So a map $\nH R_{\wgh} \to \nH \underl{R}$ is determined on $\pi_0$. By definition of $\psi_H$ the lower horizontal map is identic on $\pi_0$, so it suffices to show that $\Phi^H(\zeta_G)$ and $\zeta_{\wgh}$ identify after passing to $\pi_0$. By \cite[Section 5.2]{gs14} we have that $\pi_0(\nH R_{\wgh}) \cong \bA \otimes R \cong \bA_R$, and the map $\pi_0(\zeta_{\wgh}): \pi_0(\nH R_{\wgh}) \to \pi_0(\nH \underl{R})$ identifies with the unit map $\bA_R \to \underl{R}$ of $\mack_R(\wgh)$. To identify $\Phi^H(\zeta_G)$ on $\pi_0$, note that the map $S^0 \to \tilde{E}\cF_H$ induces a commutative diagram
\[\begin{tikzcd}
	{(\nH R_G)^H} && {(\nH \underl{R})^H} \\
	{(\nH R_G \otimes \Sigma^{\infty}\tilde{E}\cF_H)^H} && {(\nH \underl{R} \otimes \Sigma^{\infty}\tilde{E}\cF_H)^H}
	\arrow["{(\zeta_G)^H}", from=1-1, to=1-3]
	\arrow[from=1-1, to=2-1]
	\arrow[from=1-3, to=2-3]
	\arrow["{(\zeta_G \otimes \tn{id})^H}", from=2-1, to=2-3]
\end{tikzcd}\]
where the lower horizontal map is given by $\Phi^H(\zeta_G)$. By definition of $\zeta_G$ the top horizontal map on $\pi_0$ is the unit map $\bA_R \to \underl{R}$ in $\mack_R(\wgh)$, hence it is $\zeta_{\wgh}$. Now by \Cref{lem:geometricfixedpointsofem} and \Cref{rec:geometricfixedpoints} the two vertical maps induce isomorphisms on $\pi_0$, which finishes the proof.
\end{proof}

\begin{prop}\label{prop:eqmodfixnested}
Let $H \leq G$ be a $p$-subgroup and let $\bar{K} = K/H \leq \wgh$ be another $p$-subgroup, where $H \leq K \leq N_G(H)$. Then there is a natural inclusion
$$\weyl{(\wgh)}{(\bar{K})} = N_{\wgh}(\bar{K})/\bar{K} \hookrightarrow N_G(K)/K = \weyl{G}{K}$$
and the following diagram commutes:
\[\begin{tikzcd}
	{\Mod_{\nH \underl{R}}(\specg)} && {\Mod_{\nH \underl{R}}(\specwgh)} \\
	{\Mod_{\nH \underl{R}}(\spec^{\weyl{G}{K}})} && {\Mod_{\nH \underl{R}}(\spec^{\weyl{(\wgh)}{\bar{K}}}).}
	\arrow["{\Psi^{H;G}}", from=1-1, to=1-3]
	\arrow["{\Psi^{K;G}}"', from=1-1, to=2-1]
	\arrow["{\Psi^{\bar{K};\wgh}}", from=1-3, to=2-3]
	\arrow["{\tn{res}_{\weyl{(\wgh)}{(\bar{K})}}^{\weyl{G}{K}}}", from=2-1, to=2-3]
\end{tikzcd}\]
\end{prop}
\begin{proof}
By \Cref{rec:geometricfixedpoints} we have that $\Phi^{\bar{K};\wgh} \circ \Phi^{H;G} \simeq \tn{res}_{\weyl{(\wgh)}{(\bar{K})}}^{\weyl{G}{K}} \circ \Phi^{K;G}$, hence $\Phi^{\bar{K};\wgh}\Phi^{H;G}(\nH \underl{R}) \simeq \tn{res}_{\weyl{(\wgh)}{(\bar{K})}}^{\weyl{G}{K}}(\Phi^{K;G}(\nH \underl{R}))$ and the top left square in the following diagram commutes.
\[\begin{tikzcd}[scale cd=0.88]
	{\Mod_{\nH \underl{R}}(\specg)} & {\Mod_{\Phi^{H}(\nH \underl{R})}(\specwgh)} & {\Mod_{\nH \underl{R}}(\specwgh)} \\
	{\Mod_{\Phi^{K}(\nH \underl{R})}(\spec^{\weyl{G}{K}})} & {\Mod_{\Phi^{\bar{K}}\Phi^{H}(\nH \underl{R})}(\spec^{\weyl{(\wgh)}{(\bar{K})}}) \;} & {\; \Mod_{\Phi^{\bar{K}}(\nH \underl{R})}(\spec^{\weyl{(\wgh)}{(\bar{K})}})} \\
	{\Mod_{\nH \underl{R}}(\spec^{\weyl{G}{K}})} & {\Mod_{\tn{res}(\nH \underl{R})}(\spec^{\weyl{(\wgh)}{(\bar{K})}})} & {\Mod_{\nH \underl{R}}(\spec^{\weyl{(\wgh)}{(\bar{K})}})}
	\arrow["{\Phi^{H;G}}", from=1-1, to=1-2]
	\arrow["{\Phi^{K;G}}", from=1-1, to=2-1]
	\arrow["{\psi_{H,!}}", from=1-2, to=1-3]
	\arrow["{\Phi^{\bar{K};\wgh}}", from=1-2, to=2-2]
	\arrow["{\Phi^{\bar{K};\wgh}}", from=1-3, to=2-3]
	\arrow["{\tn{res}}", from=2-1, to=2-2]
	\arrow["{\psi_{K,!}}", from=2-1, to=3-1]
	\arrow["{\Phi^{\bar{K}}(\psi_{H})_!}", from=2-2, to=2-3]
	\arrow["{\psi_{\bar{K},!}}", from=2-3, to=3-3]
	\arrow["{\tn{res}}", from=3-1, to=3-2]
    \arrow["{\tn{res}}", from=3-1, to=3-2]
    \arrow["{\tn{res}(\psi_{K})_!}", from=2-2, to=3-2]
    \arrow["{\nu_{\weyl{(\wgh)}{(\bar{K})},!}^{\weyl{G}{K}}}", from=3-2, to=3-3]
    \arrow["\sim"', draw=none, from=3-2, to=3-3]
\end{tikzcd}\]
The functoriality in \cite[Theorem 4.8.5.11]{lur17} implies that the top right and the bottom left square commute. It remains to show that $\psi_{\bar{K}}\Phi^{\bar{K}}(\psi_H) \simeq \nu \, \tn{res}(\psi_K)$ as algebra maps 
$$\Phi^{\bar{K};\wgh}\Phi^{H;G}(\nH \underl{R}) \simeq \tn{res}_{\weyl{(\wgh)}{(\bar{K})}}^{\weyl{G}{K}}(\Phi^{K;G}(\nH \underl{R})) \to \nH \underl{R}.$$
Since the domain is connective and the target is in the t-heart, such a map is determined on $\pi_0$, where it is a map $\underl{R} \to \underl{R}$ by \Cref{lem:geometricfixedpointsofem}, and hence a map $R \to R$. Now $\psi_H$ is produced as a map of $R$-algebras on $\pi_0$, and $\psi_{\bar{K},!}\Phi^{\bar{K}}(\psi_{H}) \simeq \Psi^{\bar{K}}(\psi_{H})$ by definition of $\Psi^{\bar{K}}$, hence $\psi_{\bar{K},!}\Phi^{\bar{K}}(\psi_{H})_! \simeq \Psi^{\bar{K}}(\psi_{H})_!$. By \Cref{prop:eqmodfixhrlinear} $\Psi^{\bar{K}}$ is a functor under $\Mod_{\nH R}(\spec)$, hence $\Psi^{\bar{K}}(\psi_{H})_!$ corresponds to a map of $R$-algebras $R \to R$ on $\pi_0$. Using \Cref{cons:restrictionmodules} and \Cref{re:restrictionunderhr} the same reasoning applies to $\nu_{\weyl{(\wgh)}{(\bar{K})}}^{\weyl{G}{K}}\tn{res}_{\weyl{(\wgh)}{(\bar{K})}}^{\weyl{G}{K}}(\psi_{K})$. Since there is only one $R$-algebra map $R \to R$ the bottom right square commutes.
\end{proof}

\begin{cor}
Let $H \leq K \leq G$ be two $p$-subgroups with $H \trianglelefteq G$ normal. Then $\weyl{(G/H)}{(K/H)} = N_{G/H}(K/H)/(K/H) \cong \weyl{G}{K}$ and the following triangle commutes:
\[\begin{tikzcd}
	{\Mod_{\nH \underl{R}}(\specg)} & {\Mod_{\nH \underl{R}}(\spec^{G/H})} \\
	& {\Mod_{\nH \underl{R}}(\spec^{\weyl{G}{K}}).}
	\arrow["{\Psi^{H;G}}", from=1-1, to=1-2]
	\arrow["{\Psi^{K;G}}"', from=1-1, to=2-2]
	\arrow["{\Psi^{K/H;G/H}}", from=1-2, to=2-2]
\end{tikzcd}\]
\end{cor}
\begin{proof}
This follows from \Cref{prop:eqmodfixnested} as $\tn{res}^{\weyl{G}{K}}_{\weyl{(G/H)}{(K/H)}}$ is an equivalence.
\end{proof}

\begin{prop}\label{prop:eqmodfixrestriction}
Let $\alpha: G \to G'$ be a group homomorphism, $H \leq G$ a $p$-subgroup and set $H' \defeq \alpha(H)$, inducing a homomorphism $\overl{\alpha}: \wgh \to \weyl{G'}{H'}$. Then the following square commutes:
\[\begin{tikzcd}
	{\Mod_{\nH \underl{R}}(\spec^{G'})} & {\Mod_{\nH \underl{R}}(\spec^{G})} \\
	{\Mod_{\nH \underl{R}}(\spec^{\weyl{G'}{H'}})} & {\Mod_{\nH \underl{R}}(\specwgh).}
	\arrow["{\alpha^*}", from=1-1, to=1-2]
	\arrow["{\Psi^{H';G'}}"', from=1-1, to=2-1]
	\arrow["{\Psi^{H;G}}", from=1-2, to=2-2]
	\arrow["{\overl{\alpha}^*}", from=2-1, to=2-2]
\end{tikzcd}\]
\end{prop}
\begin{proof}
By \Cref{rec:geometricfixedpoints} we have $\Phi^{H;G}\alpha^* \simeq \overl{\alpha}^*\Phi^{H';G'}$, and in particular it follows that $\Phi^{H}\alpha^*(\nH \underl{R}) \simeq \overl{\alpha}^*\Phi^{H'}(\nH \underl{R})$. Consequently the top left square in the following diagram commutes:
\[\begin{tikzcd}
	{\Mod_{\nH \underl{R}}(\spec^{G'})} & {\Mod_{\alpha^*(\nH \underl{R})}(\spec^{G})} & {\Mod_{\nH \underl{R}}(\spec^{G})} \\
	{\Mod_{\Phi^{H'}(\nH \underl{R})}(\spec^{\weyl{G'}{H'}})} & {\Mod_{\Phi^H\alpha^*(\nH \underl{R})}(\specwgh) \;\;} & {\;\; \Mod_{\Phi^H(\nH \underl{R})}(\specwgh)} \\
	{\Mod_{\nH \underl{R}}(\spec^{\weyl{G'}{H'}})} & {\Mod_{\bar{\alpha}^*(\nH \underl{R})}(\specwgh)} & {\Mod_{\nH \underl{R}}(\specwgh)}
	\arrow["{\alpha^*}", from=1-1, to=1-2]
	\arrow["{\Phi^{H';G'}}", from=1-1, to=2-1]
	\arrow["{\nu(\alpha)_!}", from=1-2, to=1-3]
	\arrow["{\Phi^{H;G}}", from=1-2, to=2-2]
	\arrow["{\Phi^{H;G}}", from=1-3, to=2-3]
	\arrow["{\bar{\alpha}^*}", from=2-1, to=2-2]
	\arrow["{\psi_{H',!}}", from=2-1, to=3-1]
	\arrow["{\Phi^H(\nu(\alpha))_!}", from=2-2, to=2-3]
	\arrow["{\bar{\alpha}^*(\psi_{H'})_!}", from=2-2, to=3-2]
	\arrow["{\psi_{H,!}}", from=2-3, to=3-3]
	\arrow["{\bar{\alpha}^*}", from=3-1, to=3-2]
	\arrow["{\nu(\bar{\alpha})_!}", from=3-2, to=3-3]
\end{tikzcd}\]
The top right and lower left square commute by the functoriality in \cite[Theorem 4.8.5.11]{lur17}. For the lower right square we will argue as in the proof of \Cref{prop:eqmodfixnested}. We have to identify $\psi_{H,!}\Phi^H(\nu(\alpha))$ and $\nu(\bar{\alpha})_!\alpha^*(\psi_{H'})$ as algebra maps 
$$\Phi^{H}\alpha^*(\nH \underl{R}) \simeq \bar{\alpha}^*\Phi^{H'}(\nH \underl{R}) \to \nH \underl{R}.$$ Since the domain is connective and the target is in the t-heart, such a map is determined on $\pi_0$. The map of Mackey functors obtained in this way has target the constant Mackey functor, hence it is already determined on the sections for the trivial subgroup $\pi^1_0$. The maps $\nu(\alpha)$ and $\psi_{H'}$ are produced on $\pi_0^1$ as maps of $R$-algebras, see \Cref{cons:restrictionmodules} and \Cref{defi:eqmodularfixedpoints}. We have $\psi_{H,!}\Phi^H(\nu(\alpha))_! \simeq \Psi^H(\nu(\alpha))_{!}$ by definition of $\Psi^H$, and the latter is a functor under $\Mod_{\nH R}(\spec)$ by \Cref{prop:eqmodfixhrlinear}. Similar reasoning applies to $\nu(\bar{\alpha})_!\bar{\alpha}^*$ using \Cref{re:restrictionunderhr}. Thus both maps correspond to maps of $R$-algebras on $\pi^1_0$, and it suffices to show that $\pi^1_0(\bar{\alpha}^*\Phi^{H'}(\nH \underl{R}))$ is itself given by $R$, which only admits one $R$-algebra map to $R$. \\
When $\alpha$ is the inclusion of a subgroup $G \hookrightarrow G'$ the induced $\bar{\alpha}: \wgh \hookrightarrow \weyl{G'}{H}$ is an inclusion as well, hence $\pi^1_0(\tn{res}^{\weyl{G'}{H}}_{\wgh}\Phi^{H'}(\nH \underl{R})) \cong \pi^1_0(\Phi^{H'}(\nH \underl{R})) \cong R$. When $\alpha$ is a quotient map $G \twoheadrightarrow G/N$ then the induced homomorphism $\bar{\alpha}$ is the canonical quotient map $\wgh \twoheadrightarrow (\wgh)/(HN/N) \cong \weyl{(G/N)}{(HN/N)}$, so by \Cref{lem:pi0ofinflation} we have $\pi^1_0(\tn{infl}^{\wgh}_{\weyl{(G/N)}{(HN/N)}}\Phi^{H'}(\nH \underl{R})) \cong \pi^1_0(\Phi^{H'}(\nH \underl{R})) \cong R$, finishing the proof.
\end{proof}

\begin{cor}\label{cor:restrictiontonormalsubgroup}
Let $H \leq G$ be a $p$-subgroup. Then the following triangle commutes:
\[\begin{tikzcd}
	{\Mod_{\nH \underl{R}}(\spec^{G})} && {\Mod_{\nH \underl{R}}(\spec^{N_G(H)})} \\
	&& {\Mod_{\nH \underl{R}}(\spec^{\wgh}).}
	\arrow["{\tn{res}^G_{N_G(H)}}", from=1-1, to=1-3]
	\arrow["{\Psi^{H;G}}"', from=1-1, to=2-3]
	\arrow["{\Psi^{H;N_G(H)}}", from=1-3, to=2-3]
\end{tikzcd}\]
\end{cor}
\begin{proof}
This follows from \Cref{prop:eqmodfixrestriction} applied to $\alpha: N_G(H) \hookrightarrow G$, so $H'=H$ and $\weyl{N_G(H)}{H} = \weyl{G}{H}$.
\end{proof}

\begin{cor}
Let $N \trianglelefteq G$ be a normal $p$-subgroup. Then the composition
$$\Psi^N \circ \tn{infl}^G_{G/N}: \Mod_{\nH \underl{R}}(\spec^{G/N}) \to \Mod_{\nH \underl{R}}(\specg) \to \Mod_{\nH \underl{R}}(\spec^{G/N})$$ is equivalent to the identity.
\end{cor}
\begin{proof}
Apply \Cref{prop:eqmodfixrestriction} to $\alpha: G \twoheadrightarrow G/N = G'$ and $H=N$, so that $H'=1$ and $\weyl{G'}{H'} = G/N$.
\end{proof}

We will now extend the defined modular fixed point functors to profinite groups. One approach to do so for $H$ a closed pro-$p$-subgroup of some profinite group $G = \lim_i G_i = \lim_i G/N_i$ would be to mimic \Cref{defi:eqmodularfixedpoints}. This would mean defining it as the functor induced by the geometric $H$-fixed points on module categories, followed by an extension of scalars. Instead of arguing about the $\pi_0$ of geometric fixed points of $\nH \underl{R}$ in the profinite case we will use a colimit description of $\Mod_{\nH \underl{R}}(\specg)$ to formally extend the functor. \\
As argued in the proof of \Cref{thm:mackeyasspgmodules}, the equivalence $\specg \simeq \tn{colim}_i \; \spec^{G_i}$ induces an equivalence $\Mod_{\nH \underl{R}}(\specg) \simeq \tn{colim}_i \, \Mod_{\nH \underl{R}}(\spec^{G_i})$ for module categories, where the colimit is taken in $\calg(\prlst)$ along the maps induced on module categories by the restriction functors $f_{ij}^*: \spec^{G_j} \to \spec^{G_i}$ followed by a base change, as in \Cref{cons:restrictionmodules}. We provide some general framework to deal with the higher coherences involved when mapping out of this colimit. The construction of the natural transformation this requires is lengthy, but in a nutshell it says that modular fixed points highly coherently commute with restriction along a group homomorphism, which is an enhancement of \Cref{prop:eqmodfixrestriction}.

\begin{cons}\label{cons:eqmodfixasnattrans}
Write $\gppairs$ for the category of pairs $(G,H)$ with $G$ a finite group and $H \leq G$ a $p$-subgroup, a morphism $\alpha: (G,H) \to (G',H')$ is a group homomorphism $\alpha: G \to G'$ with $\alpha(H) = H'$. We denote by $\glo$ the global category of finite groups, group homomorphisms and conjugations. It is a (2,1)-category which we view as an $\infty$-category via the Duskin nerve. By \cite[Corollary 10.6, Theorem 11.10, Remark p. 1382]{lnp25} there is a functor 
$$\spec^{(-)}: \glo^\op \to \calg(\prlst), \; G \mapsto \specg,$$
called global category of equivariant spectra, and a symmetric monoidal equivalence 
\begin{equation}\label{eq:laxlim}
\spec_{\tn{Fin-gl}} \simeq \underset{G \in (\glo^\op)^{\dagger}}{\tn{laxlim}^{\dagger}} \; \specg
\end{equation}
between Schwedes's $\infty$-category of Fin-global spectra which is defined as the underlying $\infty$-category of the Fin-global model structure on orthogonal spectra from \cite[Theorem 4.3.17]{sch18} and the partially lax limit over $(\glo^\op)^{\dagger}$ of the functor $\spec^{(-)}$ as defined in \cite[Section 11]{lnp25}. We also refer to Sections 4 and 5 of \cite{lnp25} for a formal treatment of partially lax limits of (symmetric monoidal) $\infty$-categories. There are functors $\gppairs \to \glo$ which map 
\begin{align*}
&(G,H) \mapsto G, \; (\alpha: (G,H) \to (G',H')) \mapsto (\alpha: G \to G'),\\
&(G,H) \mapsto \wgh, \; (\alpha: (G,H) \to (G',H')) \mapsto (\bar{\alpha}: \wgh \to \weyl{G'}{H'}),
\end{align*}
where $\bar{\alpha}$ is induced by $\alpha$ on Weyl groups as in \Cref{prop:eqmodfixrestriction}. Composing with $\spec^{(-)}: \glo^\op \to \calg(\prlst)$ yields functors $\cM, \cN: \gppairs^\op \to \calg(\prlst)$ which map
\begin{align*}
&\cM: (G,H) \mapsto \specg, \; (\alpha: (G,H) \to (G',H')) \mapsto (\alpha^*: \spec^{G'} \to \specg), \\
&\cN: (G,H) \mapsto \specwgh, \; (\alpha: (G,H) \to (G',H')) \mapsto (\bar{\alpha}^*: \spec^{\weyl{G'}{H'}} \to \specwgh).
\end{align*}
The fact that geometric fixed points are natural with respect to restriction along a group homomorphism as in \Cref{rec:geometricfixedpoints} can be phrased highly coherently in terms of a natural transformation $\eta_{\tn{geo}}: \gppairs^\op \times \Delta^1 \to \calg(\prlst)$ between the above functors. This can be deduced by deriving an appropriate model categorical natural transformation \cite[Section V]{mm02} or by using the parametrised approach of Bachmann-Hoyois \cite{bh21}, see \cite[Section 6.B]{bbb24}. \\
Associated to the ring $R$ there exists a constant Fin-global Green functor $\underl{R}_{\tn{Fin-gl}}$ which can be promoted to an ultracommutative Fin-global Eilenberg-MacLane ring spectrum $\nH \underl{R}_{\tn{Fin-gl}} \in \calg(\spec_{\tn{Fin-gl}})$, for its construction we refer to \cite[Section 5.3]{sch18}. Under the equivalence \ref{eq:laxlim} it corresponds to the partially lax family of ring spectra which informally is given by the Eilenberg-MacLane spectrum $\nH \underl{R} \in \calg(\specg)$ for every $G \in \glo$ and compatible ring maps $\nu(\alpha): \alpha^*(\nH \underl{R}) \to \nH \underl{R}$ for every $\alpha: G \to G'$ which are equivalences when $\alpha$ is injective, as given in \Cref{cons:restrictionmodules} (and more structure which comes with higher coherences, see \cite[Section 1]{lnp25}). Restriction along the two functors $\gppairs \to \glo$ gives a (partially) lax family of ring spectra indexed by $\gppairs$, and hence an object $\cH \underl{R}$ in $\calg(\tn{laxlim} \, \cM)$ and $\calg(\tn{laxlim} \, \cN)$. The natural transformation $\eta_{\tn{geo}}$ induces a map 
$$\phi: \calg(\tn{laxlim} \, \cM) \to \calg(\tn{laxlim} \, \cN),$$
and it takes $\cH \underl{R}$ to $\phi(\cH \underl{R})$, the latter is informally given by the family
\begin{gather*}
\Phi^H(\nH \underl{R}) \in \calg(\specwgh) \tn{ for } (G,H) \in \gppairs,\\
\Phi^H(\nu(\alpha)): \bar{\alpha}^*\Phi^{H'}(\nH \underl{R}) \simeq \Phi^H\alpha^*(\nH \underl{R}) \to \Phi^H(\nH \underl{R}) \tn{ for } \alpha: (G,H) \to (G',H').
\end{gather*}
Given a functor $\cC_{\bullet}: \cI \to \calg(\prlst)$ and an object $S \in \calg(\tn{laxlim} \; \cC_{\bullet})$ which corresponds to a lax family of algebra objects $S_i \in \cC_i$, \cite[Theorem 5.10]{lnp25} produces a module functor $\Mod_{S_{\bullet}}(\cC_{\bullet}): \cI \to \calg(\prlst), \; i \mapsto \Mod_{S_i}(\cC_i)$. We claim that this construction is functorial in both natural transformations of functors $\cI \to \calg(\prlst)$ and morphisms of algebra objects in $\calg(\tn{laxlim} \; \cC_{\bullet})$. Indeed, using the notation of the proof of op. cit., if $\cC_{\bullet}$ corresponds to the $\cI^{\amalg}$-monoidal $\infty$-category $\cC^{\otimes} \to \cI^{\amalg}$ under symmetric monoidal unstraightening \cite[Corollary A.12]{dg22}, then the functor $\Mod_{S_{\bullet}}(\cC_{\bullet})$ is identified as the straightening of the $\cI^{\amalg}$-monoidal $\infty$-category 
$$\alg_{\mcal{CM}/\cI^{\amalg}}(\cC)^{\otimes} \times_{\alg_{\finstar/\cI^{\amalg}}(\cC)^{\otimes}} \cI^{\amalg} \to \cI^{\amalg}.$$
The $\infty$-operad $\mcal{CM}^{\otimes}$ is defined in \cite[Appendix]{lnp25}. The commutative algebra object $S \in \calg(\tn{laxlim} \; \cC_{\bullet})$ is considered as a section of $\alg_{\finstar/\cI^{\amalg}}(\cC)^{\otimes} \to \cI^{\amalg}$ in $\infoperads$, here $\alg_{\finstar/\cI^{\amalg}}(\cC)^{\otimes} \to \cI^{\amalg}$ is constructed from the fibration of $\infty$-operads $\cC^{\otimes} \to \cI^{\amalg}$ using \cite[Construction 3.2.4.1]{lur17}, see \cite[Construction 5.11]{lnp25}. Viewing $S$ as such a section uses an equivalence
\begin{equation}\label{eq:calglaxlim}
\calg(\tn{laxlim} \; \cC_{\bullet}) \simeq \alg_{\cI^{\amalg}}(\alg_{\finstar/\cI^{\amalg}}(\cC)^{\otimes})
\end{equation}
which is deduced from \cite[Proposition 5.8]{lnp25}. If we are given a natural transformation $\cI \times \Delta^1 \to \calg(\prlst)$ between functors $\cC_{\bullet}, \cD_{\bullet}: \cI \to \calg(\prlst)$, it unstraightens to an $\cI^{\amalg}$-monoidal functor $\cC^{\otimes} \to \cD^{\otimes}$. A map $S \to S'$ in $\calg(\tn{laxlim} \; \cC_{\bullet})$ under the equivalence \ref{eq:calglaxlim} corresponds to a natural transformation of sections of the map $\alg_{\finstar/\cI^{\amalg}}(\cC)^{\otimes} \to \cI^{\amalg}$. Hence in both cases we obtain maps of $\cI^{\amalg}$-monoidal $\infty$-categories which are expressed as fibre products, and they straighten to natural transformations $\Mod_{S_{\bullet}}(\cC_{\bullet}) \Rightarrow \Mod_{S_{\bullet}}(\cD_{\bullet})$ and $\Mod_{S_{\bullet}}(\cC_{\bullet}) \Rightarrow \Mod_{R_{\bullet}}(\cC_{\bullet})$. \\
In our case, we obtain module functors 
$\cM_{\cH \underl{R}}, \, \cN_{\phi(\cH \underl{R})}, \, \cN_{\cH \underl{R}}: \gppairs^\op \to \calg(\prlst)$
which map $(G,H)$ to $\Mod_{\nH \underl{R}}(\specg)$, $\Mod_{\Phi^H(\nH \underl{R})}(\specwgh)$, and $\Mod_{\nH \underl{R}}(\specwgh)$ respectively. As above $\eta_{\tn{geo}}$ induces a natural transformation $\bar{\eta}_{\tn{geo}}: \cM_{\cH \underl{R}} \Rightarrow \cN_{\phi(\cH \underl{R})}$ which is pointwise given by the functors induced by geometric fixed points on module categories. To invoke the functoriality in morphisms of algebra objects, we will need that the maps $\psi_{H;G}: \Phi^H(\nH \underl{R}) \to \nH \underl{R}$ of \Cref{defi:eqmodularfixedpoints} assemble into a morphism $\phi(\cH \underl{R}) \to \cH \underl{R}$ in $\calg(\tn{laxlim} \; \cN)$. We will do this by computing the mapping space between (the sections corresponding to) $\phi(\cH \underl{R})$ and $\cH \underl{R}$ in $\alg_{(\gppairs^{\op})^{\amalg}}(\alg_{\finstar/(\gppairs^{\op})^{\amalg}}(\cN)^{\otimes})$. Here $\cN^{\otimes} \to (\gppairs^{\op})^{\amalg}$ is the cocartesian symmetric monoidal unstraightening of $\cN$. \\
Write $p: \alg_{\finstar/(\gppairs^{\op})^{\amalg}}(\cN)^{\otimes} \to (\gppairs^{\op})^{\amalg}$ for the map of $\infty$-operads constructed from $\cN^{\otimes} \to (\gppairs^{\op})^{\amalg}$ as above, and we furthermore write $S_{\phi(\cH \underl{R})}$ and $S_{\cH \underl{R}}$ for the sections of $p$ corresponding to $\phi(\cH \underl{R})$ and $\cH \underl{R}$ respectively. The $\infty$-category 
$$\alg_{(\gppairs^{\op})^{\amalg}}(\alg_{\finstar/(\gppairs^{\op})^{\amalg}}(\cN)^{\otimes}) \simeq \alg_{(\gppairs^{\op})^{\amalg}/(\gppairs^{\op})^{\amalg}}(\cN^{\otimes})$$ 
is the full subcategory of 
$$\fun_{(\gppairs^{\op})^{\amalg}}((\gppairs^{\op})^{\amalg},\cN^{\otimes})$$
on maps of $\infty$-operads, and the latter functor category sits in the pullback diagram
\begin{equation}\label{diag:pullbackfunctors}
\begin{tikzcd}
	{(-)} & {\fun((\gppairs^{\op})^{\amalg},\cN^{\otimes})} \\
	{\Delta^0} & {\fun((\gppairs^{\op})^{\amalg},(\gppairs^{\op})^{\amalg}).}
	\arrow[from=1-1, to=1-2]
	\arrow[from=1-1, to=2-1]
	\arrow["{p_*}", from=1-2, to=2-2]
	\arrow["{\tn{id}}", from=2-1, to=2-2]
    \arrow["\usebox\pullback"{anchor=center, pos=0.13}, draw=none, from=1-1, to=2-2]
\end{tikzcd}
\end{equation}
By \cite[Proposition 2.3]{gla16} mapping spaces in the two functor categories on the right-hand side of \ref{diag:pullbackfunctors} are computed as ends, and hence the mapping space 
\begin{equation}\label{eq:mappingspace}
\Map_{\alg_{(\gppairs^{\op})^{\amalg}}(\alg_{\finstar/(\gppairs^{\op})^{\amalg}}(\cN)^{\otimes})}(S_{\phi(\cH \underl{R})},S_{\cH \underl{R}})
\vspace{0.2cm}
\end{equation}
is computed as the pullback
\begin{equation}\label{diag:pullbackmappingspaces}
\begin{tikzcd}
	{(-)} & {\int_{(\gppairs^{\op})^{\amalg}} \Map_{\cN^{\otimes}}(S_{\phi(\cH \underl{R})}(-),S_{\cH \underl{R}}(-))} \\
	{\ast} & {\int_{(\gppairs^{\op})^{\amalg}} \Map_{(\gppairs^{\op})^{\amalg}}(\tn{id}(-),\tn{id}(-)).}
	\arrow[from=1-1, to=1-2]
	\arrow[from=1-1, to=2-1]
	\arrow["{p_*}", from=1-2, to=2-2]
	\arrow[from=2-1, to=2-2]
    \arrow["\usebox\pullback"{anchor=center, pos=0.13}, draw=none, from=1-1, to=2-2]
\end{tikzcd}
\end{equation}
The above ends are limits over the twisted arrow category $\tn{Tw}((\gppairs^{\op})^{\amalg})^\op$ (see \cite[Section 5.2.1]{lur17}), and we can view the terminal space $\ast$ as a limit over the constant diagram $\tn{Tw}((\gppairs^{\op})^{\amalg})^\op \to \spc$. Assigning to a morphism in $(\gppairs^{\op})^{\amalg}$ its corresponding point in the mapping space that appears in the bottom right end provides the lower horizontal map in \ref{diag:pullbackmappingspaces}. Hence the mapping spaces \ref{eq:mappingspace} are given by the limit of the functor $F: \tn{Tw}((\gppairs^{\op})^{\amalg})^\op \to \; \spc$ which sends an object $\alpha^\op: (G',H') \to (G,H)$ of $\tn{Tw}((\gppairs^{\op})^{\amalg})^\op$ to the pullback of the cospan
\[\begin{tikzcd}
	& {\Map_{\cN^{\otimes}}([(G',H'),\Phi^{H'}(\nH\underl{R})],[(G,H),\nH\underl{R}])} \\
	{\ast} & {\Map_{(\gppairs^{\op})^{\amalg}}((G',H'),(G,H)).}
	\arrow[from=1-2, to=2-2]
	\arrow["\alpha^\op", from=2-1, to=2-2]
\end{tikzcd}\]
Using the description of the $(\gppairs^{\op})^{\amalg}$-monoidal $\infty$-category $\cN^{\otimes}$ given in \cite[Remark A.13]{dg22} the pullback is given by $\Map_{\calg(\specwgh)}(\bar{\alpha}^*\Phi^{H'}(\nH \underl{R}),\nH \underl{R})$. For a full description of the operad $(\gppairs^{\op})^{\amalg}$ we refer to \cite[Constructions 2.4.3.1, 2.4.4.1]{lur17}, we only included length 1 strings of objects in the above description. As in the proof of \Cref{prop:eqmodfixrestriction} maps $\Phi^H\alpha^*(\nH \underl{R}) \simeq \bar{\alpha}^*\Phi^{H'}(\nH \underl{R}) \to \nH \underl{R}$ are determined on $\pi_0^1$, and there is an equivalence of discrete spaces 
$$\Map_{\calg(\specwgh)}(\bar{\alpha}^*\Phi^{H'}(\nH \underl{R}),\nH \underl{R}) \simeq \tn{End}(R).$$ 
On morphisms of $\tn{Tw}((\gppairs^{\op})^{\amalg})^\op$, the functor $F$ maps
\[\begin{tikzcd}
	{(G',H')} & {(P',T')} && {\Map_{\calg(\specwgh)}(\bar{\alpha}^*\Phi^{H'}(\nH \underl{R}),\nH \underl{R})} \\
	&& \mapsto & {\Map_{\calg(\spec^{\weyl{P}{T}})}(\bar{\delta}^*\bar{\alpha}^*\Phi^{H'}(\nH \underl{R}),\nH \underl{R})} \\
	{(G,H)} & {(P,T)} && {\Map_{\calg(\spec^{\weyl{P}{T}})}(\bar{\gamma}^*\Phi^{T'}(\nH \underl{R}),\nH \underl{R}),}
	\arrow["{\alpha^\op}"', from=1-1, to=3-1]
	\arrow["{\gamma^\op}"', from=1-2, to=1-1]
	\arrow["{\beta^\op}", from=1-2, to=3-2]
	\arrow["{\nu(\bar{\delta})_* \circ \, \bar{\delta}^*}", from=1-4, to=2-4]
	\arrow["{(\bar{\delta}^*\bar{\alpha}^*\Phi^T(\nu(\gamma)))^*}", from=2-4, to=3-4]
	\arrow["{\delta^\op}", from=3-1, to=3-2]
\end{tikzcd}\]
and the right-hand composition on mapping spaces is an equivalence since the maps involved are either compositions with maps that are identic on $\pi^1_0$ or maps induced by restriction along a group homomorphism on mapping spaces which do not change the $\pi^1_0$ of the objects at play. Since for any morphism $\alpha: (G,H) \to (G',H')$ in $\gppairs$ there is a commutative diagram
\[\begin{tikzcd}
	{(G,H)} & {(G,H)} & {(1,1)} \\
	{(G',H')} & {(1,1)} & {(1,1)}
	\arrow["\alpha"', from=1-1, to=2-1]
	\arrow["{\tn{id}}"', from=1-2, to=1-1]
	\arrow[from=1-2, to=1-3]
	\arrow[from=1-2, to=2-2]
	\arrow[from=1-3, to=2-3]
	\arrow[from=2-1, to=2-2]
	\arrow[from=2-3, to=2-2]
\end{tikzcd}\]
the diagram given by $F: \tn{Tw}((\gppairs^{\op})^{\amalg})^\op \to \; \spc$ is connected. Hence \ref{eq:mappingspace} is discrete (which also follows from \cite[Remark 7.4.1.5]{kerodon}) and equivalent to $\tn{End}(R)$. So the maps $\psi_{H;G}$ assemble into a map $\phi(\cH \underl{R}) \to \cH \underl{R}$ in $\calg(\tn{laxlim} \; \cN)$, and this map corresponds to the identity of $R$ under the equivalence to the computed mapping space. \\
We obtain a natural transformation $\eta_{\tn{mod}}: \Mod_{\nH \underl{R}}(\spec^{(-)}) \Rightarrow \Mod_{\nH \underl{R}}(\spec^{\weyl{(-)}{(-)}})$ of functors $\gppairs^\op \to \calg(\prlst)$ which on $(G,H)$ is given by the modular $H$-fixed points $\Psi^{H;G}: \Mod_{\nH \underl{R}}(\specg) \to \Mod_{\nH \underl{R}}(\specwgh)$.
\end{cons}

Let us now return to the case where $G = \lim_{i \in I} G_i = \lim_{i \in I} G/N_i$ is a general profinite group with maps $\alpha_i: G \twoheadrightarrow G_i = G/N_i$. 

\begin{defi}\label{defi:eqmodularfixedpointsprofinite}
Let $H \leq G$ be a closed pro-$p$-subgroup. We define the \emph{equivariant modular fixed point functor}
$$\Psi^H = \Psi^{H;G}: \Mod_{\nH \underl{R}}(\specg) \to \Mod_{\nH \underl{R}}(\specwgh)$$
as the essentially unique symmetric monoidal left adjoint functor
$$\tn{colim}_i \, \Mod_{\nH \underl{R}}(\spec^{G_i}) \to \tn{colim}_i \, \Mod_{\nH \underl{R}}(\spec^{\weyl{G_i}{H_i}})$$ 
induced by the natural transformation 
\[\begin{tikzcd}
	{I^\op \times \Delta^1} & {\gppairs^\op \times \Delta^1} & \calg(\prlst),
	\arrow["{\iota \times \tn{id}}", from=1-1, to=1-2]
	\arrow["\eta_{\tn{mod}}", from=1-2, to=1-3]
\end{tikzcd}\]
where the map $\iota$ is given by $i \mapsto (N_i, H_i)$ and $\eta_{\tn{mod}}$ is the natural transformation between $\Mod_{\nH \underl{R}}(\spec^{(-)})$ and $\Mod_{\nH \underl{R}}(\spec^{\weyl{(-)}{(-)}})$ of \Cref{cons:eqmodfixasnattrans}.
\end{defi}

\begin{re}
Since the functor $\Psi^H: \Mod_{\nH \underl{R}}(\specg) \to \Mod_{\nH \underl{R}}(\specwgh)$ is formally extended from finite to profinite groups, it inherits all properties from the finite group case which are compatible with the description of $\Mod_{\nH \underl{R}}(\specg)$ as a colimit. In particular, it makes the square of geometric fixed points and free functors commute (\Cref{lem:eqmodfixfree}), it is weight exact when restricted to compacts (\Cref{cor:weightexact}), it is $\nH R$-linear (\Cref{prop:eqmodfixhrlinear}), nesting modular fixed point functors up to restriction gives another modular fixed point functor (\Cref{prop:eqmodfixnested}), and it is compatible with restriction along a group homomorphism (\Cref{prop:eqmodfixrestriction})
\end{re}

\begin{thm}\label{thm:identifymodfix}
Under (the \tn{Ind}-extension of) the equivalence of \Cref{thm:eqhrmoddpermweight}, the equivariant modular fixed point functor
$$\Psi^H: \Mod_{\nH \underl{R}}(\specg) \to \Mod_{\nH \underl{R}}(\specwgh)$$
for a closed pro-$p$-subgroup $H \leq G$ identifies with the modular fixed point functor on derived permutation modules
$$\Psi^H: \cD\Permgr \to \cD\Perm(\wgh;R)$$
of Balmer-Gallauer \cite[Definition 5.10]{bg25a}, \cite[Construction 2.11]{bg25b}.
\end{thm}
\begin{proof}
Both modular fixed point functors are obtained by formally extending their construction from finite to profinite groups, see \cite[Remark 2.12]{bg25b}. It hence suffices to consider the case where $G$ is finite and $H$ is a $p$-subgroup. \\
Since all functors involved are colimit preserving functors of compactly generated $\infty$-categories, it suffices to show commutativity on compact parts. These carry bounded weight structures. Since the functors at play are weight exact, it suffices by \cite[Proposition 3.3]{sos19} to identify their restrictions to weight hearts. In fact by \cite[Proposition 3.27]{bgv25} this will identify the functors on $\infty$-categories as symmetric monoidal functors. Since the targets are idempotent complete, we can consider the weight hearts before idempotent completion. By \Cref{cor:restrictiontonormalsubgroup}, \cite[Definition 5.10]{bg25a} and since the equivalence of \Cref{thm:eqhrmoddpermweight} is natural in restrictions to subgroups we can furthermore assume that $H$ is normal in $G$. So we reduced to showing that the square
\[\begin{tikzcd}
	{\cH(G)} & {\cH(G/H)} \\
	{\permgr} & {\perm(G/H;R)}
	\arrow["{\Psi^H}", from=1-1, to=1-2]
	\arrow[from=1-1, to=2-1, "\rotatebox{270}{$\sim$}"']
	\arrow[from=1-2, to=2-2, "\rotatebox{90}{$\sim$}"]
	\arrow["{\Psi^H}", from=2-1, to=2-2]
\end{tikzcd}\]
commutes, where we write $\cH(G)$ and $\cH(G/H)$ for the weight hearts of \Cref{prop:weightonhrmod} with an emphasis on the ambient category. \\
We have $\Phi^H(\nH \underl{R} \otimes \Sigma^{\infty}_+ G/K) \simeq \nH \underl{R} \otimes \Sigma^{\infty}_+ (G/K)^H$ for $K \leq G$ by \Cref{lem:eqmodfixfree}, and on the permutation module side $R(G/K)$ is mapped to $R((G/K)^H)$ by \cite[Proposition 5.12]{bg25a}. Hence the square commutes on objects. \\
Let us consider morphisms. By \cite[Remark 4.14]{bg23} the behaviour on morphisms of an $R$-linear symmetric monoidal functor on $\permgr$ is uniquely determined by its image of conjugations $(-)^g: R(G/K) \to R(G/K^g)$ and projections to subgroups $\pi^L_K: R(G/L) \to R(G/K)$ for $L \leq K$. We can assume that $H$ is subconjugate to $L$, as otherwise the $H$-fixed points of $G/L$ are empty. So since $H$ is normal in $G$, we can assume that $H \leq L \leq K$. Again by \cite[Proposition 5.12]{bg25a}, the modular fixed point functor on permutation modules sends $(-)^g: R(G/K) \to R(G/K^g)$ to $(-)^{[g]_H}: R((G/H)/(K/H)) \to R((G/H)/(K^g/H))$, and $\pi^L_K: R(G/L) \to R(G/K)$ which corresponds to an element $[x]_K \in N_G(L,K)/K$ is sent to $\pi^{L/H}_{K/H}: R((G/H)/(L/H)) \to R((G/H)/(K/H))$, which corresponds to the element $[x]_{K/H} \in N_{G/H}(H/L,K/H)/(K/H)$. \\
On the other hand, the map $\tn{id} \otimes (-)^g: \nH \underl{R} \otimes \Sigma^{\infty} G/K_+ \to \nH \underl{R} \otimes \Sigma^{\infty} G/K^g_+$ by the geometric fixed points on $\nH \underl{R}$-modules is sent to 
$$\tn{id} \otimes (-)^{[g]_H}: \Phi^H(\nH \underl{R}) \otimes \Sigma^{\infty} (G/H)/(K/H)_+ \to \Phi^H(\nH \underl{R}) \otimes \Sigma^{\infty} (G/H)/(K^g/H)_+,$$
and base changing back to $\nH \underl{R}$ via $\psi: \Phi^H(\nH \underl{R}) \to \nH \underl{R}$ from \Cref{defi:eqmodularfixedpoints} gives 
$$\tn{id} \otimes (-)^{[g]_H}: \nH \underl{R} \otimes \Sigma^{\infty} (G/H)/(K/H)_+ \to \nH \underl{R} \otimes \Sigma^{\infty} (G/H)/(K^g/H)_+,$$ since $\psi$ is induced by the identic map of constant Mackey functors. Similarly, the image of the projection $\tn{id} \otimes \pi^L_K: \nH \underl{R} \otimes \Sigma^{\infty} G/L_+ \to \nH \underl{R} \otimes \Sigma^{\infty} G/K_+$ is 
$$\tn{id} \otimes \pi^{L/H}_{K/H}: \nH \underl{R} \otimes \Sigma^{\infty} (G/H)/(L/H)_+ \to \nH \underl{R} \otimes \Sigma^{\infty} (G/H)/(K/H)_+.$$
This finishes the proof.
\end{proof}

We are indebted to Drew Heard and Achim Krause for pointing out the following.

\begin{re}
The proof of the previous theorem shows that a (symmetric monoidal) left adjoint functor $\Mod_{\nH \underl{R}}(\specg) \to \Mod_{\nH \underl{R}}(\specwgh)$ which is weight exact when restricted to compacts is completely determined by its restriction to weight hearts $\perm(G;R)^{\natural} \to \perm(\wgh;R)^\natural$. \\
This also means that the functor $\Mod_{\nH \underl{R}}(\specg) \to \Mod_{\nH \underl{R}}(\specwgh)$ we obtain from the universal property of $\Mod_{\nH \underl{R}}(\specg)$ provided by \Cref{cor:modhunderlraspresheaves} and \Cref{rec:psigma}, applied to the composition of classical Brauer quotients with the spectral Yoneda embedding
$$\perm(G;R) \to \perm(\wgh;R) \hookrightarrow \psigma(\perm(\wgh;R);\spec) \simeq \Mod_{\nH \underl{R}}(\specwgh),$$
is given by modular fixed points.
\end{re}

\section{Picard Groups and Borel-Smith Class Functions}\label{sec:picardgroups}

In this section we use the established machinery to study the Picard group of $\Mod_{\nH \underl{k}}(\specg)$, where now $k$ is a field of prime characteristic $p > 0$. We first assume $G$ to be a general finite group and later restrict to $p$-groups. In the end we will formally extend our results for $p$-groups to pro-$p$-groups. \\
Since $\Mod_{\nH \underl{k}}(\specg)$ is rigidly compactly generated, invertible objects are compact, so we can restrict our attention to the Picard group of the compact part. The arguments presented in \Cref{prop:lambdaclassfunctions} and \Cref{lem:lambdaimage} that produce an injective homomorphism from the Picard group to the group of \emph{Borel-Smith class functions} defined on $p$-subgroups to a large extent are direct topological analogues of arguments of Miller \cite{mil25b}. The advantage of the topological approach is that we can invoke a result on dimension functions of real representations to show surjectivity of this homomorphism, see \Cref{thm:picborelsmith}. In what comes below we will consider the family of functors
\[\begin{tikzcd}
	{\bar{\Psi}^H: \; \Mod_{\nH \underl{k}}(\specg)} & {\Mod_{\nH \underl{k}}(\specwgh)} & {\Mod_{\nH k}(\spec) \simeq \cD(\Mod(k))}
	\arrow["{\Psi^H}", from=1-1, to=1-2]
	\arrow["{\tn{res}^{\wgh}_1}", from=1-2, to=1-3]
\end{tikzcd}\]
indexed on the (conjugacy classes of) $p$-subgroups $H \leq G$, and their restrictions to compact parts
\[\begin{tikzcd}
	{\bar{\Psi}^H: \; \Mod_{\nH \underl{k}}^{\omega}(\specg)} & {\Mod_{\nH \underl{k}}^{\omega}(\specwgh)} & {\Mod_{\nH k}^{\omega}(\spec) \simeq \cD_{\tn{b}}(\tsf{mod}(k)),}
	\arrow["{\Psi^H}", from=1-1, to=1-2]
	\arrow["{\tn{res}^{\wgh}_1}", from=1-2, to=1-3]
\end{tikzcd}\]
where $\tsf{mod}(k)$ is the category of finitely generated $k$-vector spaces.
\begin{prop}\label{prop:conservativity}
The family of functors $(\bar{\Psi}^H)_{H \in \, \subp(G)/G}$ is jointly conservative and weight conservative.
\end{prop}
\begin{proof}
For conservativity, we have a commutative diagram
\[\begin{tikzcd}
	{\Mod_{\nH \underl{k}}^{\omega}(\specg)} & {\Mod_{\nH \underl{k}}^{\omega}(\specwgh)} & {\cD_{\tn{b}}(\tsf{mod}(\wgh;k))} \\
	& {\cD_{\tn{b}}(\tsf{mod}(k))}
	\arrow["{\Psi^H}", from=1-1, to=1-2]
	\arrow["{\bar{\Psi}^H}"', from=1-1, to=2-2]
	\arrow["{\Upsilon_{\wgh}}", from=1-2, to=1-3]
	\arrow["{\tn{res}^{\wgh}_1}", from=1-2, to=2-2]
	\arrow["{\tn{res}^{\wgh}_1}", from=1-3, to=2-2]
\end{tikzcd}\]
in which $\Upsilon_{\wgh}: \Mod_{\nH \underl{k}}^{\omega}(\specwgh) \simeq \cK_{\tn{b}}(\perm(\wgh;k)^{\natural}) \to  \cD_{\tn{b}}(\tsf{mod}(\wgh;k))$ is the canonical functor induced by $\perm(\wgh;k) \hookrightarrow \tsf{mod}(\wgh;k)$. By \cite[Theorem 6.12]{bg25a} and \Cref{thm:identifymodfix} the family of functors $(\Upsilon_{\wgh}\Psi^H)_{H \in \, \subp(G)/G}$ is jointly conservative, and since the right-hand restriction functors are conservative, the family $(\bar{\Psi}^H)_{H \in \, \subp(G)/G}$ is jointly conservative. \\
Let us consider weight conservativity. The functors $\bar{\Psi}^H$ for $H \in \subp(G)/G$ are weight exact by \Cref{cor:weightexact}, so by \Cref{prop:weightconservativecharac} it suffices to show that their restriction to weight hearts jointly detects sections and retractions. By \Cref{thm:identifymodfix} the restriction to weight hearts can be identified with the classical Brauer quotient $\perm(G;k)^{\natural} \to \perm(\wgh;k)^{\natural}$, and hence the result follows from \cite[Proposition 5.8.11]{lin18} and its dual version, which holds since we are using field coefficients.
\end{proof}

The following is the topological analogue of \cite[Theorem 3.4]{mil24}.

\begin{cor}\label{cor:invertiblemodules}
A compact $\nH \underl{k}$-module $X \in \Mod_{\nH \underl{k}}^{\omega}(\specg)$ is invertible if and only if for all $H \in \subp(G)/G$ the graded $k$-vector space $\pi_*(\bar{\Psi}^H(X))$ is isomorphic to $k$, concentrated in a single degree.
\end{cor}
\begin{proof}
We have $\pic(\Mod_{\nH k}(\spec)) \cong \bZ$ with invertible objects given by $\Sigma^n \nH k$ for $n \in \bZ$. Since $\bar{\Psi}^H$ is symmetric monoidal the image of an invertible $\nH \underl{k}$-module is invertible, so of the form $\Sigma^n Hk$. Hence the statement about the homotopy. \\
Conversely, for a module $X \in \Mod_{\nH \underl{k}}^{\omega}(\specg)$ and all $p$-subgroups $H \leq G$ assume that $\pi_*(\bar{\Psi}^H(X))$ vanishes in all but one degree in which it is isomorphic to $k$. Then $\bar{\Psi}^H(X) \simeq \Sigma^n \nH k$ for some $n$, and hence $\bar{\Psi}^H(X)$ is invertible for all $p$-subgroups $H \leq G$. Since $X$ is compact it is dualisable, and hence it is invertible if and only if the evaluation and coevaluation maps $X^* \otimes_{\nH \underl{k}} X \to \nH \underl{k}$ and $\nH \underl{k} \to X \otimes_{\nH \underl{k}} X^*$ are equivalences. As the $\bar{\Psi}^H(X)$ for $H \in \subp(G)/G$ are invertible the images of these maps under $\bar{\Psi}^H$ are equivalences. Hence conservativity from \Cref{prop:conservativity} implies that $X$ is invertible.
\end{proof}

\begin{defi}[cf. {\cite[Theorem 1.2]{mil25a}}]
Let $\Lambda(G;k)$ be the abelian group of equivalence classes of compact modules $X \in \Mod_{\nH \underl{k}}^{\omega}(\specg)$ for which the following holds: If $H \in \subp(G)$ then $\pi_*(\bar{\Psi}^H(X))$ is nonzero in exactly one degree and if $H \in \sylp(G)$ then this $k$-vector space has dimension $1$. The group structure is induced by the symmetric monoidal structure on $\Mod_{\nH \underl{k}}(\specg)$.
\end{defi}

Note that \Cref{cor:invertiblemodules} implies that $\pic(\Mod_{\nH \underl{k}}(\specg))$ is a subgroup of $\Lambda(G;k)$.

\begin{prop}[cf. {\cite[Theorem 2.12]{mil25b}}]\label{prop:lambdaclassfunctions}
Write $\tn{CF}(G,p)$ for the group of integral class functions defined on the conjugacy classes of $p$-subgroups of $G$. There is a surjective group homomorphism
$$\lambda: \Lambda(G;k) \to \tn{CF}(G,p), \; X  \mapsto (\lambda_X: (H) \mapsto n, \tn{ where } \pi_n(\bar{\Psi}^H(X)) \neq 0).$$
If $G$ is a $p$-group, then it is an isomorphism.
\end{prop}
\begin{proof}
We use the fact that the abelian group $\tn{CF}(G,p)$ has a $\bZ$-basis given by 
$$(\omega_P)_{P \, \in \, \subp(G)/G}, \hspace{0.2cm} \omega_P(H) = \begin{cases}
    1, & \tn{if } H \leq_G P, \\
    0, & \tn{otherwise.}
\end{cases}$$
So for surjectivity we need to show that there are $X_P \in \Lambda(G;k)$ with $\lambda_{X_P} = \omega_P$. If $P \in \sylp(G)$ we set $X_P \defeq \Sigma \nH \underl{k}$. Then $\bar{\Psi}^H(X_P) \simeq \Sigma \nH k$ for all $H \in \subp(G)/G$, and hence $\lambda_{X_P} = \omega_P$. If $P \in \subp(G) \backslash \sylp(G)$ we set $X_P \defeq \nH \underl{k} \otimes \Sigma^{\infty}_+ Y_P$, where 
$$Y_P \defeq G/G \cup_f (G/P \times D^1)$$
with trivial attaching map $f: G/P \times S^0 \to G/G$. Then $Y_P$ is a wedge of $[G:P]$ many copies of $S^1$ on which $G$ acts by permutation. Then $\bar{\Psi}^H(X_P) \simeq \nH \underl{k} \otimes \Sigma^{\infty}_+ (Y_P)^H$. The $\wgh$-CW complex $(Y_P)^H$ has cell structure 
$$(\wgh)/(\wgh) \cup_g \coprod_i (\wgh)/L_i \times D^1$$
with trivial attaching map $g$, where $\coprod_i (\wgh)/L_i \cong N_G(H,P)/P$ as $\wgh$-sets and $N_G(H,P)$ is the relative normaliser which consists of elements of $G$ which conjugate $H$ into $P$. This $\wgh$-CW complex corresponds to the cellular chain complex of $(\wgh;k)$-permutation modules
\[\begin{tikzcd}
	\dots & 0 & {k(N_G(H,P)/P)} & k & 0 & \dots
	\arrow[from=1-1, to=1-2]
	\arrow[from=1-2, to=1-3]
	\arrow[from=1-3, to=1-4]
	\arrow[from=1-4, to=1-5]
	\arrow[from=1-5, to=1-6]
\end{tikzcd}\]
with augmentation map $k(N_G(H,P)/P) \to k$ going from degree $1$ to degree $0$. Since $N_G(H,P)$ is nonempty precisely when $H \leq_G P$ this complex has homology concentrated in degree $1$ precisely when $H \leq_G P$ and in degree $0$ otherwise. Thus we have $\lambda_{X_P} = \omega_P$. \\
For injectivity, if $\pi_*(\bar{\Psi}^H(X))$ is concentrated in degree $0$ for all $H \in \subp(G)/G$ then the whole family $(\bar{\Psi}^H)_{H \in \, \subp(G)/G}$ sends $X$ into the weight heart of $\Mod_{\nH k}^{\omega}(\spec)$. It follows from the weight conservativity of \Cref{prop:conservativity} that $X$ lives in the weight heart of $\Mod_{\nH \underl{k}}^{\omega}(\specg)$. If now $G$ is a $p$-group, then $k(G/G)$ is the only element of $\permgr^{\natural} = \permgr$ which has $k$-dimension $1$. Hence $X \simeq \nH \underl{k}$ and we conclude that $\lambda$ is injective.
\end{proof}

\begin{re}
In the context of modular representation theory, the elements of the Picard group $\pic(\nK_{\tn{b}}(\perm(G;k)^{\natural}))$ are known as \emph{endotrivial complexes}, and the elements in the group analogous to $\Lambda(G;k)$ are the \emph{$V$-endosplit-trivial complexes}, where $V = \bigoplus_{Q \, \in \, \subp(G) \backslash \sylp(G)} k(G/Q)$, cf. \cite{mil24}, \cite{mil25a}. In this context the homomorphism $\lambda: \Lambda(G;k) \to \tn{CF}(G,p)$ is also called \emph{h-marks}.
\end{re}

\begin{defi}{\cite[p. 210]{td87}}
An integral class function $h \in \tn{CF}(G,p)$ defined on the $p$-subgroups of $G$ satisfies the \textit{Borel-Smith conditions} if the following holds:
\begin{enumerate}
    \item[(1)] If $N \trianglelefteq H \leq G$ are $p$-subgroups of $G$ such that $H/N \cong (\bZ/p\bZ)^2$, then 
    $$h(N) - \sum_{N < K < H}h(K) +ph(H) = 0.$$
    \item[(2)] If $N \trianglelefteq H \leq G$ are $p$-subgroups of $G$ such that $H/N$ is cyclic of order $p$ for an odd prime $p$, or cyclic of order $4$, then $h(N) \equiv h(K) \mod 2$, where $K/N$ is the unique subgroup of prime order of $H/N$.
    \item[(3)] If $N \trianglelefteq H \leq G$ are $p$-subgroups of $G$ such that $H/N$ is a quaternion group of order $8$, then $h(N) \equiv h(K) \mod 4$, where $K/N$ is the unique subgroup of order $2$ of $H/N$.
\end{enumerate}
The class functions satisfying the Borel-Smith conditions form an additive subgroup of $\tn{CF}(G,p)$, denoted by $\tn{CF}_{\tn{b}}(G,p)$.
\end{defi}

From now on we assume that $G$ is a $p$-group. We will consequently write $\tn{CF}(G)$ and $\tn{CF}_{\tn{b}}(G)$ instead of $\tn{CF}(G,p)$ and $\tn{CF}_{\tn{b}}(G,p)$. We will use a representation theoretic input to show that $\lambda: \Lambda(G;k) \to \tn{CF}(G)$ lands in the class functions satisfying the Borel-Smith conditions when restricted to $\pic(\Mod_{\nH \underl{k}}(\specg))$.

\begin{lem}[cf. {\cite[Proposition 4.1, Theorem 4.6]{mil25b}}]\label{lem:lambdaimage}
The homomorphism $\lambda: \Lambda(G;k) \to \tn{CF}(G)$ maps elements of the Picard group $\pic(\Mod_{\nH \underl{k}}(\specg))$ to class functions in $\tn{CF}_{\tn{b}}(G)$.
\end{lem}
\begin{proof}
We can define a homomorphism $\gamma: \Lambda(G;k) \to D_k(G)$ which sends a module $X \in \Lambda(G;k)$ to the class of the $(k;G)$-module $\pi^{1}_{\lambda_X(1)}(X)$. Here $D_k(G)$ is the \emph{Dade group}, it is defined as the set of capped endopermutation modules up to  isomorphism of caps, and it is an abelian group under tensor product over $k$, see \cite{dad78}. We furthermore need another homomorphism to the Dade group, the \emph{Bouc homomorphism}\footnote{The target of the Bouc homomorphism normally is a certain subgroup of $D_k(G)$, but since we are only interested in the kernel this does not matter for our purposes.} $\beta: \tn{CF}(G) \to D_k(G)$, which on the $\bZ$-basis $(\omega_P)_{P \, \in \, \sub(G)/G}$ of $\tn{CF}(G)$ is defined by
$$\omega_P \mapsto \begin{cases}
    [k], & \tn{if } P = G, \\
    \Omega_P, & \tn{otherwise},
\end{cases}$$
where $\Omega_P$ is the equivalence class represented by the kernel of the augmentation map $k(G/P) \twoheadrightarrow k$. \\
Recall from the proof of \Cref{prop:lambdaclassfunctions} that there is a $\bZ$-basis $(X_P)_{P \, \in \, \sub(G)/G}$ of $\Lambda(G;k)$ with $\lambda(X_P) = \omega_P$. We have that $\beta \lambda(X_P) = \beta(\omega_P) = \begin{cases}
    [k], & \tn{if } P = G, \\
    \Omega_P, & \tn{otherwise}.
\end{cases}$
On the other hand, $\gamma(X_P) = \gamma(\Sigma \nH \underl{k}) = [\pi^{1}_{1}(\Sigma \nH \underl{k})] = [k]$ for $P=G$ and $\gamma(X_P) = \gamma(\nH \underl{k} \otimes \Sigma^{\infty}_+ Y_P) = [\pi^{1}_{1}(\nH \underl{k} \otimes \Sigma^{\infty}_+ Y_P)] = [\tn{ker}(k(G/P) \twoheadrightarrow k)]$ for $P \neq G$, using the notation of \Cref{prop:lambdaclassfunctions}. This shows $\beta \circ \lambda = \gamma$. Now by \Cref{cor:invertiblemodules} the subgroup $\pic(\Mod_{\nH \underl{k}}(\specg))$ is contained in the kernel of $\gamma$, hence $\lambda$ maps them into the kernel of $\beta$. But by \cite[Theorem 1.2]{by07} the kernel of the Bouc homomorphism $\beta$ is the group of Borel-Smith class functions $\tn{CF}_{\tn{b}}(G)$.
\end{proof}

At this point we could show that $\pic(\Mod_{\nH \underl{k}}(\specg))$ is not only contained in but equal to the kernel of the composition $\beta \circ \lambda: \Lambda(G;k) \to \tn{CF}(G) \to D_k(G)$. Instead of doing so we will use a topological argument to show surjectivity of the induced homomorphism 
$$\theta: \pic(\Mod_{\nH \underl{k}}(\specg)) \to \tn{CF}_{\tn{b}}(G)$$
which relies on work of tom Dieck on dimension functions of orthogonal representations \cite[Section III.5]{td87}, which in turn builds on work of Atiyah-Tall and Serre on Adams operations on $RO(G)$.

\begin{thm}\label{thm:picborelsmith}
The map $\theta: \pic(\Mod_{\nH \underl{k}}(\specg)) \to \tn{CF}_{\tn{b}}(G)$ is an isomorphism.
\end{thm}
\begin{proof}
Injectivity was shown in \Cref{prop:lambdaclassfunctions}, we need to show that $\theta$ is surjective. For a real orthogonal $G$-representation $V$ there is an associated dimension class function $\tn{dim}(V): (H) \mapsto \tn{dim}_{\bR}(V^H)$. The assignment $V \mapsto \tn{dim}(V)$ extends to a homomorphism of additive groups $\tn{dim}: RO(G) \to \tn{CF}(G)$ defined on the ring of real orthogonal $G$-representations. Using results of Smith and Borel \cite[Section III.4]{td87} about fixed points of finite $G$-CW complexes which are $\bF_p$-cohomology spheres this dimension homomorphism lands in the Borel-Smith class functions $\tn{CF}_{\tn{b}}(G)$, see \cite{td87} before Definition III.5.1. \\
Mapping a real orthogonal $G$-representation $V$ to its representation sphere $S^V$ also extends to a group homomorphism $\eta: RO(G) \to \pic(\specg)$. The constructed map $\theta: \pic(\Mod_{\nH \underl{k}}(\specg)) \to \tn{CF}_{\tn{b}}(G)$ sends the invertible $\nH \underl{k}$-module $\nH \underl{k} \otimes S^V$ to the class function mapping a subgroup $H \leq G$ to the unique degree in which the cellular chain complex of $(\wgh;k)$-permutation modules for $(S^V)^H \cong S^{V^H}$ has nonzero homology. But this degree is precisely the real dimension of $V^H$. Hence we obtain a factorisation
\[\begin{tikzcd}
	{RO(G)} & {\tn{CF}_{\tn{b}}(G)} \\
	{\pic(\specg)} & {\pic(\Mod_{\nH \underl{k}}(\specg)),}
	\arrow["\tn{dim}", from=1-1, to=1-2]
	\arrow["\eta", from=1-1, to=2-1]
	\arrow["{f_{\nH \underl{k}}}", from=2-1, to=2-2]
	\arrow["\theta"', from=2-2, to=1-2]
\end{tikzcd}\]
where $f_{\nH \underl{k}}$ is the homomorphism induced by the (symmetric monoidal) free functor on Picard groups. By \cite[Theorem III.5.4]{td87} the top map is surjective, hence $\theta$ is surjective as well. This finishes the proof.
\end{proof}

\begin{re}
From the isomorphism $\pic(\Mod_{\nH \underl{k}}(\specg)) \cong \tn{CF}_{\tn{b}}(G)$ which holds when $G$ is a $p$-group one can deduce an isomorphism
$$\pic(\Mod_{\nH \underl{k}}(\specg)) \cong \tn{CF}_{\tn{b}}(G,p) \times \Hom(G,k^{\times})$$
which holds for arbitrary finite groups $G$. It requires showing that restriction to a certain \emph{$G$-stable} subgroup of $\pic(\Mod_{\nH \underl{k}}(\specg))$ is surjective \cite[Theorem 12.6]{mil25a}. We refer to \cite[Section 6.1]{mil25b}, where the proof is given in the language of representation theory.
\end{re}

The result for $p$-groups still formally extends to pro-$p$-groups. We use the notation for profinite groups from the previous sections.

\begin{cor}\label{cor:picborelsmithprop}
Let $G = \lim_{i \in I} G_i$ be a pro-$p$-group (so all $G_i=G/N_i$ are $p$-groups). Then
$$\pic(\Mod_{\nH \underl{k}}(\specg)) \cong \tn{CF}_{\tn{b}}(G),$$
where $\tn{CF}_{\tn{b}}(G)$ is the subgroup of $\tn{CF}_{\tn{o}}(G)$, the group of integral class functions on the conjugacy classes of open pro-subgroups of $G$, which consists of class functions $h$ for which there exists an open normal subgroup $N \unlhd G$ such that $h$ lies in $\tn{CF}_{\tn{b}}(G/N)$ when restricted to the conjugacy classes of subgroups of $G/N$.
\end{cor}
\begin{proof}
Recall the equivalence $\Mod_{\nH \underl{k}}(\specg) \simeq \tn{colim}_i \, \Mod_{\nH \underl{k}}(\spec^{G_i})$, where the colimit is taken in stable presentably symmetric monoidal $\infty$-categories $\calg(\prlst)$. Since the colimit diagram lies in compactly generated $\infty$-categories we can restrict to compact parts and obtain an equivalence $\Mod^{\omega}_{\nH \underl{k}}(\specg) \simeq \tn{colim}_i \, \Mod^{\omega}_{\nH \underl{k}}(\spec^{G_i})$, where the colimit is taken in $\calg(\tn{Cat}_{\infty}^{\tn{perf}})$. As the diagram is filtered it is computed in symmetric monoidal $\infty$-categories $\calg(\catinf)$, where it commutes with Picard groups. So \Cref{thm:picborelsmith} gives $\pic(\Mod_{\nH \underl{k}}(\specg)) \cong \tn{colim}_i \, \tn{CF}_b(G_i)$, and it remains to reformulate the right-hand colimit. \\
The maps $\tn{CF}_{\tn{b}}(G_i) \to \tn{CF}_{\tn{o}}(G), \; h \mapsto ((H) \mapsto h(HN_i/N_i))$ obtained by precomposition with $\alpha_i: G \twoheadrightarrow G_i$ are injective, since every subgroup of $G_i$ is of the form $HN_i/N_i$ for some open pro-subgroup $H \leq G$. One checks that the transition maps $\tn{CF}_{\tn{b}}(G_j) \to \tn{CF}_{\tn{b}}(G_i)$ induced by the $f_{ij}^*: \Mod_{\nH \underl{k}}(\spec^{G_j}) \to \Mod_{\nH \underl{k}}(\spec^{G_i})$ on Picard groups are the maps $\tn{CF}_{\tn{b}}(G_j) \to \tn{CF}_{\tn{b}}(G_i), \; h \mapsto ((H/N_i) \mapsto h(HN_j/N_j))$. They are injective since every subgroup of $G_j$ is of the form $HN_j/N_j$ for some subgroup $H \leq G_i$, and the maps $\tn{CF}_{\tn{b}}(G_i) \to \tn{CF}_{\tn{o}}(G)$ for $i \in I$ are compatible with them, i.e. whenever $f_{ij} \circ \alpha_i = \alpha_j$ the corresponding triangle of class function groups commutes. This shows that $\tn{colim}_i \, \tn{CF}_b(G_i)$ is given by the union $\bigcup_{i \in I} \tn{CF}_{\tn{b}}(G_i)$ inside $\tn{CF}_{\tn{o}}(G)$. For each open normal subgroup $N \unlhd G$ the surjective homomorphism $\alpha: G \twoheadrightarrow G/N$ induces an injection $\alpha^{-1}: (\sub(G/N))/(G/N) \hookrightarrow \sub_{\tn{o}}(G)/G, \; (H/N) \mapsto (H)$ by taking preimages, where the codomain is the set of conjugacy classes of open pro-subgroups of $G$. Since every open normal subgroup of $G$ is of the form $N_i$ for some $i \in I$ it is clear that the subgroups $\bigcup_{i \in I} \tn{CF}_{\tn{b}}(G_i)$ and $\tn{CF}_{\tn{b}}(G)$ of $\tn{CF}_{\tn{o}}(G)$ agree.
\end{proof}

\printbibliography
\end{document}